\documentclass[12pt,a4paper]{amsart} 
\calclayout
\numberwithin{equation}{section}

\allowdisplaybreaks
\theoremstyle{plain}
\newtheorem{theorem}{Theorem}[section]
\newtheorem{lemma}[theorem]{Lemma}
\newtheorem{proposition}[theorem]{Proposition}
\newtheorem{corollary}[theorem]{Corollary}
\theoremstyle{remark}
 
\newtheorem{remark}[theorem]{Remark}
\newtheorem*{remark*}{Remark}
\usepackage{amssymb,color} 
\usepackage{mathrsfs} 
\usepackage{mathtools}\mathtoolsset{showonlyrefs}
\usepackage{comment} 
\usepackage{hyperref} 
\usepackage{environ}

\newcommand{\md}{\mathrm{d}}

\newcommand{\e}{\epsilon}
\newcommand{\R}{\mathbb{R}}

\newcommand{\p}{\partial}

\renewcommand{\v}{\mathrm v}
\newcommand{\ex}{\mathbb{E}}
\newcommand{\Pe}{\mathrm P}
\newcommand{\wW}{\widetilde{W}}

\newcommand{\cC}{\mathcal C}
\newcommand{\Cunit}{\cC_{[0,1]}}
\newcommand{\CI}{\mathcal C_{I}}
\newcommand{\Ctem}{\mathcal C_{\mathrm{tem}}}

\newcounter{statement}
\newenvironment{statement}{\setcounter{statement}{\value{equation}} \list{(\theequation)}{\usecounter{equation}} \setcounter{equation}{\value{statement}}}{\endlist}

\newenvironment{note}{\color{gray}}{\ignorespacesafterend}
\RenewEnviron{note}{}

\RenewEnviron{old}{}

\begin{document}
	\title[Effect of noise on RDE with non-Lipschitz drift]{Effect of small noise on the speed of reaction-diffusion equations with non-Lipschitz drift} 
	\author[C. Barnes]{Clayton Barnes}
	\address[C. Barnes]{Faculty of Industrial Engineering and Management, Technion, Israel Institute of Technology, Haifa 3200003, Israel} 
	\email{cbarnes@campus.technion.ac.il} 
	\author[L. Mytnik]{Leonid Mytnik}
	\address[L. Mytnik]{Faculty of Industrial Engineering and Management, Technion, Israel Institute of Technology, Haifa 3200003, Israel} 
	\email{leonid@ie.technion.ac.il}  
	\author[Z. Sun]{Zhenyao Sun}
%	\address[Z. Sun]{Faculty of Industrial Engineering and Management, Technion, Israel Institute of Technology, Haifa 3200003, Israel}
	\address[Z. Sun]{School of Mathematics and Statistics \\ Beijing Institute of Technology \\ Beijing 100081 \\ China} %
	\email{zhenyao.sun@gmail.com}
\begin{abstract}
	We consider the $[0,1]$-valued solution $(u_{t,x}:t\geq 0, x\in \mathbb R)$ to the one dimensional stochastic reaction diffusion equation with Wright-Fisher noise 
\[
	\partial_t u
	= \partial_x^2 u + f(u) + \epsilon \sqrt{u(1-u)} \dot W.
\]
Here, $W$ is a space-time white noise, $\epsilon > 0$ is the noise strength, and $f$ is a continuous function on $[0,1]$ satisfying 
$\sup_{z\in [0,1]}|f(z)|/ \sqrt{z(1-z)} < \infty.$ 
	We assume the initial data satisfies $1 - u_{0,-x} = u_{0,x} = 0$ for $x$ large enough.
	Recently, it was proved in (Comm. Math. Phys. \textbf{384} (2021), no. 2) that the front of $u_t$ propagates with a finite deterministic speed $V_{f,\epsilon}$, and under slightly stronger conditions on $f$, the asymptotic behavior of $V_{f,\epsilon}$ was derived as the noise strength $\epsilon$  approaches $\infty$. 
	In this paper we complement the above result by obtaining the asymptotic behavior of 
	$V_{f,\epsilon}$ as the noise strength $\epsilon$ approaches $0$:
	for a given $p\in [1/2,1)$, if $f(z)$ is non-negative and is comparable to $z^p$ for sufficiently small $z$, then $V_{f,\epsilon}$ is comparable to $\epsilon^{-2\frac{1-p}{1+p}}$ for sufficiently small $\epsilon$. 
\end{abstract}
\maketitle
	
\section{Introduction} \label{sec:I}
\subsection{Background and motivation}
In 1937, Fisher \cite{Fisher} and Kolmogorov, Petrovsky, Piskunov \cite{KPP} independently studied the wave propagation properties arising from the FKPP equation 
on $\R_+\times \R$
\begin{equation}
	\label{eq:I.04}
	\p_th = \p_{x}^2 h + f(h).
\end{equation}
Fisher was interested in how quickly an advantageous gene (or virus) would propagate through a population living in a linear habitat, such as a shoreline (or train). 
The solution $h$ measures the proportion of the population carrying this advantageous gene as the biological system evolves.
Under a mild assumption on the Lipschitz function $f$ with $f(0)=f(1)= 0$, for any velocity $\mathrm v$ greater or equal to the minimal velocity 
\begin{equation}
\label{eq:minv_1}
\mathrm v_\mathrm{min} = \sqrt{2f'(0)},
\end{equation}
 there exists a traveling wave solution $h(t,x)= F_\mathrm v(x - \mathrm vt)$ with wave profile denoted by $F_\mathrm v$. 
With Heaviside initial data $h(0, x) = \mathbf{1}_{x \leq 0}$, the shifted solution $h(t, x + m(t))$  converges uniformly to the wave profile $F_{\mathrm v_\mathrm{min}}$ with the shift $m(t)$ having asymptotic speed $\mathrm v_{\mathrm{min}}$.
These results have been generalized to include a wider class of initial conditions, a more detailed description of the lower order terms of the wave position, and tail behavior of the wave shape; see \cite{BacaerN,Bramson1983Convergence,McKean1, McKean2}.

Because the FKPP equation is noiseless, it can be thought to represent a mean field approximation of a microscopic reaction-diffusion process
as motivated from a statistical physics perspective \cite{BrunetDerrida2001Effiect}.
Consequently, there has been recent interest in understanding analogous questions regarding propagating speed of waves for solutions to the FKPP equation with Wright-Fisher noise, given by the SPDE on $\R_+\times \R$:
\begin{equation}\label{eq:I.1} 
	\partial_t u 
	= \partial_x^2 u + f(u) + \epsilon\sqrt{u(1-u)}  \dot W
\end{equation}
where $f$ is a continuous function on $[0,1]$ with $f(0)= f(1) = 0$ satisfying some regularity conditions, and $W$ is a space-time white noise on  $\R_+\times \R$.  
 Similar to \eqref{eq:I.04}, the solution $u$ represents the proportion of the population exhibiting the gene, but now the equation incorporates the random interaction among  the populace. 
The  noise term $\epsilon\sqrt{u(1-u)}\dot{W}$ is motivated by the assumption that these interactions are affected by i.i.d.\ mean zero random variables independent of time and space, while the variance of the outcome is proportional to the rate of interaction between those with the gene and those without the gene, which is $u(1-u).$
The function $f$ continues to describe the deterministic  evolution of the population exhibiting the gene.

The example of $f(u)=u(1-u)$ was extensively studied in the literature and weak uniqueness, compact interface property, finite speed of the front propagation and other properties were established (see \cite{ConlonDoering2005travelling, MuellerSowers,  {Shiga1988Stepping}, Tribe1995Large}).
Moreover there has been a great interest in the asymptotic behavior of the speed of the front propagation,  $V_{f,\e}$. For a large class of Lipschitz functions $f$, including $f(u)=u(1-u)$, and also for more general noise coefficients, 
 Mueller, Mytnik, and Quastel, in~\cite{MuellerMytnikQuastel2011Effect}, proved the Brunet-Derrida conjecture (see \cite{BrunetDerrida1997Shift}) on the asymptotic of $V_{f,\e}$ for small $\e$. 
 
\begin{note}
For $g \in C(\R, [0, 1])$, let 
\[
R(g) := \sup\{ x : g(x) > 0\}, \, L(g) = \inf\{x : g(x) < 1\},
\]
and let $\mathcal{D}(\R, [0,1])$ be the subset with $g(x) \overset{x \to -\infty}{\longrightarrow} 1$ and $g(x) \overset{x \to \infty}{\longrightarrow} 0$.
Elements in $\mathcal{D}$ depict a wave decreasing from a height of one to zero; we say such a function $g$ has compact interface when $R(g), L(g)$ are finite.
\end{note}

Another motivation for the study of the stochastic reaction diffusion equation is its duality relation to the branching-coalescing Brownian motion.  
Consider a system of particles moving as independent one-dimensional Brownian motions on $\R$ with generator $\partial_x^2$. 
Assume that each particle independently branches with rate $1$ into a random number of particles according to an offspring law $(q_k)_{k\in \mathbb Z_+}$, and each pair of particles independently coalesce at rate $\epsilon^2$ according to their intersection local time. 
 Denote by $(x_i(t): t\geq 0, i\leq N_t)$ the positions of the particles  
where $N_t$ is the number of all particles at time $t\geq 0$.
In the case of binary branching, i.e. $q_2 = 1$, the following duality relation is due to~\cite{Shiga1988Stepping} (see also \cite{DoeringMuellerSmereka2003Interacting}):
Let $u$ be a solution to the SPDE \eqref{eq:I.1} with $f(u)=u(1-u)$.
	Assume that the random field $u$, as well as its driving noise $W$, is independent of the particle system.
Then
\begin{equation}
\label{eq:dual_sh_1}
\ex\Big[\prod_{i = 1}^{N_0}(1-u_{t,x_i(0)})\Big]=\ex\Big[\prod_{i = 1}^{N_t}(1-u_{0,x_i(t)})\Big], \quad t\geq 0.  
\end{equation}

This duality relation was first constructed to study the weak uniqueness of the stochastic FKPP equation. 
It can also be used to study the propagation of the extremal particle in the branching-coalescing Brownian motion. 
Assume that the particle system starts with a single particle at the position $0$. 
Let $R(t)$ be the position of  the rightmost particle in the system at time $t$.  
From the above duality, by taking $u_{0,\cdot}=\mathbf{1}_{(-\infty, 0)}$ and using symmetry, one can get 
$\Pe(R(t)>x) =\ex[u_{t,x}]$ for every $(t,x)\in \R_+\times \R.$ 
Then from the existence of the speed of front propagation $V_{f,\e}$ for $u$, one can derive the upper bound on the speed of $R(t)$: 
for any $\delta>0$,
\begin{equation}
\mathrm P\big(R(t) \leq (V_{f,\epsilon} + \delta)t\big)
\xrightarrow[t\to \infty]{} 1.
\end{equation}

One expects that the duality \eqref{eq:dual_sh_1} holds also for more general drift function $f$ when it takes the form
\begin{equation}\label{eq:I.2}
f(z)= 1 - z - g(1-z), \quad z\in [0,1],
\end{equation}
where $g(z)=\sum_{n=0}^\infty z^n q_n, z\in [0,1],$ is the probability generating function of the offspring law in the branching-coalescing system. 
This is established for some offspring laws with finite first moment, see \cite[Theorem 1]{AthreyaTribe2000Uniqueness}. 
The case of the offspring law being heavy-tailed, without existence of the first moment, is of particular interest. 
For example, if one considers the following heavy-tailed offspring law
\begin{equation}
	q_0 =q_1 = 0; \quad q_n = \frac{-1}{(n-1)!}\prod_{k=0}^{n-2} (k-p), \quad n \in \mathbb Z \cap [2,\infty)
\end{equation}
where $p \in (0,1)$ is a given constant, then, the drift function $f$ defined via  \eqref{eq:I.2} takes the form 
\begin{equation} \label{eq:I.3} f(z) = z^p(1-z), \quad z\in [0,1],\end{equation} 
which is not Lipschitz at $0$. 
This $(q_n)_{n=1}^\infty$ is also known (see \cite{ChristophSchreiber2000Scaled} for example) as the law of $1 + S_p$ where $S_p$ is a Sibuya random variable with parameter $p$. 
As far as we know, 
the Shiga duality 
relation \eqref{eq:dual_sh_1} for such cases has  not been proved yet. 
However, we do conjecture that this duality holds for some  H\"older drift functions $f$ having representation~\eqref{eq:I.2}, and this gives us another motivation to study the SPDE \eqref{eq:I.1} with 
non-Lipschitz drift functions.

In fact, one of such cases has been studied  recently in~\cite{MuellerMytnikRyzhik2019TheSpeed},  where weak solutions $u$ to the SPDE \eqref{eq:I.1} are investigated  under the conditions that  
\begin{equation}
\label{asp:0}
\text{$f$ is continuous, and} \sup_{z\in [0,1]}\frac{|f(z)|}{\sqrt{z(1-z)}}
<\infty 
\end{equation}
and that the initial value has compact interface, that is, $u_{0,-x} -1 = u_{0,x} = 0$ for large enough $x$.
This includes examples like \eqref{eq:I.3} with parameter $p \in [1/2, 1)$.
Note that in the deterministic case of $\epsilon = 0$, the solution to the reaction diffusion equation~\eqref{eq:I.04}, with $f$ given by~\eqref{eq:I.3} and with a non-trivial initial value, does exist, but it does not exhibit propagating waves with a finite linear speed as in the case when $f$ is Lipschitz.  
Intuitively, this is clear from~\eqref{eq:minv_1} as $f'(0)$ is infinite. 
In fact, such solutions $h$ do not have super-linear speed either, 
as it follows from a similar argument in \cite{AguirreEscobedo} that 
$\inf_{x \in \R}h(t, x) \to 1$ as $t \to \infty$.
However, this behavior changes drastically when introducing the Wright-Fisher noise. 
 In \cite{MuellerMytnikRyzhik2019TheSpeed}, the authors 
 established the weak uniqueness, compact interface property, and finiteness of front propagation speed for \eqref{eq:I.1} when $\epsilon > 0$.
 In particular, they proved that there exists a deterministic $V_{f,\epsilon}\in \mathbb R$, which only depends on the drift function $f$ and the noise strength $\epsilon$, so that 
\begin{equation} \label{eq:speed}
	\frac{\sup\{x\in \mathbb R:u_{t,x}\neq 0\}}{t} \xrightarrow[t\to \infty]{} V_{f,\epsilon}, \quad \text{a.s.}
\end{equation}

 It is then natural to study  the asymptotics of  the speed of this propagation in terms of the strength of the noise. 
 In \cite{MuellerMytnikRyzhik2019TheSpeed}, 
 the authors studied
 the asymptotic behavior of $V_{f,\epsilon}$, when $\epsilon$ goes to $\infty$, under a condition slightly stronger than \eqref{asp:0}.
As for the small $\e$,  the case of 
Lipschitz $f$ was treated
already in ~\cite{MuellerMytnikQuastel2011Effect}, and it was shown there   
how fast $V_{f,\e}$ converges to $V_{f,0}$ as $\e\downarrow 0$.
 In this paper, we complement the above results and consider the asymptotic behavior of $V_{f,\epsilon}$ when $\epsilon$ converges to $0$ and $f$ is not necessarily Lipschitz. 
According to our discussion about the deterministic case of $\e=0$, it is intuitively clear that if, for example, $f$ is given by~\eqref{eq:I.3}, then $V_{f,\e}$ should converge to $\infty$ as 
$\e\downarrow 0$. Our main result shows this, but also answers the much more delicate question: {\it At what rate does $V_{f,\e}$  converge to $\infty$ as $\e\downarrow 0$?} 

\subsection{Main result}
 To state our main result we  need to introduce the following conditions on $f$. 
 \begin{statement}
 \item \label{asp:1}
	 $f$ is non-negative and there exists  $p_0\in [1/2,1)$  such that  $\liminf_{z\downarrow 0} f(z)/z^{p_0} > 0$.  
 \item \label{asp:2}  There exists  $p\in [1/2,1)$ such that
	$\limsup_{z\downarrow 0} f(z)/z^p < \infty$.
\end{statement}
Note that \eqref{eq:I.3} is an example of $f$ satisfying \eqref{asp:0}-\eqref{asp:2} with $p_0 = p$.    Let us now state our main result.  
\begin{theorem} \label{thm:LU} 
		Suppose that $f$ is a function on $[0,1]$ satisfying \eqref{asp:0}.
		For every $\epsilon > 0$, denote by $V_{f,\epsilon}$ the propagation speed of the SPDE \eqref{eq:I.1} given as in \eqref{eq:speed}.  
	\begin{itemize}
	\item[(a)]
		If $f$ satisfies \eqref{asp:1}, then  $ \liminf_{\epsilon \downarrow 0}\epsilon^{2\frac{1-p_0}{1 + p_0}} V_{f,\epsilon} > 0.$ 
	\item[(b)]
If $f$ satisfies \eqref{asp:2}, then
$
	\limsup_{\epsilon \downarrow 0}\epsilon^{2\frac{1-p}{1 + p}} V_{f,\epsilon}
	< \infty.
$
\end{itemize}
\end{theorem}
 If the drift function $f$ satisfies \eqref{asp:0}-\eqref{asp:2}  with $p_0 = p$, then Theorem \ref{thm:LU} implies that there exists $\epsilon_0>0$ and $c,C>0$ such that 
\[
	c\e^{-2\frac{1-p}{1+p}} \leq V_{f,\e} \leq C \e^{-2\frac{1-p}{1+p}}, \quad \e\in (0,\e_0).
\]
Note that the exponent $-2\frac{1-p}{1+p}$ shows up in both the upper bound and the lower bound, and therefore cannot be improved.
	This exponent appears when we analyze the free-boundary travelling wave problem \eqref{eq:I.4} below. 
In the next subsection, we give some comments on the proof strategy for our main result.

\subsection{Proof strategy} \label{sec:PS}
For the lower bound, we simply replace the drift $f$ by some smaller Lipschitz drift $H\leq f$.  
The comparison principle then gives us a lower bound $V_{H,\epsilon} \leq V_{f,\epsilon}$. 
Of course, by choosing different Lipschitz functions $H$, one can obtain a family of lower bounds.
To obtain the optimal one, we take $H$ depending on the noise strength $\epsilon$ in a certain way so that $H'(0)$ is comparable to $\epsilon^{-4\frac{1-p}{1+p}}$. 
	(Recall that $\sqrt{2H'(0)}$ is the minimal traveling wave velocity for the FKPP equation \eqref{eq:I.04} with drift $f$ being replaced by $H$.) 
We then use a known result on the propagation speed of stochastic FKPP equation \cite{MuellerMytnikQuastel2011Effect} to get the desired lower bound.

For the upper bound, the strategy of replacing the drift by Lipschitz functions is not fruitful because for any Lipschitz function $H$ greater than a drift function $f$ satisfying  \eqref{asp:0}-\eqref{asp:2}, it always holds that $H(0)>0$.
For a solution $u$ corresponding to such a drift $H$, the state $0$ is not locally stable anymore, and typically, $\sup\{x:u_{t,x}\neq 0\}$ is not even finite.
Instead, we use a similar strategy as in \cite{MuellerMytnikQuastel2011Effect} to decompose our solution $u$ as 
\[ u=v+w,\]
where $v$ is a weak solution to the SPDE \eqref{eq:I.1} with a moving Dirichlet boundary condition on the line $\{(t,x): x = \mathrm vt\}$, i.e.  
\[
\begin{cases}
	\partial_t v 
	=\partial^2_x v + f(v) + \epsilon \sqrt{v(1-v)} \dot W^v,
	& x < \mathrm v t,
	\\v
	= 0,
	& x\geq \mathrm v t.
\end{cases}
\]
For the details on the above decomposition of $u$ see Proposition~\ref{thm:C}. Let us just note that since $u,v,w$ are defined on the same probability space, then the white noises 
  $W^v$ and $W$ are also defined on the same probability space. Also, note that the velocity of the moving boundary $\mathrm v$ is left to be chosen.

It is intuitively clear that if one chooses $\mathrm v$ to be larger than $V_{f,\epsilon}$ then the deviation between $u$ and $v$, which is $w$, should be small and not propagate; and if one chooses $\mathrm v$ to be smaller than $V_{f,\epsilon}$, then $w$ will be large and will propagate. 
Therefore, by searching for a balanced value $\mathrm v$ so that $w$ lies in between those two phases, one can obtain a good estimation on $V_{f,\epsilon}$.

An insight from \cite{MuellerMytnikQuastel2011Effect} suggests that such a value $\mathrm v$ can be predicted by finding the solution $(F, \mathrm v)$ to a free-boundary travelling wave problem  
\begin{equation}\label{eq:I.4}
	\begin{cases}
	\varrho_{t,x} = F(x-\mathrm vt) \geq 0,
	\\ \partial_t \varrho = \partial_x^2 \varrho + f(\varrho), &\quad x< \mathrm vt,
	\\ \varrho = 0, &\quad x\geq \mathrm vt,
		\\ \lim_{x\uparrow \mathrm vt}\partial_x \varrho_{t,x} = -\epsilon^2.
	\end{cases}
\end{equation}
Replacing the drift $f$ in \eqref{eq:I.4} by some approximating Lipschitz functions, the solution $(F,\mathrm v)$ is computable using a similar argument used in \cite[Proof of Proposition 2.1]{MuellerMytnikQuastel2011Effect}, and one can calculate that the balancing value $\mathrm v$ should be comparable to $\epsilon^{-2\frac{1-p}{1+p}}$  
	for small $\epsilon$, which gives us another intuitive explanation for the exponent $-2\frac{1-p}{1+p}$.

To analyze the behavior of $w$ under this balancing value $\mathrm v\sim \epsilon^{-2\frac{1-p}{1+p}}$, we observe that it satisfies the following equation
\[
	\partial_t w = \partial_x^2 w + f(u) - f(v) + \epsilon \sqrt{u(1-u)} \dot W - \epsilon \sqrt{v(1-v)} \dot W^v + \dot A_t \delta_{\mathrm vt}(x) 
\]
where $A_t$ is the accumulated mass of $v$ being ``killed'' at its boundary before time $t\geq 0$; note that, as  it is shown in Proposition~\ref{thm:C},  up to a certain 
stopping time, $w$ can be constructed in a way that it  satisfies an SPDE similar to the above but driven by a single noise $W^w$.  
Note that $f$ is typically not Lipschitz, so unlike in \cite{MuellerMytnikQuastel2011Effect} we cannot control the drift term $f(u) - f(v)$ by $\|f\|_{\text{Lip}}w$.
To overcome this, we use Dawson's Girsanov transformation and remove this drift term under a new probability measure.
However, similarly to what often happens for finite dimensional diffusion processes, one cannot control the Radon-Nikodym derivative in Dawson's Girsanov transformation for a long time. So we need to chop off the time into small intervals $\{[nT,(n+1)T):n\in \mathbb Z_+\}$, and only perform Dawson's Girsanov transformation on each of those intervals.
By choosing the parameter $T$ small enough, the transformed $w$ will then serve as a good approximation of the original $w$ on each of those intervals. 
On the other hand, in order to get a reasonably good upper bound for the long time propagation speed, we cannot take $T$ too small either. 
So a balanced value has 
 to be chosen for this interval length $T$.

 Our philosophy of choosing such a value for $T$ works as follows.  
Consider the random field $w$ by standing at the moving boundary $\{(t,x):x=\mathrm vt\}$, that is to say, consider the random field $(w_{t,\mathrm vt+x}:(t,x)\in \mathbb R_+\times \mathbb R)$.
Say $L$ is a typical distance for the support of this random field to travel in a time interval of length $T$.
We want our $T$ to be chosen so that this $L$ not only can be explained by the (parabolic) thermal diffusivity, but also does not give excess speed. That is, we want both $L \sim \sqrt{T}$ and $L/T \lesssim \mathrm v$.   
Recalling our choice of $\mathrm v \sim \epsilon^{-2\frac{1-p}{1+p}}$, 
we end up choosing $T\sim \epsilon^{4\frac{1-p}{1+p}}$. It turns out that this is also a time span on which  Dawson's Girsanov transformation argument works. 
We hope this idea, of performing Dawson's Girsanov transformations on time intervals with a balanced length, can also be useful for finding propagation speed in other spacial stochastic models.  

Note that we only considered the Wright-Fisher noise $\sqrt{u(1-u)}\dot W$. 
It would be interesting to also consider more general noise $\sigma(u) \dot W$. 
We comment here that both in the proofs of our result Theorem \ref{thm:LU} and of \cite[Theorem 1.1]{MuellerMytnikRyzhik2019TheSpeed}, the Wright-Fisher noise is not essential: what has been really used is the property that $\sqrt{z(1-z)} \sim \sqrt{z}$ for small $z$.
However, to explore the most general conditions for the noise term $\sigma$  is out of the scope of the current paper.

\subsection{Paper outline}
The rest of the paper is organized as follows. 
In Section \ref{sec:P}, we recall some preliminary terminology including the solution concept for the SPDE \eqref{eq:I.1}.
In Section \ref{sec:L}, we give the proof of Theorem \ref{thm:LU}(a).
We give the proof of Theorem \ref{thm:LU}(b) in Section~\ref{sec:U} while the proofs of the results used for its proof  
are given in Sections \ref{sec:H}-\ref{sec:B}.

\section{Preliminary} \label{sec:P} 
In this section, we recall some preliminary terminology including the solution concept to the SPDE \eqref{eq:I.1}.
We first give some notation.
We say a filtered probability space $(\Omega, \mathcal G, (\mathcal{F}_t)_{t \geq 0}, \mathrm P)$ satisfies the usual hypotheses if $(\Omega, \mathcal G, \mathrm P)$ is a complete probability space with right-continuous filtration $(\mathcal {F}_t)_{t\geq 0}$ satisfying $\{A\in \mathcal G:\mathrm P(A) = 0\} \subset \mathcal F_0$.
We impose the usual hypotheses on every filtered probability spaces that will be considered in this paper.
Given such a space, denote by $\mathscr M_\mathrm{loc}$ the family of adapted continuous local martingales.
For any continuous semi-martingale $M$, denote by $\langle M \rangle$ its quadratic variation. 
Given two continuous semi-martingales $M, N$, let $\langle M, N\rangle$  denote their quadratic covariation.
In this paper, we say $g$ is a random field if it is an $\mathbb R$-valued stochastic process indexed by $\mathbb R_+\times \mathbb R$.
Denote by $\mathscr L^2_\mathrm{loc}$ the family of predictable random fields $g$ satisfying
\[ \iint_0^t g^2_{s,y} \mathrm ds\mathrm dy < \infty, \quad t\geq 0, \quad \text{a.s.} \]
Let $\mathcal B_F(\mathbb R)$ be the collection of Borel subsets of $\mathbb R$ with finite Lebesgue measure.
We say $W=(W_s(A):A \in \mathcal B_F(\mathbb R),s\in \mathbb R_+)$ is a white noise if it is an adapted orthogonal martingale measure so that for any $A,B\in \mathcal B_F(\mathbb R)$ almost surely 
\[
	\langle W_{\cdot}(A),W_{\cdot}(B)\rangle_{t} 
= t \cdot \operatorname{Leb}(A\cap B), 
	\quad t\geq 0,
\] 
where $\operatorname{Leb}(\cdot)$ is the Lebesgue measure on $\mathbb R$.
Given a white noise $W$, Walsh's stochastic integral for $W$ is a map from $\mathscr L_{\mathrm{loc}}^2$ to $\mathscr M_{\mathrm{loc}}$ which will be denoted by  
\[
	g
	\mapsto \iint_0^\cdot g_{s,y}W(\mathrm ds\mathrm dy).
\]
We refer our reader to \cite{Iwata1987AnInfinite, Walsh1986An-introduction} for more details.

\begin{note}
In this paper, we will consider the propagation speed for solutions of the following SPDE  
\begin{equation} \label{eq:I.1}
	\partial_t u 
	= \partial_x^2 u + f(u) + \epsilon \sigma(u) \dot W
\end{equation}
where $\epsilon > 0$,
$W$ is a white noise,
$\sigma(z) := \sqrt{z(1-z)}\mathbf 1_{[0,1]}(z)$ for each $z\in \mathbb R$; 
and $f$ is a real-valued function on $\mathbb R$.
We always assume that $f$ satisfies the following conditions:	
\begin{statement}
\item \label{asp:1} 
$f(0) = f(1) = 0$;
\item \label{asp:2}
$\liminf_{z \downarrow 0} f(z) / z^p > 0;$  
\item \label{asp:3}
	$f$ is H\"older continuous on $[0,1]$ with exponent $p\in [1/2,1)$, i.e.
\[
	\sup_{x,y\in [0,1]}  \frac{| f(x) - f(y) |}{|x-y|^p}< \infty.
\]
\end{statement}
\end{note}

Let us now be precise about the solution concept for the SPDE \eqref{eq:I.1}. 
Denote by $\mathcal C_{\mathrm{tem}}$ the space of continuous functions $g$ on $\mathbb R$ such that 
\[\|g\|_{(-\lambda)}:= \sup_{x\in \mathbb R} |e^{-\lambda |x|} g(x)| < \infty, \quad \forall \lambda > 0.\]
Let $\mathcal C_{\mathrm{tem}}$ be equipped with the topology generated by the norms   $(\|\cdot\|_{(-\lambda)}: \lambda > 0)$,  and set $\mathcal C_{\mathrm{tem}}^+$ as the collection of non-negative elements in $\mathcal C_{\mathrm{tem}}$. 
Let $\cC(\R_+, \Ctem)$ be the space of continuous $\Ctem$-valued paths with the topology of uniform convergence on bounded time sets. 
We say a Borel function $f$ on $\mathbb R$ satisfies the linear growth condition if 
$
	\sup_{z\in \mathbb R}|f(z)|/(1+|z|) < \infty.
$
Assume that Borel functions $f$ and $\sigma$ on $\mathbb R$ satisfy the linear growth condition.
We say that $(\Omega, \mathcal G, (\mathcal F_t)_{t\geq 0}, \mathrm P, u, W)$ is a weak solution to the SPDE 
\begin{equation}
	\label{eq:P.1}
	\partial_t u 
	= \partial_x^2 u + f(u) + \sigma(u)  \dot W,
\end{equation}
if $(\Omega, \mathcal G, (\mathcal F_t)_{t\geq 0}, \mathrm P)$ is a filtered probability space on which a predictable random field $u$ and a white noise $W$ are defined so that $(u_{t,\cdot}:t\geq 0)\in\cC(\R_+, \mathcal C_{\text{tem}})$ and 
\begin{equation}\label{eq:P.2} 
	u_{t,x} = \iint_0^t G_{s,y;t,x} M^u(\mathrm ds\mathrm dy) 
	\quad \text{a.s.}
	\quad (t,x)\in (0, \infty)\times \mathbb R,
\end{equation}
where 
\begin{equation}\label{eq:P.3}
	G_{s,y;t,x}
	:= \frac{e^{-\frac{(x-y)^2}{4(t-s)}}}{\sqrt{4\pi (t-s)}}\mathbf 1_{s<t}, \quad (s,y), (t,x)\in \mathbb R_+\times \mathbb R
\end{equation}
and 
\[
M^u(\mathrm ds\mathrm dy) 
:= u_{0,y} \delta_0(\mathrm ds)\mathrm dy 
+ f(u_{s,y})\mathrm ds\mathrm dy 
+ \sigma(u_{s,y}) W(\mathrm ds\mathrm dy).
\]
Here, the right hand side of \eqref{eq:P.2} is a mixture of the classical integral and Walsh's stochastic integral defined in an obvious way using linearity.
With some abuse of notation, 
we sometimes just use the random field $u$  
to represent the weak solution if there is no risk of confusion.  
We refer our reader to \cite[Theorem 2.1]{Shiga1994Two} for an equivalent definition.

Given a subset $\hat{\mathcal C} \subset \Ctem$, we say a weak solution $u$ to the SPDE \eqref{eq:P.1} is a $\hat{\mathcal C}$-valued weak solution if $(u_{t,\cdot}:t\geq 0)$ is a $\hat{\mathcal C}$-valued process. 
We say the weak existence of the SPDE \eqref{eq:P.1} holds in 
 $\hat\cC$ 
for an initial condition $g\in \hat{\mathcal C}$ if there exists 
a $\hat \cC$-valued weak solution $u$ to \eqref{eq:P.1} 
such that $u_{0,\cdot} = g$. 
We say the weak uniqueness of the SPDE \eqref{eq:P.1} holds in $\hat\cC$ for an initial condition $g\in \hat{\mathcal C}$ if, whenever $u$ and $u'$ are two $\hat{\mathcal C}$-valued weak solutions to the SPDE \eqref{eq:P.1} such that $u_{0,\cdot} = g$ and $u'_{0,\cdot}= g$, the $\Ctem$-valued processes $(u_{t,\cdot}:t\geq 0)$ and $(u'_{t,\cdot}:t\geq 0)$ induce the same 
law on $\cC(\R_+, \Ctem)$.

The weak existence, weak uniqueness, and compact propagation property of the SPDE \eqref{eq:I.1} under condition \eqref{asp:0} and $\epsilon > 0$ is studied in \cite{MuellerMytnikRyzhik2019TheSpeed}. 
Denote by 
 $\Cunit$  
the space of continuous functions on $\R$ taking values in $[0,1]$.  
Denote by $\CI$ the family of functions $g$ in $\Cunit$ with compact interface, i.e. $g\in \Cunit$ and $-\infty < L(g) < R(g) < \infty$ where  
$
	L(g)
	:= \inf \{x\in \mathbb R: g(x) \neq 1\}$ 
and 
	$R(g)
	:= \sup\{x\in \mathbb R: g(x)\neq 0\}.
$

	\begin{theorem}[\cite{MuellerMytnikRyzhik2019TheSpeed}]\label{intro:thm_existenceUniqueness}
	For any $\epsilon > 0$ and function $f$ satisfying \eqref{asp:0}, the following holds.
\begin{enumerate} 
\item 
	For any initial condition $g\in \CI$, the weak existence and weak uniqueness of \eqref{eq:I.1} 
	holds in $\Cunit$. 
\item 
	For any $\Cunit$-valued weak solution $u$ to \eqref{eq:I.1} with $u_{0,\cdot} \in \CI$, it holds that \[\mathbb E\Big[\sup_{s\in [0,t]} |R(u_{s,\cdot}) - L(u_{s,\cdot})|\Big] < \infty, \quad t\geq 0.\]
	In particular, for any initial condition $g\in \CI$, the weak existence and weak uniqueness of \eqref{eq:I.1} 
	holds in $\CI$. 
\item 	
	There exists a deterministic $V_{f,\epsilon}\in \mathbb R$ such that for any $\CI$-valued weak solution $u$ to \eqref{eq:I.1}, 
\[
	\lim_{t\to \infty} \frac{R(u_t)}{t} = V_{f,\epsilon}, \quad \text{a.s.}
\]
\end{enumerate}
\end{theorem}
\begin{remark} \label{thm:P.2}
We refer to $V_{f, \epsilon}$ as the speed of the traveling front of the SPDE \eqref{eq:I.1}.
We emphasize here that $V_{f,\epsilon}$ depends only on the drift function $f$ and the noise strength $\epsilon$.
To see that  it is independent of the initial value $g$, we note that for any other $\tilde g\in \CI$, there exists a constant $c\in \mathbb R$ such that 
\[
	g(x + c) \leq \tilde g(x) \leq g(x-c), \quad x\in \mathbb R.
\]
Therefore, using the comparison principle and the weak uniqueness, for any weak solutions $u$ and $\tilde u$ to the SPDE \eqref{eq:P.1} with $u_{0,\cdot} = g$ and $\tilde u_{0,\cdot} = \tilde g$ respectively, random field $\tilde u$ will be dominated stochastically by $(u_{t,x-c}:(t,x)\in \mathbb R_+\times\mathbb R)$ from above, and by $(u_{t,x+c}:(t,x)\in \mathbb R_+\times \mathbb R)$ from below. 
This indicates that the fronts of $u$ and $\tilde u$ have the same speed.
\end{remark}
Note that in Remark \ref{thm:P.2} we used the comparison principle in the absence of the Lipschitz condition.
This is justified by the following lemma whose variants have already appeared in the literature, see \cite[Theorem 2.6]{Shiga1994Two} and \cite[p. 412]{MuellerMytnikQuastel2011Effect} for example.
\begin{lemma}[Comparison Principle]\label{intro:comparison}
	Let $f,\tilde f$ and $\sigma$ be continuous functions on $\mathbb R$ satisfying the linear growth condition.  
	Assume that $\tilde f\leq f$ on $\mathbb R$.
	Let $g, \tilde g \in \Ctem$ satisfy $\tilde g \leq g$ on $\mathbb R$. 
	Then, there exists a weak solution $u$ to the SPDE \eqref{eq:P.1} with $u_{0,\cdot} = g$, and a weak solution $\tilde u$ to the SPDE  
	\begin{equation}\label{eq:P.15} 
	\partial_t \tilde u = \partial_x^2 \tilde u^2 + \tilde f(\tilde u) + \sigma(\tilde u) \tilde W 
\end{equation} 
	with $\tilde u_{0,\cdot} = \tilde g$,  so that the random field $\tilde u$ 
	is stochastically dominated by the random field $u$, i.e. $\tilde u$ and $u$ can be coupled in one probability space so that $\tilde u_{t,x}\leq u_{t,x}$ for every $t\geq 0$ and $x\in \mathbb R$ almost surely.
\end{lemma}
\begin{remark*}
	Under the condition of the above lemma, if we further assume that the weak uniqueness holds for both the SPDEs \eqref{eq:P.1} and \eqref{eq:P.15}, then  the weak solution $\tilde u$ to SPDE \eqref{eq:P.15} with initial value $\tilde g$ is stochastically dominated by  the weak solution $u$ to the SPDE \eqref{eq:P.1} with initial value $g$.
\end{remark*}
One can prove the above lemma by following the routine arguments in the proof of \cite[Theorem 2.6]{Shiga1994Two}.
Note that when $\tilde f(0) = \sigma(0) = 0$ and $\tilde g\equiv 0$, it actually follows directly from \cite[Theorem 2.6]{Shiga1994Two}.
\begin{note}
	Zhenyao: I personally haven't go though the proof of the above lemma. I hope we are not making any mistakes here.
\end{note}

Another general result that will be used very often is the following rescaling lemma of the SPDE \eqref{eq:P.1} which can be proved using a similar argument as in \cite[Section 4.1]{MuellerMytnikRyzhik2019TheSpeed}.
\begin{lemma}[Rescaling]\label{rescaling}
	Suppose that Borel functions $f$ and $\sigma$ on $\mathbb R$ satisfies the linear growth condition.
	Suppose that $u$ is a weak solution to the SPDE \eqref{eq:P.1}.  
	Let $\alpha, \beta > 0$, and $v_{t,x} := \beta u_{\alpha^{-4}t,\alpha^{-2}x}$ for each $(t,x)\in \mathbb R_+\times \mathbb R$.
	Then there exists a white noise $W^v$ such that $v$ 
 is a weak solution to the SPDE
\[
	\p_t v = \p_x^2v + \alpha^{-4}\beta f(\beta^{-1} v) + \alpha^{-1} \beta \sigma(\beta^{-1} v)\dot{W}^v.
\]
\end{lemma}
 
	We end this section by collecting some notations for function spaces.  
	Given a locally compact separable metric space $E$, we denote by $\cC(E)$ the space of continuous functions on $E$.  
	We write $\cC_\mathrm b(E)$, $\cC_0(E)$, and $\cC_\mathrm c(E)$, respectively, for the space of continuous functions on $E$ that are bounded, vanishing at $\infty$, and having compact support, respectively.
	Use $\cC_{\bullet}(E)$ to represent one of $\cC(E)$, $\cC_\mathrm b(E)$, $\cC_0(E)$ or $\cC_\mathrm c(E)$. 
	If $E\subset \mathbb R^d$, we define $\cC_{\bullet}^0(E) := \cC_\bullet(E)$, and inductively 
	\[\cC_{\bullet}^n(E) := \{ \phi \in \cC_\bullet(E): \partial_{x_k} \phi\in \cC_\bullet^{n-1}(E), \forall k =1,\dots,d   \}, \quad n\in \mathbb N,\]  
	and
	$\cC_{\bullet}^\infty(E) := \cap_{n=1}^\infty \cC_\bullet^n(E)$.
	If $E= \mathbb T\times \mathbb R$ with the time interval $\mathbb T\subset \mathbb R_+$, we define 
	\[\cC_{\bullet}^{1,2}(E) := \{\phi\in \cC_\bullet(E): \partial_t \phi, \partial_x^2\phi \in \cC_\bullet(E)\}.\]  

\begin{note}
\section{Lower bound} \label{sec:L}
In this section, we give the proof of Theorem~\ref{thm:LU}(a).
Our general strategy will be to use a rescaling lemma of the SPDE \eqref{eq:I.1} to move the $\epsilon$ away from the noise term.  
A comparison principle will then reduce the proof of the lower bound to showing that a particular SPDE has a positive speed. The rescaling result is Lemma \ref{rescaling} below. For any function $f$ satisfying \eqref{asp:1}, there exists constants $a>0$ and $0 < \delta < 1$ such that
\[
az^p\mathbf{1}(0 \leq z \leq \delta) \leq f(z) 
\]
for all $z \in [0, 1].$
From the comparison inequality (Lemma \ref{intro:comparison}) $V_{f_1, \e} \geq V_{f_2, \e}$ if $f_1 \geq f_2$, assuming the two corresponding solutions have the same initial condition. Consequently, it suffices to demonstrate Theorem~\ref{thm:LU}(a) in the case that
\[
f(z) = az^p\mathbf{1}(0 \leq z \leq \delta).
\]
By the rescaling Lemma \ref{rescaling} below, for any fixed $d > 0$ we see
\(
d^{1/2}V_{d^{-1}f, d^{-1/4}\epsilon} = V_{f, \epsilon},
\)
so that
$\liminf_{\epsilon \downarrow 0}\e^{2\frac{1 - p}{1 + p}}V_{f, \e}$ is positive if and only if $\liminf_{\epsilon \downarrow 0}\e^{2\frac{1 - p}{1 + p}}V_{df, \e}$ is positive. Therefore we can assume without loss of generality that $a = 1$ and for the remainder of this section take $f(z) = z^p\mathbf{1}(0 \leq z \leq \delta).$
\begin{lemma}[Rescaling]\label{rescaling}
	Let $\e \in (0, 1)$ and $u_{0,z} = \mathbf{1}_{z < 0}.$ 
	Let $u_{t, x}$ be a mild solution to \eqref{eq:I.1} with initial condition $u_{0,\cdot}$. 
	For $a, b > 0$ there exists a white noise $W^v$ such that $v_{t, x} := b\cdot u_{a^{-4}t, a^{-2} x}$ solves a mild form of the SPDE
\[
\p_t v_{t, x} = \p_x^2v_{t,x} + f_v(v_{t, x}) + b^{1/2}a^{-1}\e\sigma_v(v_{t, x})\dot{W}^v_{t, x}
\]
with initial condition $bu_{0,z}$, where $f_v(z) = a^{-4}bf\big(\frac{1}{b}\cdot z\big), \sigma_v(z) = \sqrt{z(1 - b^{-1}z)}\mathbf{1}_{z \in [0, b]}.$
\end{lemma}
\begin{proof}
The proof of this rescaling Lemma is done in \cite[Section 4.1]{MuellerMytnikRyzhik2019TheSpeed}. Essentially, one writes the three integral terms of the mild form expression and then applies the  rescaling properties of the Gaussian kernel and white noise to rewrite each term.
\end{proof}
\begin{lemma}\label{thm:L.1}
	Let $c  = c_{\eqref{eq:lb_scale}}$ be given according to Proposition \ref{thm:L.3},$\epsilon\in (0,\delta^{1/q}c^{1/q})$ and $u_{0,z} = \mathbf 1_{z< 0}$.
	Then there exists a usual probability space $(\Omega, \mathcal G,(\mathcal F_t)_{t\geq 0}, \mathbb P)$
	and $(u,\underline u, W^u, W^{\underline u} )$ defined on it such that
	\begin{enumerate} 
		\item 
	$W^u$ and $W^{\underline u}$ are white noises;
		\item 
	$u$ is a solution to \eqref{eq:I.1} with initial data $u_{0, \cdot}$ and white noise $W$; 
		\item 
	$\underline u$ is a solution to
\begin{align} 
&  \partial_t \underline u_{t,x} = \partial_x^2 \underline u_{t,x} + \underline f_{c, \e}(\underline u_{t,x}) + \epsilon \sqrt{\underline u_{t,x}(1-\underline u_{t,x})} \dot W^{\underline u}_{t,x}, \quad t\geq 0, x\in \mathbb R
\end{align}
	with initial data $u_{0, \cdot}$, where 
\begin{align} 
& \underline f_{c, \e}(z): = \begin{cases}
0, & \quad z\in (-\infty, 0],
\\c^{1-p}\epsilon^{-(1-p)q} z, & \quad z \in (0,\frac{\epsilon^q}{2c}],
\\- c^{1-p}\epsilon^{-(1-p)q}(z - \frac{\epsilon^q}{c}), &\quad z\in (\frac{\epsilon^q}{2c}, \frac{\epsilon^q}{c}],
\\ 0 , & \quad z\in (\frac{\epsilon^q}{c}, \infty),
\end{cases} 
\end{align}
 and $q : = \frac{4}{1+p}$; and
\item
 $\underline u_{t,x} \leq u_{t,x}$ for all $t\geq 0$ and $x\in \mathbb R$.
\end{enumerate}
 \end{lemma}

\begin{proof}
	This is the weak comparison principle for SPDEs, see Shiga \cite[Corollary 2.4]{Shiga1994Two}
	\footnote{LM: Shiga's result is for Lipschitz coefficients. Here the noise coefficient is Holder. 
	In fact we explain a page above in {\it Comparison  Principle} that one can construct two weak solutions via limit of approximations by solutions to equations with Lipschitz coefficients where  comparison holds. So part of the lemma is also existence of a probability space on which these limiting weak  solutions will be defined. We should somehow combine the comparison principle above with this lemma.  Do I miss something?}. One only has to verify the analytic fact that  $\underline f_{c, \e}\leq f$. Notice that $f(z)$ is concave, and the $\underline{f}_{c, \e}$ is a linear tent function. Hence it suffices to check the values of $f$ and $\underline{f}_{c, \e}$ at the endpoints of the intervals in the piecewise definition above.
\end{proof}
 
\begin{corollary} \label{thm:L.2}
	 Fix $c > 0$, $q : = \frac{4}{1+p}$, $\epsilon \in (0,c^{1/q})$ and $u_{0, z} = \mathbf 1_{z<0}$. Let $(\underline u, W^{\underline u})$, $(\Omega, \mathcal G, (\mathcal F_t)_{t\geq 0}, \mathrm P)$ be as in Lemma \ref{thm:L.1}.
	Let $a := c^{\frac{1 - p}{4}}\epsilon^{- \frac{(1-p)q}{4}}$ and $b := \frac{c}{2\epsilon^q}$.
	Then there exists  an $\mathcal F_{t}$-white noise $W^{\tilde  u}$ such that $\tilde  u_{t,x}:= b u_{a^{-4}t,a^{-2}x}$ is a mild solution to
\begin{align} \label{eq:u_tilda}
	 \partial_{t} {\tilde  u}_{t,x} =
	\partial_x^2 {\tilde  u}_{t,x} +  {\tilde f}({\tilde  u}_{t,x}) +  c^{\frac{1}{2} - \frac{1 + p}{4}} \cdot\tilde{ \sigma}_{q, c, \e}({\tilde  u}_{t,x}) \dot W_{t,x}^{\tilde{ u}}, \quad t\geq 0, x\in \mathbb R,
\end{align}
	with initial data $bu_{0, \cdot}$. 
	Here 
\begin{align} 
& {\tilde  f}(z) :=
\begin{cases}
0, & \quad z\in (-\infty, 0],
\\z, & \quad z\in (0, \frac{1}{4}],
\\\frac{1}{2}- z, & \quad z\in (\frac{1}{4}, \frac{1}{2}],
\\0, & \quad z\in (\frac{1}{2}, \infty),
\end{cases}  
\end{align}
and 
\begin{align} \label{eq:sigma_tilda}
& {\tilde \sigma}_{q, c, \e}(z) :=  2^{-\frac{1}{2}} \sqrt{z(1 -  2c^{-1}\epsilon^q z)} \mathbf 1_{z\in [0,b]}, \quad z\geq 0.
\end{align}
\end{corollary} 
\begin{proof}
	Apply the rescaling result of Lemma \ref{rescaling}.
\end{proof}

\begin{proposition}\label{thm:L.3}
	There exists $c=c_{\eqref{eq:lb_scale}}, \e_0 \in (0, c_{\eqref{eq:lb_scale}}^{1/q})$, such that for each $\epsilon \in (0,\e_0)$, and $\tilde { u}^\epsilon$ as in Corollary \ref{thm:L.2}, we have
\begin{equation}
\label{eq:lb_scale}
\liminf_{t\to \infty} 
\frac{R(\tilde { u}^\epsilon_t)}{t} \geq 2 
- \frac{\pi^2}{|\log c^2|^2} - \frac{2\pi^2[11\log|\log c| - \log \alpha(|\log c|^{-3})}{|\log c^2|^3} =: \delta > 0, \quad \text{a.s.}
\end{equation}
where $q = 4/(1 + p)$ and $\alpha(a)$ is the largest $\alpha$ such that $(1-a) u \mathbf 1_{u\leq a} \leq {\tilde f}(u)$.
\end{proposition}
\begin{remark}
The rescaling result of Corollary \ref{thm:L.2} reduces the proof of the lower bound to 
showing that for a chosen positive $c$, the solution to equation \eqref{eq:u_tilda} has a 
speed bounded away from zero for arbitrarily small $\e.$ Furthermore, only the noise coefficient of \eqref{eq:u_tilda} and the initial condition depends on $\e.$  The drift ${\tilde{f}}$
remains unchanging with ${\tilde f}'(0) = 1$ and is Lipschitz with Lipschitz constant one.
\end{remark}
The proof of Proposition \ref{thm:L.3} follows from the proposition below. The coefficient $b = c/\e^q$  of the initial condition grows as $\e$ approaches zero. 
By the comparison principle, for small $\e$, we can couple
 a solution $u^{(1)}$ of \eqref{eq:u_tilda} with initial condition $u^{(1)}_{0, z} = b\mathbf 1(z < 0)$ on the same probability space as a solution $u^{(2)}$ of \eqref{eq:u_tilda}, with initial condition $u_{0, z}^{(2)} = 1(z <0)$, such that 
 $u^{(1)}_{t, x} \geq u^{(2)}_{t, x}$ for all $t, x$, almost surely. Therefore $R({u}^{(1)}_t) \geq R(u^{(2)}_t)$ and Proposition \ref{thm:L.3} follows from the proposition below. 

\begin{proposition}{\cite[Theorem 1.1]{MuellerMytnikQuastel2011Effect}}\label{thm:lb_sigma}
Assume that $u_{t, x}$ is a solution to the SPDE
\[
\p_t u_{t, x} = \p_x^2u_{t, x} + g(u_{t, x}) + \e\sigma(u_{t, x})\dot{W}
\]
with $u_{0, x} = \mathbf{1}{(x < 0)}$. Assume that $g$ is a Lipschitz function with Lipschitz constant one and satisfies
the standard KPP conditions:
\[
g(0)  = g(1) = 0; g'(0) = 1; \, 0< g(u) \leq u, u \in (0, 1).
\]
Further assume that \[\sigma^2(u) \leq u.\]
Then there exists an  $\e_0$ independent of $\sigma$, such that for all $\e \in (0, \e_0)$
\[
\liminf_{t \to \infty}\frac{R(u_{t, x})}{t} \geq 2 - \frac{\pi^2}{|\log \e^2|^2} - \frac{2\pi^2[11\log|\log\e| - \log \alpha(|\log \e|^{-3})]}{|\log \e^2|^3}.
\]
Here $\alpha(a)$ is the largest $\alpha$ such that 
\[
(1 - a)u\mathbf{1}_{(u \leq \alpha)} \leq g(u).
\]
\end{proposition}
\textbf{L: We need to be careful, since the initial condition of our model depends on $\epsilon$.}
\begin{proof}[Sketch]
	See \cite[Section 1]{MuellerMytnikQuastel2011Effect} for the set-up and relevant definitions.
	Proposition \ref{thm:lb_sigma} is the statement of \cite[Theorem 1.1]{MuellerMytnikQuastel2011Effect} except for the fact that the lower bound in Proposition \ref{thm:lb_sigma} holds for all $\sigma$ that satisfy $\sigma^2(u) \leq u.$
	This fact follows from the proofs of \cite[Lemmas 4.1 and 4.2]{MuellerMytnikQuastel2011Effect}.
	(The proof of \cite[Lemma 4.2]{MuellerMytnikQuastel2011Effect} is in the last section of  \cite{MuellerMytnikQuastel2011Effect}). 
	To see this is the case, note that the functions $h, \rho$ and constants $\gamma, \delta$ defined in the proof of \cite[Lemma 4.2]{MuellerMytnikQuastel2011Effect} are independent of $\sigma.$
	Consequently, the inequalities of \cite[(10.7)-(10.12)]{MuellerMytnikQuastel2011Effect} also hold assuming \cite[(10.6)]{MuellerMytnikQuastel2011Effect} holds.
	Under our assumption that $\sigma^2(u) \leq u$, it follows that (10.6) in \cite{MuellerMytnikQuastel2011Effect} does indeed hold.
Hence the constant $C_{(4.18)}$ given in the statement of \cite[Lemma 4.2]{MuellerMytnikQuastel2011Effect} is valid for all $\sigma$ with $\sigma^2(u) \leq u$, and the result follows.
\end{proof}

\begin{proof}[Proof of Proposition \ref{thm:L.3}]
	For fixed $c$, $\widetilde {{\sigma}}_{q, c, \e}$ satisfies the assumptions of Proposition \ref{thm:lb_sigma} for any small $\e$. So there exists a $c_0$ such that the lower bound for the speed of 
${ \tilde u^\e}$ with equation \eqref{eq:I.1} is $2 - \frac{\pi^2}{|\log c_0^{(1-p)/2}|^2}$ (minus a smaller order correction) for any small enough $\e$, and any $c \in (0, c_0]$. 
Simply choose $c_{\eqref{eq:lb_scale}} = c_0$ so that this lower bound for the speed is positive, then the result follows.
\end{proof}

\begin{proof}[Proof of Theorem \ref{thm:LU}(a)]
	 Let $u$ and $\underline u$ be given by Lemma \ref{thm:L.1} and let $a$, $b$, $q$, and $\tilde {{u}}$  be given by Corollary \ref{thm:L.2}.
    For each $t\geq 0$,
\begin{align} 
& R(\underline u_t) 
= \sup\{x\in \mathbb R: \underline u_{t,x} > 0\}
= \sup\{a^{-2}x\in \mathbb R: b\underline u_{a^{-4}(a^4t),a^{-2}x} > 0\}
\\&\label{eq:L.05}= a^{-2} \sup\{x\in \mathbb R: b\tilde{ u}_{a^4t,x} > 0\}
= a^{-2}R(\tilde { u}_{a^4t}).
\end{align}
	Now let $\epsilon_0$ and $c_{\eqref{eq:lb_scale}}$ be given by Proposition \ref{thm:L.3}. Then for $\epsilon \in (0,\epsilon_0)$ we have
\begin{align} 
& \mathrm V_\epsilon 
= \lim_{t\to \infty} \frac{R(u_t)}{t} 
\overset{\text{Lemma \ref{thm:L.1}}}\geq \liminf_{t\to \infty} \frac{R(\underline u_t)}{t}
\overset{\eqref{eq:L.05}}= \liminf_{t\to \infty} \frac{a^{-2}R(\tilde {u}_{a^4 t})}{t}
\\&= \liminf_{t\to \infty} \frac{a^{2}R(\tilde {u}_{a^4 t})}{a^4t}
\overset{\text{Proposition \ref{thm:L.3}}}\geq a^{2} \delta = \delta c_{\eqref{eq:lb_scale}}^{1/2-\frac{1 +p}{2}}\epsilon^{- 2\frac{1-p}{1+p}}.
\qedhere
\end{align}
\end{proof}
\end{note}
\section{Proof of Theorem~\ref{thm:LU}(\texorpdfstring{$\rm a$}{a})} \label{sec:L} 
	Note that we only have to prove the result for every   $f \in \{f^{(\delta)} : \delta \in (0,1/2]\}$  where
\[
	f^{(\delta)}(z) 
	:=  z^p\mathbf{1}_{0 \leq z \leq \delta} + (2\delta^p-\delta^{p-1}z)\mathbf 1_{\delta<z\leq 2\delta}, 
	\quad z\in [0,1].
\]
This is because for any general function $f$ satisfying \eqref{asp:0} and \eqref{asp:1}, there exists $c>0$ and   $0 < \delta \leq 1/2$  such that
\[
	f(z)\geq cf^{(\delta)}(z), \quad z\in [0,1]. 
\]
Using Lemmas~\ref{intro:comparison} and \ref{rescaling} we have  
\[
	V_{f, \epsilon} =c^{1/2}V_{c^{-1}f, c^{-1/4}\epsilon} \geq  c^{1/2}V_{f^{(\delta)}, c^{-1/4}\epsilon}.
\]
Therefore, to show $\liminf_{\epsilon \downarrow 0}\e^{2\frac{1 - p}{1 + p}}V_{f, \e}$ is positive we only have to show $\liminf_{\epsilon \downarrow 0}\e^{2\frac{1 - p}{1 + p}}V_{f^{(\delta)}, \e}$ is positive.
So in the remainder of this section, without loss of generality, let us fix an arbitrary   $\delta \in (0,1/2]$  and assume that $f \equiv f^{(\delta)}$.

 The idea of the proof is to use the comparison principle and the rescaling lemma for the SPDEs to replace our non-Lipschitz drift $f$ by some continuous tent function  
\[
	H(z;l,h) 
	:= 
\begin{cases}
	0, &\quad z\in (-\infty, 0],
	\\ \frac{h}{l}z, &\quad z \in (0,l), 
	\\ h, &\quad z = l,
	\\ 2h-\frac{h}{l} z, &\quad z \in (l,2l),
	\\ 0, & \quad z\in [l, \infty).
\end{cases}
\]
Here, the  parameters $l\in (0,1/2]$  and  $h>0$ of this tent function will be chosen more precisely later. 
This will allow us to analyze the speed of the system using the following result from \cite{MuellerMytnikQuastel2011Effect}.
For any $\beta > 0$, define $\mathcal C_{I,\beta} := \{\beta g: g\in \CI\}$.  
\begin{lemma}[\cite{MuellerMytnikQuastel2011Effect}]\label{thm:lb_sigma}
	There exists a $\gamma_0>0$ so that the following statement holds.
	Suppose that 
	\begin{itemize}
	\item 
		$\gamma\in (0,\gamma_0)$ and $\beta>0$; 
	\item 
		$\sigma$ is a non-negative function on $\mathbb R_+$ such that $\sigma^2$ is Lipschitz and that $\sigma^2(z) \leq z$ for every $z\in \mathbb R_+$; 
\item 
	for any initial condition $g\in \mathcal C_{I,\beta}$, the weak existence and the weak uniqueness of the SPDE
		\begin{equation} \label{eq:L.1}
			\p_t v= \p_x^2 v +   H(v; 1/2,1/2)  + \gamma\sigma(v)\dot{W}
\end{equation}
holds in $\cC_{I,\beta}$.  
	\end{itemize}
	Then for any $\mathcal C_{I,\beta}$-valued weak solution $v$ to \eqref{eq:L.1}, it holds that 
	\begin{equation} \label{eq:L.2}
		\liminf_{t \to \infty}\frac{R(v_{t, \cdot})}{t} \geq 2 - \frac{\pi^2}{|\log \gamma^2|^2} - \frac{2\pi^2[11\log|\log\gamma| - \log(1/2)]}{|\log \gamma^2|^3}, \quad \text{a.s.}
\end{equation}
\end{lemma}
\begin{remark}
	Lemma \ref{thm:lb_sigma} is a corollary of \cite[Theorem 1.1]{MuellerMytnikQuastel2011Effect} except that now the function $\sigma$ is not required to satisfy \cite[(1.5)]{MuellerMytnikQuastel2011Effect} and the parameter $\gamma_0$ is universal.  
	We justify this by observing that condition \cite[(1.5)]{MuellerMytnikQuastel2011Effect} is actually not needed in the proof of the lower bound of \cite[Theorem 1.1]{MuellerMytnikQuastel2011Effect}, and that the parameter $\gamma_0$, chosen as the $\epsilon_0$ from \cite[Lemma 4.1]{MuellerMytnikQuastel2011Effect}, is only related to the drift function $f$, which, in our case, is the fixed tent function $H(\cdot;1,1/2)$.  
\end{remark}

\begin{proof}[Proof of Theorem \ref{thm:LU}(a)]
\emph{Step 1.}
Let us fix the value $\gamma := \min\{\gamma_0, e^{-4}\}$ where $\gamma_0$ is given as in Lemma \ref{thm:lb_sigma}.
	One can easily check now the right hand side of \eqref{eq:L.2} is larger than $1$.

	\emph{Step 2.}
	Let $q:=4/(1+p)$ and define $\epsilon_0 >0$ so that   $(\epsilon_0/(\sqrt{2}\gamma))^q = \delta$. 
	Fix an arbitrary $\epsilon  \in (0,\epsilon_0)$ and define $\varepsilon :=\epsilon/(\sqrt{2} \gamma)$.
	Observe that   $\varepsilon^q \leq \delta$. 
		As a consequence, we have that the tent function   $H(\cdot; \varepsilon^q, \varepsilon^{qp})\leq f(\cdot)$.   

\emph{Step 3.}
Let $u$ be a $\CI$-valued weak solution to the SPDE \eqref{eq:I.1}. 
Now, from Step 2 and the comparison principle, we can construct a $\CI$-valued weak solution $\underline u$
to the SPDE 
\begin{equation} \label{eq:L.3}
	\partial_t \underline u = \partial_x^2 \underline u +   H(\underline u; \varepsilon^q, \varepsilon^{qp})  + \epsilon \sqrt{\underline u(1-\underline u)}\dot W^{\underline u}
\end{equation}
where $ W^{\underline u}$ is a white noise, so that the random field $\underline u$ is stochastically dominated by $u$.

\emph{Step 4.} Define the random field 
\[
	v_{t,x} := \beta \underline u_{\alpha^{-4}t, \alpha^{-2}x}, \quad (t,x)\in \mathbb R_+\times \mathbb R
\]
where $\alpha := \varepsilon^{-(1-p)q/4}$ and $\beta:= 1/(2\varepsilon^q)$. 
Then, we
can easily verify from Lemma~\ref{rescaling} that there exists a white noise $W^{v}$ such that 
$v$
is a $\mathcal C_{I,\beta}$-valued weak solution to the SPDE
\begin{equation} \label{eq:L.4}
	\partial_t v = \partial_x^2 v +   H(v;1/2,1/2)  + \gamma \sqrt{v(1-2\varepsilon^q v)} \dot W^v. 
\end{equation}

	\emph{Step 5}. We now verify from Lemma \ref{thm:lb_sigma} and Step 1 that $\liminf_{t\to \infty} \frac{R(v_{t,\cdot})}{t} \geq 1$  almost surely.
Note that the weak uniqueness of the SPDE \eqref{eq:L.4} in $\mathcal C_{I,\beta}$ is inherited from the weak uniqueness of the SPDE \eqref{eq:L.3} in $\CI$, which is justified by Theorem \ref{intro:thm_existenceUniqueness}(2).

\emph{Final Step}. 
Note that for any $t\geq 0$,
\begin{align}
	&R(\underline u_{t,\cdot}) 
	= \sup\{x\in \mathbb R: \underline u_{t,x} \neq 0\}
	= \sup\{\alpha^{-2}x\in \mathbb R: \beta\underline u_{\alpha^{-4}(\alpha^4t),\alpha^{-2}x} \neq 0\}
	\\&= \alpha^{-2}\sup\{x\in \mathbb R: v_{\alpha^4t,x} \neq 0\} 
	= \alpha^{-2} R(v_{\alpha^4t,\cdot}).
\end{align}
Therefore, 
\begin{align} 
& V_{f,\epsilon} 
= \lim_{t\to \infty} \frac{R(u_{t,\cdot})}{t} 
\overset{\text{Step 3}}\geq \liminf_{t\to \infty} \frac{R(\underline u_{t,\cdot})}{t}
= \liminf_{t\to \infty} \frac{\alpha^{-2}R(v_{\alpha^4 t,\cdot})}{t}
\\&= \liminf_{t\to \infty} \frac{\alpha^{2}R(v_{\alpha^4 t,\cdot})}{\alpha^4t}
\overset{\text{Step 5}}\geq \alpha^2 = \varepsilon^{-2\frac{1-p}{1+p}} 
= (\sqrt{2}\gamma)^{2\frac{1-p}{1+p}} \cdot \epsilon^{-2\frac{1-p}{1+p}}. 
\end{align}
Finally, noticing that $\epsilon$ is arbitrarily chosen from $(0,\epsilon_0)$, and that $\gamma= \min \{\gamma_0, e^{-4}\}$ is independent of the choice of this $\epsilon$, by taking $\epsilon \downarrow 0$, we get 
\[
	\liminf_{\epsilon \downarrow 0} \epsilon^{2\frac{1-p}{1+p}}V_{f,\epsilon} \geq (\sqrt{2}\gamma)^{2\frac{1-p}{1+p}} > 0. \qedhere
\]
\end{proof}

\section{proof of Theorem \ref{thm:LU}(\texorpdfstring{$\rm b$}{b})} \label{sec:U} 
	From now on, we write  $\sigma(z) = \sqrt{z(1-z)}$ for $z\in [0,1]$ since we will only consider the Wright-Fisher noise.
We first assume 
without loss of generality 
that $f = \tilde f$ where \[\tilde f(z) := z^p \wedge \sqrt{1-z}, \quad z\in [0,1].\]
We can do this because for any general $f$ satisfying \eqref{asp:0} and \eqref{asp:2}, it holds that
\[
	K:=\sup_{z\in [0,1]} f(z)/ \tilde f(z) < \infty.
\]
By using Lemmas~\ref{intro:comparison} and \ref{rescaling}, 
we can then verify that  
\[
	\epsilon^{2\frac{1-p}{1+p}}V_{f,\epsilon}
	=\epsilon^{2\frac{1-p}{1+p}} \sqrt{K} V_{f/K, \epsilon/ K^{1/4}}
	\leq  K^{\frac{1}{1+p}} \tilde \epsilon^{2\frac{1-p}{1+p}} V_{\tilde f, \tilde \epsilon} 
\]
where $\tilde \epsilon = \epsilon/ K^{1/4}$.
From here, it is clear that if Theorem~\ref{thm:LU}(b) holds for $f = \tilde f$, then it also holds for every $f$ satisfying \eqref{asp:0} and \eqref{asp:2}.

	To get an upper bound for the speed, we will construct a sequence of updating frontiers  and control the propagation of $u$ using an updating procedure. 
	The updating frontiers are shifts of a non-increasing function $\tilde F\in \mathcal C_I$ which will be specified  below in \eqref{eq:U.5}.
	More precisely, the $n$-th updating frontier will be defined as
\[
\tilde F^{(n)}(x) 
:= \tilde F(x-n d\mathrm v T), 
\quad x\in \mathbb R, \quad n \in \mathbb Z_+
\] 
	where $d,\mathrm v, T>0$ are parameters that will be specified  below  in \eqref{eq:U.6} and \eqref{eq:U.3}.
	Define $\xi_0 = 0$, and inductively for each $n\in \mathbb Z_+$, construct a stopping time $\xi_{n+1}$ and a $\mathcal C_I$-valued process $t\mapsto \tilde u_t$ on $[\xi_n,\xi_{n+1})$, with a driving space-time white noise $W^{\tilde u}$, such that
\[
\begin{cases}
	&\partial_t \tilde u = \partial_x^2 \tilde u + f(\tilde u) + \epsilon \sigma(\tilde u)  \dot W^{\tilde u}, \quad t\in [\xi_n,\xi_{n+1})
	\\ & \tilde u_{\xi_n,\cdot} = \tilde F^{(n)}(\cdot)
	\\ & \xi_{n+1}
= (\xi_n+T)\wedge \inf \{t\geq \xi_n: \tilde u_{t,x} > \tilde F^{(n+1)}(x) \text{ for some }x\in \mathbb R\}.
\end{cases}
\]
	Note that the $\mathcal C_I$-valued process $(\tilde u_t)_{t\geq 0}$ is not continuous anymore, because it may jump at the stopping times $(\xi_n)_{n\in \mathbb N}$.
	By the comparison principle, $\tilde u$ will travel faster than the original process. 
	This allows us to get an upper bound of $V_f(\epsilon)$ by calculating the speed of the new process $\tilde u$.
	However, in order to get a reasonably good upper bound, we need to choose   $\tilde F$,  $d$, $\mathrm v$ and $T$, parameters in this updating procedure,    carefully according to the noise strength $\epsilon$. 
		So, for the sake of precision, let us first give our choice of   $\tilde F, d, \mathrm v$  and $T$ here, along with several other quantities that will be used throughout the rest of the paper.  
\begin{statement}
\item \label{eq:U.1}
	Let us fix a constant  $\theta \in (1/2,1)$ and define
\[ 
	\kappa
	:=( p^{\frac{p}{1-p}}- p^{\frac{1}{1-p}}) \theta^{\frac{p}{p-1}} (1-\theta)^{\frac{1}{p-1}}.
	\]
	Since $p\in [1/2,1)$ we can verify that $\kappa>0$ and $\kappa^{p-1} \leq 1$. 
\item \label{eq:U.13}
	Fix a $\mathcal K>0$ large enough so that
\[
	2^5 \sum_{n=1}^\infty \exp( -2^{-22} e^{-2\theta(2-\theta)} \mathcal K e^{(2\theta-1) n} ) \leq 1/8.
\]
\item \label{eq:U.2}
	Let us fix a constant $\gamma>0$ small enough so that 
\[ 
	2^7 \sqrt{\gamma} \exp\{3\gamma \nu\} 
	\leq 1/8;
	\qquad
	\nu^p \leq \nu/4;
	\qquad
	\mathcal K/ \gamma \geq 2.
\]
where
$
	\nu:= 2^4+ (2^5 \mathcal K  + 2^{14} \mathcal K^{1/2}) \gamma^{-1}. 
$
\item \label{eq:U.6}
	Define $k := \mathcal K/ \gamma$ and $d := k+\nu + 1$.
\end{statement}
\begin{statement}
\item \label{eq:U.3}
	For each $\epsilon > 0$ let us define $\varepsilon$, $\mathrm v$, $T$ and $L$ so that the following hold: 
	\[ 
		\varepsilon = \gamma \epsilon^2,
		\quad \varepsilon = \kappa\mathrm v^{\frac{p+1}{p-1}},
		\quad T = \mathrm v^{-2},
		\quad L = \mathrm v^{-1}.
	\]
\item \label{eq:U.4}
	Let us fix an $\epsilon_0>0$ small enough so that for any $\epsilon \in (0,\epsilon_0)$, \[\nu \varepsilon L \leq 1/4 \quad \text{and} \quad 2^{13} \sqrt{\epsilon^2 L} \leq 1/4.\]
\item\label{eq:U.5}
	For each $\epsilon > 0$, define $ F(x):= \frac{\varepsilon}{\theta \mathrm v}(e^{-\theta \mathrm v x}-1) \mathbf 1_{x\leq 0}$ and $\tilde F(x) := 1\wedge F(x)$ for every $x\in \mathbb R.$
\end{statement}
\begin{remark*}
	The constants $\theta, \kappa, \mathcal K, \gamma, \nu,k, d $ and $\epsilon_0$ above are independent of the noise strength $\epsilon$. 
	We are choosing those constants in a technical way, far from their optimal choice, in order to simplify several formulations below.
\end{remark*}
\begin{remark*}
	The variables $\varepsilon, \mathrm v, T, L, F(\cdot)$ and $\tilde F(\cdot)$  are chosen depending on the noise strength $\epsilon$. 
	The intended intuition behind those variables are discussed in Subsection \ref{sec:PS}.  
	In particular, one can verify from \eqref{eq:U.3} that the speed of the moving boundary is $\mathrm v= (\kappa^{-1} \gamma)^{-\frac{1-p}{1+p}} \epsilon^{-2\frac{1-p}{1+p}}$, the length of the time interval to apply the Girsanov transformation is $T = (\kappa^{-1} \gamma)^{2\frac{1-p}{1+p}} \epsilon^{4\frac{1-p}{1+p}}$, and the typical distance for the solution to travel in a time interval of length $T$ is $L = \sqrt{T} = \mathrm vT$. 
\end{remark*}
	With the choice of the above quantities, we can verify the following proposition whose proof is postponed to Section \ref{sec:H}.

\begin{proposition} \label{thm:H}
	For any $\epsilon \in (0,\epsilon_0)$ and 
	any $\CI$-valued weak solution $u$ 
	to the SPDE \eqref{eq:I.1} on a 
	filtered probability space 
	$(\Omega, \mathcal G, (\mathcal {F}_t)_{t \geq 0}, \mathrm P)$ with 
	$u_{0,\cdot} = \tilde F$, 
	it holds that
\[
	\mathrm P\big( \forall (t,x)\in [0,T]\times \mathbb R, u_{t,x} \leq   \tilde F^{(1)}(x)  \big) \geq 1/2. 
\]
\end{proposition}	
Below we show that this proposition is sufficient for the proof of Theorem~\ref{thm:LU}(b). 

\begin{proof}[Proof of Theorem~\ref{thm:LU}(b)]
	Fix an arbitrary $\epsilon \in (0,\epsilon_0)$, and let $u$ be a $\CI$-valued weak solution to the SPDE \eqref{eq:I.1} with   $u_{0,\cdot} = \tilde F$.  
\begin{note}
	Let $c>0$ be given by Proposition \ref{thm:H}.
	Define $\tilde {\mathrm v} := c \mathrm v$.
\end{note}
	Let the process $\tilde u$ be constructed using the updating procedure described at the beginning of this section with parameters   $\tilde F$,  $d$, $\mathrm v$ and $T$ given as in \eqref{eq:U.1}--\eqref{eq:U.5}. 
	The corresponding updating times 
	are denoted by $(\xi_n)_{n\geq 0}$. 
	By the comparison principle, without loss of generality, we assume that $\tilde u$ and $u$ are constructed on the same filtered probability space $(\Omega, \mathcal G,(\mathcal F_t)_{t\geq 0}, \mathrm P)$  
	such that $\tilde u_{t,x}\geq u_{t,x}$ for each $(t,x)\in \mathbb R_+\times \mathbb R$. 
	Note, from the strong Markov property, that $(\xi_{n+1}-\xi_n)_{n\in \mathbb Z_+}$ is a sequence of i.i.d. random variables.
	Also note from Proposition \ref{thm:H} that
\[
	\mathrm P( \xi_1=T) 
	\geq \mathrm P\big(\forall (t,x)\in [0,T] \times \mathbb R, u_{t,x} \leq   \tilde F^{(1)}(x)  \big) 
	\geq 1/2.
\]
	So by the strong law of large numbers, we have almost surely
\[
	\lim_{n\to \infty} \frac{\xi_n}{n}
	=\mathbb E[\xi_1]
	\geq T \cdot \mathrm P(\xi_1 = T) 
	\geq  T/2.
\]
	Also observe that from the way $\tilde u$ is constructed, we always have
\[
	R(\tilde u_{\xi_n,\cdot}) 
	= R(  \tilde F^{(n)} )
	= n d \mathrm v T, 
\quad n\in \mathbb N.
\]
	Now we can verify that 
\begin{align}
	&V_{f,\epsilon}
	:=\lim_{t\to \infty} \frac{R(u_t)}{t}
	\leq \liminf_{t\to \infty}  \frac{R(\tilde u_t)}{t}
	\leq \liminf_{n\to \infty} \frac{R(\tilde u_{\xi_n})}{\xi_n}
	\\&\leq 2 d \mathrm v
	= 2d(\kappa^{-1}\gamma \epsilon^2)^{-\frac{1-p}{1+p}}.
\end{align}
	Finally, note that $d$, $\kappa$ and $\gamma$ are independent of the choice of $\epsilon\in (0,\epsilon_0)$, and hence we are done.
\end{proof}

\section{proof of Proposition \ref{thm:H}}\label{sec:H}
	Let us fix 
	an arbitrary $\epsilon\in (0,\epsilon_0)$. 
	Let $\varepsilon, \mathrm v, L, T, F, k$ and $\nu$ be given as in \eqref{eq:U.1}-\eqref{eq:U.5}. 
	The function $F$ plays an important role in the updating procedure described in 
	Section \ref{sec:U}.
	The main reason we choose $F$ as in \eqref{eq:U.5} is given by the following analytical lemma.	
	Let us 
	define
	\begin{equation} \label{eq:H.04}
	\varrho_{t,x}
	:= F(x-\mathrm vt), 
	\quad t\geq 0, x\in \mathbb R;
	\quad
	\bar{f}(z) 
	:= (1-\theta)\theta \mathrm v^2 z+ (1-\theta)\mathrm v \varepsilon,
	\quad z\in \mathbb R_+.
\end{equation}
 As explained in the beginning of Section \ref{sec:U}, we only consider the case $f = \tilde f$. 

\begin{lemma} \label{thm:H.05}
	For every $z \in [0,1]$, it holds that $\bar f(z) \geq f(z)$.
	Moreover, $(\varrho_{t,x}: t\geq 0, x\in \mathbb R)$ is the solution to the PDE
\begin{equation}\label{eq:H.06}
	\begin{cases}
	\partial_t \varrho = \partial_x^2 \varrho + \bar{f}(\varrho), &\quad x< \mathrm vt,
	\\ \varrho = 0, &\quad x\geq \mathrm vt,
\end{cases}
\end{equation}
with initial condition $\varrho_{0,\cdot} = F$.
\end{lemma}
\begin{proof}
	\emph{Step 1.}
	It can be verified directly that $\varrho$ satisfies \eqref{eq:H.06}.
\begin{note}
In fact, for any $x< \mathrm vt$,
\begin{align}
	&\partial_t \varrho_{t,x} - \partial_{x}^2 \varrho_{t,x}
	= \partial_t F(x-\mathrm vt) - \partial_{x}^2 F(x-\mathrm vt)
	\\& = \partial_t \Big(\frac{\varepsilon}{\theta\mathrm v} (e^{-\theta \mathrm v(x-\mathrm vt)} - 1)\Big) - \partial_{x}^2 \Big(\frac{\varepsilon}{\theta\mathrm v} (e^{-\theta \mathrm v(x - \mathrm vt)} - 1)\Big)
	\\& =  (1 -\theta) \mathrm v \varepsilon e^{-\theta \mathrm v(x - \mathrm vt)} 
	= (1-\theta)\mathrm v (\varepsilon + \theta \mathrm v F(x-\mathrm vt)) 
	 = \bar f(\varrho_{t,x}).
\end{align}
It is also clear that $\varrho_{t,x}= 0$ for all $x\geq \mathrm vt$.
\end{note}

	\emph{Step 2.}
	To finish the proof, we show that $\bar f(z) \geq z^p$ for all $z\geq 0$. 
	Note that 
\begin{itemize} 
	\item 
	$\bar f$ is a linear function with slope $(1-\theta) \theta \mathrm v^2$; and 
\item
	$z\mapsto z^p$ is a concave function on $[0,\infty)$. 
\end{itemize}
	So we only have to show that $\bar f(z_0)\geq z_0^p$ where $z_0>0$ solves $\partial_z z^p |_{z = z_0} = (1-\theta) \theta \mathrm v^2$.
	Actually, it is easy to calculate that $z_0 = \big(p^{-1}(1-\theta) \theta \mathrm v^2\big)^{\frac{1}{p-1}}$.
	From this, and how $\kappa$ and $\varepsilon$ are defined in \eqref{eq:U.1} and \eqref{eq:U.3}, we can verify that $\bar f(z_0) - z_0^p  = 0$.
\begin{note}
	\begin{align}
		&\bar f(z_0) - z_0^p 
		= (1-\theta)\mathrm v (\varepsilon + \theta \mathrm v z_0) - z_0^p
		\\& = (1-\theta)\mathrm v \Big(\varepsilon + \theta \mathrm v \big(p^{-1}(1-\theta) \theta \mathrm v^2\big)^{\frac{1}{p-1}}\Big) - (p^{-1}(1-\theta) \theta \mathrm v^2)^{\frac{p}{p-1}}
		\\& =\varepsilon(1-\theta)\mathrm v +  p^{\frac{1}{1-p}}\big((1-\theta) \theta \mathrm v^2\big)^{\frac{p}{p-1}}- p^{\frac{p}{1-p}}\big((1-\theta) \theta \mathrm v^2\big)^{\frac{p}{p-1}}
		\\& =\kappa\mathrm v^{\frac{p+1}{p-1}} (1-\theta)\mathrm v +  p^{\frac{1}{1-p}}\big((1-\theta) \theta \mathrm v^2\big)^{\frac{p}{p-1}}- p^{\frac{p}{1-p}}\big((1-\theta) \theta \mathrm v^2\big)^{\frac{p}{p-1}}
		\\& =( p^{\frac{p}{1-p}}- p^{\frac{1}{1-p}}) \theta^{\frac{p}{p-1}} (1-\theta)^{\frac{1}{p-1}}  \mathrm v^{\frac{p+1}{p-1}} (1-\theta)\mathrm v +  (p^{\frac{1}{1-p}} - p^{\frac{p}{1-p}})\big((1-\theta) \theta \mathrm v^2\big)^{\frac{p}{p-1}}
		\\& = 0.
		\qedhere
	\end{align}
\end{note}
\end{proof}
	 Recall that $\sigma(z)= \sqrt{z(1-z)}$ for $z\in [0,1]$. 
	To build a connection between $\varrho$ and $u$ we use the following two SPDEs:
	\begin{equation} \label{eq:H.1}
\begin{cases}
	\partial_t v 
	=\partial^2_x v + f(v) + \epsilon \sigma(v) \dot W^v,
	& x < \mathrm v t,
	\\v
	= 0,
	& x\geq \mathrm v t;
\end{cases}
\end{equation}
and
\begin{equation}\label{eq:H.15} 
\begin{cases}
	\partial_t \bar v 
	=\partial^2_x \bar v + \bar f(\bar v) + \epsilon \sigma(\bar v) \dot W^{\bar v},
	& x < \mathrm v t,
	\\\bar v
	= 0,
	& x\geq \mathrm v t.
\end{cases}
\end{equation}

Let us be precise about the solution concept of \eqref{eq:H.1} and \eqref{eq:H.15} by first introducing a kernel $G^{(\mathrm v)}$. 
For each $(s,y)\in \mathbb R_+\times \mathbb R$, let $B = (B_t)_{t\geq s}$ be a one dimensional Brownian motion with generator $\partial_x^2$ initiated at time $s$ and position $y$ defined on a filtered probability space with probability measure denoted as $\Pi_{s,y}$. 
	In the sequel, we will use $\Pi_{s,y}$ for the expectation with respect to the measure $\Pi_{s,y}$ in addtion to for the measure itself.
	Let us define
	\begin{equation}
	\label{def:rho}
	\rho := \inf\{t: B_t \geq \mathrm vt \}.
	\end{equation}
 Denote by $\mathrm b\mathscr B(\mathbb R)$ the space of all bounded Borel functions on $\mathbb R$. 
	It can be verified that for each $0\leq s< t< \infty$ and $y < \mathrm vs$ there exists a unique continuous map $x \mapsto G^{(\mathrm v)}_{s,y;t,x}$ from $(-\infty,\mathrm v t)$ to $(0,\infty)$ such that
\begin{equation} 
	\int_{-\infty}^{\mathrm vt} G^{(\mathrm v)}_{s,y;t,x} \varphi(x)\mathrm dx = \Pi_{s,y}[\varphi( B_t); t < \rho], \quad \varphi \in b\mathscr B(\mathbb R).
\end{equation}
	The precise expression of $G^{(\mathrm v)}$ can be calculated using the reflection principle
	and the Girsanov transformation 
	for the Brownian motion (see \cite[Proof of Lemma 6.2]{MuellerMytnikQuastel2011Effect}).
	We define $G^{(\mathrm v)}_{s,y;t,x} = 0$ 
	on $\{(s,y;t,x): 0\leq s< t, y < \mathrm vs, x < \mathrm vt\}^c$ for convention.
\begin{note}
	For any $(s,y),(t,x)\in \mathbb R_+\times \mathbb R$ we have
\begin{equation}
	G^{\mathrm v}_{s,y;t,x} 
	= \rho^{(1)}_{s,y-\mathrm vs;t,x- \mathrm vt} - \rho^{(-1)}_{s,y - \mathrm vs; t, x- \mathrm vt}
\end{equation}
	where $\rho^{(i)}_{s,y;t,x} := e^{-\frac{\mathrm v}{2}(x - y) - \frac{\mathrm v^2}{4}(t-s)} G_{s,y;t,i x}\mathbf 1_{y,x\leq 0}$.
\end{note}

We say $(\Omega, \mathcal G, (\mathcal F_t)_{t\geq 0}, \mathrm P, v, W^v)$ is a weak solution to the SPDE \eqref{eq:H.1}, if $(\Omega, \mathcal G, (\mathcal F_t)_{t\geq 0}, \mathrm P)$ is a filtered probability space on which a predictable random field $v$ and a white noise $W^v$ are defined so that $(v_{t,\cdot}:t\geq 0)$ is a $\Ctem$-valued continuous process satisfying  
\begin{align} 
	&  v_{t,x} 
	= \iint_0^t G^{(\mathrm v)}_{s,y;t,x} M^v(\mathrm ds\mathrm dy),
	\quad \text{a.s.} 
	\quad t>0, x\in \mathbb R,
\end{align}
	where 
\begin{equation} \label{eq:H.2}
	M^v(\mathrm ds\mathrm dy) 
	:= v_{0,y} \delta_0(\mathrm ds) \mathrm dy + f(v_{s,y})\mathrm ds\mathrm dy + \sigma(v_{s,y})W^v(\mathrm ds\mathrm dy).
\end{equation}
With some abuse of notation,  
we sometimes only use the random field $v$  
to represent a weak solution to the SPDE \eqref{eq:H.1} if there is no risk of confusion.  
Given a subset $\hat{\mathcal C} \subset \Ctem$, we say a weak solution $v$ to the SPDE \eqref{eq:H.1} is a $\hat{\mathcal C}$-valued weak solution if $(v_{t,\cdot}:t\geq 0)$ is a $\hat{\mathcal C}$-valued process.  
The concept of 
weak solution to the SPDE \eqref{eq:H.15}
is given in a similar way.

	The main idea behind the proof of Proposition \ref{thm:H} is that $v$ can be shown to satisfy the property which is similar to that desired for $u$ in Proposition \ref{thm:H}; 
	and if $u$ and $v$ have the same initial value   $\tilde F$  then they can be coupled in such a way that they don't deviate from each other ``too much" before time $T$. 
	This coupling is described in the following proposition whose proof is postponed to Section \ref{sec:C}.   
	The difference between $u$ and $v$ in the coupling will be controlled by a 
	random field $w$. 

\begin{proposition}\label{thm:C}
	There exists $(v, W^v; \bar v, W^{\bar v}; w, W^{w}; u,W^u)$ defined 
	on a filtered probability space $(\Omega, \mathcal G, (\mathcal F_t)_{t\geq 0}, \mathrm P)$ 
	such that the followings holds. 
\begin{enumerate}
\item 
		$W^v$, $W^{\bar v}$, $W^w$ and $W^u$ are space-time white noises adapted to the same filtration $(\mathcal F_t)_{t\geq 0}$. 
	Furthermore, $W^v$ and $W^w$ are independent of each other, that is to say, the two families of random variables $\{W^v_t(A):t\geq 0, A \in \mathcal B_F(\mathbb R)\}$ and $\{W^w_t(A):t\geq 0, A\in \mathcal B_F(\mathbb R)\}$ are independent. 
\item 
	$v$ is a $\Cunit$-valued weak solution to the SPDE \eqref{eq:H.1} with   $v_{0,\cdot} = \tilde F$.  
\item
$\bar v$  is a $\mathcal C_{\mathrm {tem}}^+$-valued weak solution to the SPDE \eqref{eq:H.15} with $\bar v_{0,\cdot}=F$.
\item
	Almost surely $\bar v \geq v$ on $\mathbb R_+\times \mathbb R$.
\item 
	$u$ is a $\CI$-valued weak solution to the SPDE \eqref{eq:I.1} with    $W=W^u$ and  $u_{0,\cdot} = \tilde F$.  
\item 
	$w$ is a non-negative predictable random field such that $(w_{t,\cdot}:t\geq 0)$ is a $\Ctem$-valued continuous process, and for every
	$\phi \in \cC_{\mathrm{c}}^\infty(\mathbb R_+ \times \mathbb R)$  
	and $t\geq 0$,
\begin{equation} \label{eq:H.3}
\begin{multlined}
	\int \phi_{t,x}w_{t,x} \mathrm dx 
	 = \iint_0^t w_{s,y}(\partial_s \phi_{s,y} + \partial_y^2 \phi_{s,y}) \mathrm ds\mathrm dy  + \int_0^t \phi_{s,\mathrm vs}\mathrm dA_s
	 \\ + \iint_0^t \phi_{s,y} \big(f^w_{s,y}\mathrm ds\mathrm dy + \epsilon \sigma^{w}_{s,y}W^{w}(\mathrm ds\mathrm dy) \big), \quad \text{a.s.}
\end{multlined}
\end{equation}
Here, $f^w$ and $\sigma^w$ are random fields defined as follows: for every $(s,y)\in \mathbb R_+\times \mathbb R$,
\begin{align} 
	& f_{s,y}^w 
	:= |f(v_{s,y} + w_{s,y}) - f(v_{s,y})| \mathbf 1_{y\in [ - L,  \mathrm vT + L], v_{s,y}+w_{s,y}\leq \nu \varepsilon L},
	\\ & \sigma_{s,y}^w 
	:= \sqrt{|\sigma(v_{s,y} + w_{s,y})^2 - \sigma (v_{s,y})^2| \vee \frac{w_{s,y}}{2}};
\end{align}
	and $(A_t)_{t\geq 0}$ is an adapted non-decreasing continuous process 
	such that for every $t\geq 0$,
\[
	A_t = \iint_0^t \Pi_{s,y}(\rho \leq t) M^v(\mathrm ds\mathrm dy), \quad \text{a.s.}
\]

	where $M^v$ is defined in \eqref{eq:H.2}.

\item
	It holds almost surely that 
	\[
		u = v+ w \quad \text{on} \quad [0,\tau]\times \mathbb R.
	\]
	Here the optional time 
	\begin{equation}
	\tau:=  \min \{T, \tau_1, \tau_2\}
	\end{equation}
	  is defined using
\begin{align} 
	& \tau_1 
	:= \inf \Big\{t \in [0,T]: \int_0^t \mathrm ds \int_{[-L,\mathrm vT+L]^c}  w_{s,y}\mathrm dy >0 \Big\}, 
      \\& \tau_{2} 
	:= \inf \big\{t \in [0,T]: v_{t,x} + w_{t,x} \geq \nu \varepsilon L \text{ for some } x\in [ -L, \mathrm vT + L]\big\},
\end{align} 
	with the convention that the infimum of the empty set is infinite.
\end{enumerate}
\end{proposition}
\begin{remark*}
	In the above proposition, $\tau_1$ and $\tau_2$ are the stopping times for the field $w$ getting too large. 
	In particular, $\tau_1$ is the stopping time when the support of $w$ can not be contained in $[-L,\mathrm vT+L]$, and $\tau_2$ is the stopping when the maximum of $v+w$ on $[-L, \mathrm vT+L]$ exceeds  the level $\nu \varepsilon L$. 
\end{remark*}

	We will show that $v$ satisfies a similar property which we desired for $u$.  
	This is done in the following proposition whose proof is postponed to Section \ref{sec:V}.
\begin{proposition}\label{thm:V}
		Let $v$ be  given by Proposition \ref{thm:C}. 
		Then $\mathrm P \big( \tau_{3}< T\big) < 1/8$ where 
		\begin{equation}
			\tau_3 
			:= \inf\{t\in [0,T]: v_{t,x}\geq F(x-\mathrm v t) + k \varepsilon L e^{-\theta \mathrm v(x - \mathrm vt)} \mathbf 1_{x\leq \mathrm vt} \text{ for some } x\in (-\infty, \mathrm vt]\}.
		\end{equation}	
	\end{proposition}
	
	From Proposition \ref{thm:C} (7), the difference between $u$ and $v$ can be controlled by the process $w$ up to the stopping time $\tau$.
	We use the following two propositions to control this stopping time. 
	Their proofs are postponed later to Sections \ref{sec:W} and \ref{sec:B} respectively.
\begin{proposition} \label{thm:W}
	Let $\tau_1$ be given by Proposition \ref{thm:C}.
	Then it holds that 
	$
	\mathrm P(\tau_1 < T) < 1/8.
	$
\end{proposition}
	
\begin{proposition} \label{thm:B}
	Let $\tau_1$ and $\tau_2$ be given by Proposition \ref{thm:C}.
	Let $\tau_3$ be given by Proposition \ref{thm:V}.
	Then it holds that $\mathrm P(\tau_2 < T, \tau_3 \geq T, \tau_1 \geq T) < 1/8$.
\end{proposition}
	We are now ready to give the proof of Proposition \ref{thm:H} using Propositions \ref{thm:C}-\ref{thm:B}. 
\begin{proof}[Proof of Proposition \ref{thm:H}]
	Thanks to the weak uniqueness, we only have to prove the desired result for a specific $\CI$-valued weak solutions with initial value   $\tilde F$. 
	So, let us take the weak solution $u$ to the SPDE~\ref{eq:I.1}  given as in Proposition \ref{thm:C}. 
	Let  
	  also $v,w, \tau_1, \tau_2$ be as in Proposition \ref{thm:C}, and $\tau_3$ as in Proposition \ref{thm:V}.
	To get the desired result
	 we only have to verify that
\begin{equation} \label{eq:H.4} 
	\bigcap_{i=1,2,3}\{\tau_i \geq T\} 
	\subset \big\{ \forall (t,x)\in [0,T]\times \mathbb R, u_{t,x} \leq \tilde F^{(1)}(x)\big\},
\end{equation}
since by Propositions \ref{thm:V}--\ref{thm:B},
\begin{align}
	&\mathrm P\Big(\bigcap_{i=1,2,3}\{\tau_i \geq T\}\Big) 
	 = 1- \mathrm P( \{\tau_1 < T\} \cup  \{\tau_2 < T, \tau_1 \geq T, \tau_3 \geq T\}\cup \{\tau_3 < T\}) 
	 \\&\geq 1 - \big(\mathrm P(\tau_1 < T) + \mathrm P( \tau_2< T, \tau_1 \geq T, \tau_3 \geq T) + \mathrm P( \tau_3< T) \big)
	 \geq 1/2.
\end{align}
	In the rest of the proof, we verify \eqref{eq:H.4}. 
	First note that for any $x\in \mathbb R$ and $l>0$, 
\begin{align}
	&F(x - l) - F(x) 
	= \frac{\varepsilon}{\theta \mathrm v}(e^{-\theta \mathrm v (x-l)}-1) \mathbf 1_{x-l\leq 0} - \frac{\varepsilon}{\theta \mathrm v}(e^{-\theta \mathrm v x}-1) \mathbf 1_{x\leq 0}
	\\ 
	\label{eq:prop41_a}
	&\geq \frac{\varepsilon}{\theta \mathrm v}e^{-\theta \mathrm v x}(e^{\theta \mathrm v l} - 1) \mathbf 1_{x\leq 0}
	\geq \varepsilon le^{-\theta \mathrm v x} \mathbf 1_{x\leq 0}
	\geq \varepsilon l \mathbf 1_{x\leq 0}.
\end{align}
	Then notice that almost surely on the event $\cap_{i=1,2,3}\{\tau_i \geq T\}$, we have 
\begin{align}
	u_{t,x} 
	&= v_{t,x}+w_{t,x},
	& \quad t\in [0,T], x\in \mathbb R;
	\\ w_{t,x} 
	&\leq \nu \varepsilon L \mathbf 1_{x\in [-L, \mathrm vT+L]},
	& \quad t\in [0,T], x\in \mathbb R;
	\\ v_{t,x} 
	&\leq F(x- \mathrm v t) + k \varepsilon L e^{-\theta \mathrm v(x-\mathrm vt)} \mathbf 1_{x \leq \mathrm vt},
	& \quad t\in [0,T], x\in \mathbb R.
\end{align}
Therefore, almost surely on event $\cap_{i=1,2,3}\{\tau_i \geq T\}$, we have that for any $(t,x)\in [0,T]\times \mathbb R$,
\begin{align}
	&u_{t,x}  = v_{t,x} + w_{t,x} 
	\leq F(x- \mathrm v t - kL)  + \nu \varepsilon L \mathbf 1_{x\in [-L, \mathrm vT+L]}
	\\ & \leq F(x- \mathrm v T - kL)  + \nu  \varepsilon L \mathbf 1_{x - \mathrm vT- kL\leq 0}
	\leq F\big(x - \mathrm vT - (k+\nu)L\big).
\end{align}
	In the second   inequality  above, we used the fact that $F$ is non-increasing and that $k\geq 1$. 
The third inequality follows easily by~\eqref{eq:prop41_a}. 
 
Finally, noticing that $u_{t,x}\leq 1$ and according to \eqref{eq:U.3} that $L = \mathrm vT$, \eqref{eq:H.4} follows. 
\end{proof}

\section{proof of Proposition \ref{thm:C}} \label{sec:C}
	The main idea is that the SPDE \eqref{eq:H.1} can be written equivalently as
\[
	\partial_t v = \partial_x^2 v + f(v) +  \epsilon \sigma(v) \dot W^v  - \delta_{\mathrm vt}(x) \dot A_t
\]
	where $(A_t)_{t\geq 0}$ is 
		this adapted, real-valued, continuous, non-decreasing process for Proposition \ref{thm:C} (6).
We will refer to $(A_t)_{t\geq 0}$ as the killing process of $v$ at its boundary.
		The existence of this killing process is given by the next lemma.
	Recall that, under probability $\Pi_{s,y}$, $(B_r)_{r\geq s}$ is a Brownian motion with generator $\partial_x^2$ initiated at time $s$ and position $y$, and $\rho$ is given by \eqref{def:rho}. 
\begin{lemma} \label{thm:C.1}
	Suppose that  
	$v$ is a $\Cunit$-valued weak solution to the SPDE \eqref{eq:H.1} with $v_{0,\cdot} =   \tilde F$. 
	Then 
\begin{enumerate}
	\item for each $\phi \in \cC_{\mathrm c}^{1,2}(\mathbb R_+\times \mathbb R)$ and $t\geq 0$ it holds almost surely that,
\begin{equation}
\begin{multlined}
\label{eq:6.1.1}
	\int \phi_{t,x} v_{t,x} \mathrm dx
	=  \iint_0^t \phi_{s,y} M^v(\mathrm ds\mathrm dy) + \iint_0^t v_{s,y}(\partial_s \phi_{s,y} +
	\partial^2_y\phi_{s,y}) \mathrm ds \mathrm dy  
	\\ - \iint_0^t  \Pi_{s,y}[\phi_{\rho, B_\rho};t\geq \rho]M^v(\mathrm ds\mathrm dy)
\end{multlined}
\end{equation}
	where $M^v$ is given by \eqref{eq:H.2} and $\rho$ is defined in ~\eqref{def:rho}; 
\item	
	there exists an adapted, real-valued, almost surely non-decreasing  continuous  process $(A_t)_{t\geq 0}$ satisfying that for each $t\geq 0$ and bounded Borel measurable function $\psi$ on $\mathbb R_+$, 
\begin{equation} \label{eq:C.2} 
	 \int_0^t \psi_{s}\mathrm dA_s=\iint_0^t  \Pi_{s,y}[\psi_{\rho};t\geq \rho]M^v(\mathrm ds\mathrm dy), \quad a.s. 
\end{equation}
\end{enumerate}
\end{lemma}	 

\begin{proof}[Proof of Lemma \ref{thm:C.1} (1)]
	\emph{Step 1.} 
	Using the stochastic Fubini theorem (cf. \cite[Lemma 2.4]{Iwata1987AnInfinite} for example) we can verify that for all $ t\geq 0$
	\begin{equation} 
		\int v_{t,x} \phi_{t,x} \mathrm dx 
		=  \int  \mathrm dx   \iint_0^t \phi_{t,x}  G^{(\mathrm v)}_{s,y;t,x}  M^v(\mathrm ds\mathrm dy) 
		= \iint_0^t M^v(\mathrm ds\mathrm dy) \int \phi_{t,x} G^{(\mathrm v)}_{s,y;t,x}  \mathrm dx, \quad \rm a.s. 
	\end{equation}
	
	\emph{Step 2.}
	Using the stochastic Fubini theorem again we can verify that  for all $t\geq 0$
	\begin{align} 
		&\iint_0^t  (\partial_r\phi_{r,x}+\partial_x^2\phi_{r,x}) v_{r,x} \mathrm dr\mathrm dx
		=\iint_0^t \mathrm dr\mathrm dx \iint_0^r  (\partial_r\phi_{r,x}+\partial_x^2\phi_{r,x}) G^{(\mathrm v)}_{s,y;r,x} M^v(\mathrm ds\mathrm dy)
		\\& = \iint_0^t M^v(\mathrm ds\mathrm dy) \iint_s^t G^{(\mathrm v)}_{s,y;r,x}(\partial_r\phi_{r,x} + \partial_x^2\phi_{r,x}) \mathrm dr\mathrm dx, \quad \rm a.s. 
	\end{align}
	
	\emph{Step 3.}
	We show that for each $(s,y)\in \mathbb R_+\times \mathbb R$,
	\begin{equation} 
		\int G^{(\mathrm v)}_{s,y;t,x}\phi_{t,x}\mathrm dx+\Pi_{s,y}[\phi_{\rho, B_{\rho}}; t\geq \rho] 
		= \phi_{s,y}+ \iint_s^t G^{(\mathrm v)}_{s,y;r,x} (\partial_r \phi_{r,x}+\partial_x^2\phi_{r,x}) \mathrm dr\mathrm dx.
	\end{equation} 
	In fact, according to Ito's formula (see \cite[p.~147]{RevuzYor1999Continuous} for example), we know that under probability $\Pi_{s,y}$,
	\begin{align} 
		& \phi_{t,B_t} - \phi_{s,y} - \int_s^t (\partial_r \phi_{r,x} + \partial_x^2 \phi_{r,x})|_{x=B_r} \mathrm dr
		= \int_s^t \partial_x \phi_{r,x} |_{x=B_r}\mathrm dB_r, \quad t\geq s,
	\end{align}
	is a zero-mean $L^2$-martingale.
	Then, according to optional sampling theorem (see \cite[Theorem 7.29]{Kallenberg2002Foundations} for example) we have 
	\begin{align} 
		&  \Pi_{s,y}[\phi_{t\wedge \rho, B_{t\wedge \rho}}] 
		= \phi_{s,y} + \int_s^t \Pi_{s,y}[(\partial_r\phi_{r,x}+\partial_x^2\phi_{r,x})|_{x=B_r}; r< \rho]\mathrm dr.
	\end{align}
	
	\emph{Step 4.}
	We note from the fact $\phi\in \cC^{1,2}_c(\mathbb R_+\times \mathbb R)$ and that $v,f,\sigma$ take values in $[0,1]$, the following stochastic integral
	\[
	\iint_0^t \phi_{s,y} M^v(\mathrm ds dy)
	= \int \phi_{0,y} v_{0,y}  \mathrm dy  + \iint_0^t \phi_{s,y}  f(v_{s,y})\mathrm ds\mathrm dy  + \iint_0^t \phi_{s,y} \sigma(v_{s,y})W^v(\mathrm ds\mathrm dy)
	\]
	is well-defined.
	
	\emph{Final Step.}
	We verify that almost surely, 
	\begin{align} 
		&  \int v_{t,x} \phi_{t,x} \mathrm dx 
		\overset{\text{Step 1}}=   \iint_0^t M^v(\mathrm ds\mathrm dy) \int G^{(\mathrm v)}_{s,y;t,x} \phi_{t,x} \mathrm dx 
		\\ &\overset{\text{Step 3}}=   \iint_0^t M^v(\mathrm ds\mathrm dy) \Big(\phi_{s,y} + \iint_s^t G_{s,y;r,x}^{(\mathrm v)} (\partial_r\phi_{r,x}+\partial_x^2\phi_{r,x}) \mathrm dr\mathrm dx - \Pi_{s,y}[\phi_{\rho, B_{ \rho}}; t\geq  \rho]\Big) 
		\\ &\begin{multlined}
			\overset{\text{Steps 2 and 4}}= \iint_0^t \phi_{s,y}M^v(\mathrm ds\mathrm dy) + \iint_0^t  (\partial_r\phi_{r,x}+\partial_x^2\phi_{r,x}) v_{r,x} \mathrm dr\mathrm dx 
			\\ - \iint_0^t \Pi_{s,y}[\phi_{ \rho, B_{ \rho}}; t\geq \rho] M^v(\mathrm ds\mathrm dy)
		\end{multlined}
	\end{align}
	as desired.
\end{proof}
\begin{proof}[Proof of Lemma \ref{thm:C.1} (2)]
	For each $t\geq 0$, choose
	a $\phi \in \cC_\mathrm c^{1,2}(\mathbb R_+\times \mathbb R)$ such that $\phi_{s,\mathrm vs} = 1$ for every $s\in [0,t]$. 
	Use  this $\phi$ in~\eqref{eq:6.1.1} to 
	get that for each $t\geq 0$ the following
	random variable is well defined:
	\begin{align}
		&\tilde A_t
		:= \iint_0^t \Pi_{s,y}[t\geq \rho] M^v(\mathrm ds\mathrm dy)
		\\&=-\int \phi_{t,x} v_{t,x} \mathrm dx+\iint_0^t \phi_{s,y} M^v(\mathrm ds\mathrm dy) + \iint_0^t v_{s,y}(\partial_s \phi_{s,y} + \partial^2_y\phi_{s,y}) \mathrm ds \mathrm dy.
	\end{align}
	It's easy to see that $(\tilde A_t)_{t\geq 0}$ has a continuous modification which will be denoted by $(A_t)_{t\geq 0}$.  
	To see that $(A_t)_{t\geq 0}$ is almost surely non-decreasing,
	define \[\phi^{(m)}_{t,x} := \varphi_{(x-\mathrm vt)m}, \quad (t,x)\in \mathbb R_+\times \mathbb R, m\in \mathbb N\]
	where
	\[
	\quad \varphi_x :=   \begin{cases} 
		0, & x\in 
		[1,\infty), 
		\\ (18x^2+6x+1)(1-x)^3,  
		& x \in [0,1],
		\\(x+1)^3, 
		& x \in [-1,0],
		\\ 0, & x \in (-\infty,-1].
	\end{cases}
	\]
	Use $\phi^{(m)}$ instead of $\phi$ in~\eqref{eq:6.1.1} to get that, for each $m\in \mathbb N$ and $t\geq 0$,  
	\[
	A_t = \mathrm I^{(m)}_t + \mathrm {II}^{(m)}_t,\quad {\rm a.s.}
	\]
	where
	\[
	\mathrm I^{(m)}_t:=\iint_0^t \phi_{s,y}^{(m)} M^v(\mathrm ds\mathrm dy) -
	\int \phi_{t,x}^{(m)}v_{t,x}\mathrm dx 
	\]
	and
	\[
	\mathrm {II}^{(m)}_t:=\iint_0^t v_{s,y}(\partial_s \phi_{s,y}^{(m)} + \partial^2_y\phi_{s,y}^{(m)}) \mathrm ds \mathrm dy. 
	\]
	Observe that $\phi_{s,y}^{(m)} \downarrow 0$ as $m\uparrow \infty$ on
	$\{(s,y)\in \mathbb R_+\times \mathbb R: y < \mathrm vs\}$. 
	This allows us to use the monotone convergence theorem and \cite[Proposition 17.6]{Kallenberg2002Foundations} to get that for each $t\geq 0$,
	$\mathrm{I}_t^{(m)}$  
	converges
	to $0$ in probability as $m\to \infty$.
	Fix
	arbitrary $r< t$ in $\mathbb R_+$. 
	\cite[Lemma 4.2]{Kallenberg2002Foundations} allows us to choose an unbounded $\mathbf{N}\subset \mathbb N$ so that $\mathrm I_t^{(m)} - \mathrm I_r^{(m)}$ convergence to $0$ almost surely as $m\to \infty, m\in \mathbf{N}$.
	Now we have almost surely
	\begin{equation}
		\label{eq:W.61}  
		\begin{split}
			&A_t - A_r = \lim_{m\to \infty, m\in \mathbf{N}} (\mathrm {II}_t^{(m)} - \mathrm {II}_r^{(m)})
			\\& = \lim_{m\to \infty, m\in \mathbf{N}} \int_r^t  \mathrm ds \int_{\mathrm vs-\frac{1}{m}}^{\mathrm vs} v_{s,y} 
			\cdot 3 \big(1+(y-\mathrm vs)m\big)\Big(-\mathrm vm \big(1+(y-\mathrm vs)m\big) + 2m^2\Big) \mathrm dy 
			\\& \geq \lim_{m\to \infty, m\in \mathbf{N}} \int_r^t  \mathrm ds \int_{\mathrm vs-\frac{1}{m}}^{\mathrm vs} v_{s,y} 
			\cdot 3 \big(1+(y-\mathrm vs)m\big)(-\mathrm vm + 2m^2) \mathrm dy 
			\geq 0.
		\end{split}
	\end{equation}
	From this and the fact that $(A_t)_{t\geq 0}$ has continuous sample path, we have  that $t\mapsto A_t$ is non-decreasing almost surely. 
	
	Denote by $\mathrm b\mathscr B(\mathbb R_+)$ the space of bounded Borel functions on $\mathbb R_+$. 
	Fix a time $t\geq 0$ and define $\mathscr H:= \{\psi \in \mathrm b \mathscr B(\mathbb R_+):  \eqref{eq:C.2} \text{ holds for }\psi \}$.
	From the definition of $(A_t)_{t\geq 0}$ and the fact that it has non-decreasing sample path almost surely, we can verify that  $\mathscr K \subset \mathscr H$ where $\mathscr K$ is given by \eqref{eq:T.63}.
	One can verify from monotone convergence theorem and \cite[Proposition 17.6]{Kallenberg2002Foundations} that $\mathscr H$ is a monotone vector space in the sense of \cite[p.~364]{Sharpe1988General}. 
	Also observe that $\mathscr K$ is closed under multiplication.
	Therefore using monotone class theorem (\cite[Theorem A0.6]{Sharpe1988General}) we get $\mathrm b \mathscr B(\mathbb R_+) = \sigma(\mathscr K) \subset \mathscr H$. 
\end{proof}

\begin{note}
	\begin{proof}[Proof of Lemma \ref{thm:W.2}]
		Denote by $m_t := e^{H_t}, t\in [0,\infty]$ where
		\[
		H_t := \iint_0^t h_{s,y} W(\mathrm ds\mathrm dy)-\frac{1}{2}\iint_0^t h_{s,y}\mathrm ds\mathrm dy,\quad t\geq 0.
		\]
		From Ito's formula \cite[Theorem 3.3 on p.~147]{RevuzYor1999Continuous} and \cite[Theorem 18.23]{Kallenberg2002Foundations}, we have 
		\[
		m_t - 1 = \iint_0^t m_s h_{s,y} W(\mathrm ds\mathrm dy), \quad t\geq 0
		\]
		is a uniformly integrable martingale.
		Observe that 
		\[
		\frac{\mathrm d\mathrm Q|_{\mathcal F_t}}{\mathrm d\mathrm P|_{\mathcal F_t}}= \mathrm P[m_\infty|\mathcal F_t] = m_t, \quad t\geq 0.
		\]
		Thanks to \cite[Proposition 18.20]{Kallenberg2002Foundations}, there is no need to distinguish between the quadratic (co)variations defined in $\mathrm P$ and $\mathrm Q$. 
		
		Let us now verify that $W'$ is a white noise under probability $\mathrm Q$. 
		To do this we only have to observe that 
		\begin{enumerate}
			\item for each $t\geq 0$ and disjoint $A,B\in \mathcal B_F(\mathbb R)$ almost surely $W'_t(A\cup B) = W'_t(A) + W'_t(B)$; 
			\item for any $A \in \mathcal B_F(\mathbb R)$, $(W'_t(A))_{t\geq 0}$ is continuous $\mathrm Q$-martingale; and
			\item for any $A\in \mathcal B_F(\mathbb R)$ and $t\geq 0$ almost surely $\langle W'_\cdot(A) \rangle_t = t \operatorname{Leb}(A)$.
		\end{enumerate}
		Here, (2) is due to \cite[Theorem 18.19]{Kallenberg2002Foundations} and the fact that the quadratic covariation between  $(W'_t(A))_{t\geq 0}$ and $(m_{t})_{t\geq 0}$ is zero.
		
		Denote by $\mathscr M_{\mathrm {loc}}^\mathrm P$ and $\mathscr M_{\mathrm {loc}}^{\mathrm Q}$ the collection of continuous local martingales with respect to probabilities $\mathrm P$ and $\mathrm Q$ respectively. 
		Consider a map $\Psi$ from $\mathscr L_{\mathrm{loc}}^2$ to $\mathscr M_{\mathrm {loc}}^\mathrm P$ given by
		\[
		\Psi: g \in \mathscr L_{\mathrm{loc}}^2 \mapsto \mathrm P\text{-}\iint_0^\cdot g_{s,y} W'(\mathrm ds\mathrm dy).
		\]
		It turns out $\Psi$ is also a map from $\mathscr L_{\mathrm{loc}}^2$ to $\mathscr M_{\mathrm {loc}}^\mathrm Q$.
		In fact, for each $g\in \mathscr L_{\mathrm{loc}}^2$, from \cite[Theorem 18.19]{Kallenberg2002Foundations} and the fact that the quadratic covariation between  $\Psi(g)$ and $(m_{t})_{t\geq 0}$ is zero, we have $\Psi(g)\in \mathscr M_{\mathrm {loc}}^\mathrm Q$.
		We now observe that the followings hold: 
		\begin{enumerate}
			\item for each $g,g'\in \mathscr L_{\mathrm {loc}}^2$ and $\lambda,\lambda'\in \mathbb R$, $\Psi(\lambda g + \lambda' g) = \lambda \Psi(g) + \lambda' \Psi(g')$; 
			\item For each $0\leq a\leq b < \infty$, $A\in \mathcal B_F(\mathbb R)$ and real-valued bounded $\mathcal F_a$ measurable random variable $X$,  it holds that $\Psi(g)= X \cdot (W'_{b\wedge\cdot }(A)- W'_{a\wedge \cdot}(A))$ where for each $(s,y)\in \mathbb R_+\times \mathbb R$, $g_{s,y} := X \mathbf 1_{(a,b]}(s)\mathbf 1_A(y)$;
			\item For each $g\in \mathscr L_{\mathrm {loc}}^2$, $\langle \Psi(g)\rangle =\iint_0^\cdot g_{s,y}^2\mathrm ds \mathrm dy$; and 
			\item For each $g\in \mathscr L_{\mathrm {loc}}^2$ and optional time $\sigma$, $\Psi(g)_{\sigma \wedge \cdot} = \Psi(g^\sigma)$ where $g^\sigma_{s,y}:= \mathbf 1_{s\leq \sigma} g_{s,y}$ for each $(s,y)\in\mathbb R_+\times \mathbb R$.
		\end{enumerate}
		Finally, note that a map $\Psi$ from $\mathscr L_{\mathrm{loc}}^2$ to $\mathscr M_{\mathrm {loc}}^\mathrm Q$ satisfying the above (1)-(4) is unique (c.f. \cite[Lemma 2.1]{Iwata1987AnInfinite}) and must coincide with the map
		\[
		g \in \mathscr L_{\mathrm{loc}}^2 
		\mapsto \mathrm Q\text{-}\iint_0^\cdot g_{s,y} W'(\mathrm ds\mathrm dy).
		\qedhere
		\] 
	\end{proof}
\end{note}

\begin{proof}[Proof of Proposition \ref{thm:C}]
	\emph{Step 1.} 
	Using a strategy similar to the proof of \cite[Proposition 5.1]{MuellerMytnikQuastel2011Effect}, we can verify that 
	there exists a filtered probability space $(\Omega, \mathcal G,(\mathcal F_t)_{t\geq 0}, \mathbb P)$ and stochastic elements $(v, W^v; \bar v, W^{\bar v}; w, W^w)$ on it such that 
\begin{itemize}
\item
	$W^v$, $W^{\bar v}$ and $W^w$ are white noises where $W^w$ is independent of $W^v$; and
	
\item
	(2), (3), (4) and (6) of Proposition \ref{thm:C} hold. 
\end{itemize}
	Lemma \ref{thm:C.1} is used here to justify that the second term on the right hand side of \eqref{eq:H.3} is well-defined.
	
	\emph{Step 2.}
	Define optional time $\tau$ as in Proposition \ref{thm:C} (7) using $v$ and $w$ constructed in Step 1. 
	Extending the space $(\Omega, \mathcal G, (\mathcal F_t)_{t\geq 0}, \mathrm P)$ if necessary, we can construct a pair 
	$(u,\widetilde W)$  
	so that
\begin{itemize}
	\item
$\wW$ is a white noise independent of $(v, W^v; \bar v, W^{\bar v}; w, W^w)$;
\item
$u$  is a $\Cunit$-valued weak solution to the SPDE  
\begin{equation}
\begin{cases}	
	u = v + w,  
	&\quad \text{ on } [0,\tau]\times \mathbb R,
	\\ \partial_t u = \partial_x^2 u + f(u) + \epsilon \sigma(u) \dot {\wW}, & \quad \text{ on } [\tau, \infty)\times \mathbb R. 
\end{cases}
\end{equation}
\end{itemize}
\begin{note}	
	More precisely, extending $\Omega$ if necessary, we can construct an $\mathcal F$-white noise $\wW$ which is independent of $(v, W^v; w, W^w)$, and an adapted $[0,1]$-valued continuous random field $u$ so that for each $(t,x)\in \mathbb R_+\times \mathbb R$ almost surely on the event $\{t\leq \tau\}$, $u_{t,x} =  v_{t,x} + w_{t,x}$; and for each $t\geq 0$ and $\phi\in C_c^\infty(\mathbb R_+\times \mathbb R)$ almost surely on the event $\{t>\tau\}$,
\begin{equation}\label{eq:C.3}
\begin{multlined}
	\int \phi_{t,x} u_{t,y} \mathrm dy - \int \phi_{\tau,y} u_{\tau,y} \mathrm dy = \iint_\tau^t \phi_{s,x} \Big( f(u_{s,y})\mathrm ds\mathrm dy + \epsilon \sigma(u_{s,y}){\wW}(\mathrm ds\mathrm dy)\Big)
	\\+  \iint_\tau^t u_{s,y}(\partial_s \phi_{s,y} + \partial^2_y\phi_{s,y}) \mathrm ds \mathrm dy.
\end{multlined}
\end{equation}
\end{note}
	The existence of such $u$ after the optional time $\tau$ is due to \cite[Theorem 2.6]{Shiga1994Two}.

\emph{Step 3.}
We will show that almost surely 
\begin{equation}\label{eq:C.3}
	\sigma^w_{s,y} = \sqrt{\sigma(v_{s,y}+w_{s,y})^2 - \sigma(v_{s,y})^2}, \quad  (s,y)\in [0,\tau] \times \mathbb R
\end{equation}
and
\[
	f^w_{s,y} = f(v_{s,y}+w_{s,y})-f(v_{s,y}), \quad (s,y)\in [0,\tau]\times \mathbb R.
\]
This is obvious for $(s,y) \in [0,\tau]\times [-L, \mathrm vT + L]^c$ since in this case $w_{s,y} = 0$.
Let us now consider the case $(s,y)\in [0,\tau]\times [-L, \mathrm vT+L]$.
Note that in this case, from the definition of $\tau$ and \eqref{eq:U.4} we have $v_{s,y} + w_{s,y}\leq \nu \varepsilon L  \leq 1/4$. 
	We also observe that for any $\mathsf v, \mathsf w\in [0,1]$ satisfying $\mathsf v+\mathsf w \leq 1/4$, we have
	$ \mathsf w/2 \leq \sigma(\mathsf v+\mathsf w)^2 - \sigma(\mathsf v)^2$
	and $0\leq f(\mathsf v+\mathsf w) - f(\mathsf v)$,
	and therefore
\[
	\sqrt{|\sigma(\mathsf v + \mathsf w)^2 - \sigma (\mathsf v)^2| \vee \frac{\mathsf w}{2}}
	=\sqrt{\sigma(\mathsf v + \mathsf w)^2 - \sigma(\mathsf v)^2}
\]
	and $|f(\mathsf v+\mathsf w) - f(\mathsf v)| = f(\mathsf v+\mathsf w) - f(\mathsf v).$
	Thus, the desired result in this step follows.

\emph{Step 4.}
	We can verify that there exists a white noise $W^u$ so that for any $g \in \mathscr L_\mathrm{loc}^2$, 
\begin{equation} \label{eq:C.4}
\begin{multlined}
	\iint_0^t g_{s,y}W^u(\mathrm ds\mathrm dy)
	= \iint_0^t \frac{g_{s,y}\mathbf 1_{B_{s,y}} }{\sigma(v_{s,y}+w_{s,y})} \big(\sigma(v_{s,y}) W^v(\mathrm ds\mathrm dy)  + \sigma^w_{s,y} W^w(\mathrm ds \mathrm dy)\big) 
	\\ + \iint_0^t g_{s,y}\mathbf 1_{B_{s,y}^c}  {\wW}(\mathrm ds\mathrm dy),
	\quad t\geq 0, \; {\rm a.s.}, 
\end{multlined}
\end{equation}
	where for each $(s,y)\in \mathbb R_+\times \mathbb R$ the event $B_{s,y}:= \{s\leq \tau, \sigma(v_{s,y} + w_{s,y})>0\}$.
	To see this, one only have to calculate the quadratic variation of the right hand side of \eqref{eq:C.4} using \eqref{eq:C.3} and the fact that $W^{w}, W^v$ and $\wW$ are mutually independent.

\emph{Final step.}
Observe from Step 4 that for any $(t,x)\in \mathbb R_+\times \mathbb R$, 
\begin{equation} 
	\iint_0^t G_{s,y;t,x}\sigma(u_{s,y}) W^u(\mathrm ds\mathrm dy)
	= \iint_0^t G_{s,y;t,x} \Big(\sigma(v_{s,y}) W^v(\mathrm ds\mathrm dy)  + \sigma^w_{s,y} W^w(\mathrm ds \mathrm dy)\Big)
\end{equation}	
	holds almost surely on the event $\{t\leq \tau\}$; also
\begin{equation} 
	\iint_\tau^t G_{s,y;t,x}\sigma(u_{s,y}) W^u(\mathrm ds\mathrm dy) = \iint_\tau^t G_{s,y;t,x}\sigma(u_{s,y}) {\wW}(\mathrm ds\mathrm dy) 
\end{equation}
	holds almost surely on the event $\{t > \tau\}$.
	We can then verify that 
  $u$ is a $\Cunit$-valued weak solution to the SPDE \eqref{eq:I.1} with $W=W^u$, $u_{0,\cdot} =   \tilde F $.  
	Thus, Proposition \ref{thm:C} (5) follows from Theorem \ref{intro:thm_existenceUniqueness} (2).
\begin{note}
	Let us now verify that $u$ and $W^u$, constructed in Steps 2 and 4 respectively, satisfy Proposition \ref{thm:C}(3).
	Now by \eqref{eq:H.3}, \eqref{eq:T.455}, \eqref{eq:C.2}, \eqref{eq:T.47} and Step 2 we get that almost surely on the event $\{t\leq \tau\}$,
\begin{equation} \label{eq:T.485}
\begin{multlined}
	\int \phi_{t,y} u_{t,y} \mathrm dy - \int \phi_{0,y}u_{0,y}\mathrm dy= \iint_0^t \phi_{s,y} \Big( f(u_{s,y})\mathrm ds\mathrm dy + \epsilon \sigma(u_{s,y})W^u(\mathrm ds\mathrm dy)\Big)
	\\+  \iint_0^t u_{s,y}(\partial_s \phi_{s,y} + \partial^2_y\phi_{s,y}) \mathrm ds \mathrm dy.
\end{multlined}
\end{equation}
	From \eqref{eq:C.3} and \eqref{eq:T.48} we get that \eqref{eq:T.485} also holds almost surely on the event $\{t> \tau\}$.
\end{note}
\end{proof}
	
\begin{note}
	We now give the proof of Lemma \ref{thm:C.1}.
\begin{proof}[Proof of Lemma \ref{thm:C.1} (1)]
\emph{Step 1.} 
	Using the stochastic Fubini theorem (cf. \cite[Lemma 2.4]{Iwata1987AnInfinite} for example) we can verify that almost surely
\begin{equation} 
	\int v_{t,x} \phi_{t,x} \mathrm dx 
	=  \int  \mathrm dx   \iint_0^t \phi_{t,x}  G^{(\mathrm v)}_{s,y;t,x}  M^v(\mathrm ds\mathrm dy) 
	= \iint_0^t M^v(\mathrm ds\mathrm dy) \int \phi_{t,x} G^{(\mathrm v)}_{s,y;t,x}  \mathrm dx. 
\end{equation}
	
\emph{Step 2.}
	Using the stochastic Fubini theorem again we can verify that almost surely
\begin{align} 
	&\iint_0^t  (\partial_r\phi_{r,x}+\partial_x^2\phi_{r,x}) v_{r,x} \mathrm dr\mathrm dx
	=\iint_0^t \mathrm dr\mathrm dx \iint_0^r  (\partial_r\phi_{r,x}+\partial_x^2\phi_{r,x}) G^{(\mathrm v)}_{s,y;r,x} M^v(\mathrm ds\mathrm dy)
	\\& = \iint_0^t M^v(\mathrm ds\mathrm dy) \iint_s^t G^{(\mathrm v)}_{s,y;r,x}(\partial_r\phi_{r,x} + \partial_x^2\phi_{r,x}) \mathrm dr\mathrm dx.
\end{align}
		
\emph{Step 3.}
	We show that for each $(s,y)\in \mathbb R_+\times \mathbb R$,
\begin{equation} 
	\int G^{(\mathrm v)}_{s,y;t,x}\phi_{t,x}\mathrm dx+\Pi_{s,y}[\phi_{\rho, B_{\rho}}; t\geq \rho] 
	= \phi_{s,y}+ \iint_s^t G^{(\mathrm v)}_{s,y;r,x} (\partial_r \phi_{r,x}+\partial_x^2\phi_{r,x}) \mathrm dr\mathrm dx.
\end{equation} 
	In fact, according to Ito's formula (see \cite[p.~147]{RevuzYor1999Continuous} for example), we know that under probability $\Pi_{s,y}$,
\begin{align} 
	& \phi_{t,B_t} - \phi_{s,y} - \int_s^t (\partial_r \phi_{r,x} + \partial_x^2 \phi_{r,x})|_{x=B_r} \mathrm dr
	= \int_s^t \partial_x \phi_{r,x} |_{x=B_r}\mathrm dB_r, \quad t\geq s,
\end{align}
	is a zero-mean $L^2$-martingale.
	Then, according to optional sampling theorem (see \cite[Theorem 7.29]{Kallenberg2002Foundations} for example) we have 
\begin{align} 
	&  \Pi_{s,y}[\phi_{t\wedge \rho, B_{t\wedge \rho}}] 
	= \phi_{s,y} + \int_s^t \Pi_{s,y}[(\partial_r\phi_{r,x}+\partial_x^2\phi_{r,x})|_{x=B_r}; r< \rho]\mathrm dr.
\end{align}
		
		\emph{Step 4.}
		We note from the fact $\phi\in C^{1,2}_c(\mathbb R_+\times \mathbb R)$ and that $v,f,\sigma$ take values in $[0,1]$, the following stochastic integral
		\[
		\iint_0^t \phi_{s,y} M^v(\mathrm ds dy)
		= \int \phi_{0,y} v_{0,y}  \mathrm dy  + \iint_0^t \phi_{s,y}  f(v_{s,y})\mathrm ds\mathrm dy  + \iint_0^t \phi_{s,y} \sigma(v_{s,y})W^v(\mathrm ds\mathrm dy)
		\]
		is well-defined.
		
		\emph{Final Step.}
		We verify that almost surely, 
		\begin{align} 
			&  \int v_{t,x} \phi_{t,x} \mathrm dx 
			\overset{\text{Step 1}}=   \iint_0^t M^v(\mathrm ds\mathrm dy) \int G^{(\mathrm v)}_{s,y;t,x} \phi_{t,x} \mathrm dx 
			\\ &\overset{\text{Step 3}}=   \iint_0^t M^v(\mathrm ds\mathrm dy) \Big(\phi_{s,y} + \iint_s^t G_{s,y;r,x}^{(\mathrm v)} (\partial_r\phi_{r,x}+\partial_x^2\phi_{r,x}) \mathrm dr\mathrm dx - \Pi_{s,y}[\phi_{\tau, B_\tau}; t\geq \tau]\Big) 
			\\ &\begin{multlined}
				\overset{\text{Steps 2 and 4}}= \iint_0^t \phi_{s,y}M^v(\mathrm ds\mathrm dy) + \iint_0^t  (\partial_r\phi_{r,x}+\partial_x^2\phi_{r,x}) v_{r,x} \mathrm dr\mathrm dx 
				\\ - \iint_0^t \Pi_{s,y}[\phi_{\tau, B_\tau}; t\geq \tau] M^v(\mathrm ds\mathrm dy)
			\end{multlined}
		\end{align}
		as desired.
\end{proof}
\begin{proof}[Proof of Lemma \ref{thm:C.1} (2)]
	For each $t\geq 0$, choosing a $\phi \in C_c^{1,2}(\mathbb R_+\times \mathbb R)$ such that $\phi_{s,\mathrm vs} = 1$ for every $s\in [0,t]$. 
	Put this $\phi$ into Lemma \ref{thm:C.1} (1), we get that the following real-valued random variable is well defined:
	\begin{align}
		&\tilde A_t
		:= \iint_0^t \Pi_{s,y}[t\geq \rho] M^v(\mathrm ds\mathrm dy)
		 \\&=-\int \phi_{t,x} v_{t,x} \mathrm dx+\iint_0^t \phi_{s,y} M^v(\mathrm ds\mathrm dy) + \iint_0^t v_{s,y}(\partial_s \phi_{s,y} + \partial^2_y\phi_{s,y}) \mathrm ds \mathrm dy.
	\end{align}
	It's easy to see that $(\tilde A_t)_{t\geq 0}$ has a continuous modification which will be denoted as $(A_t)_{t\geq 0}$. 
	To see $(A_t)_{t\geq 0}$ is almost surely non-decreasing, we put a sequence of $(\phi^{(m)})_{m\in \mathbb N} \subset C_c^{1,2}(\mathbb R_+\times \mathbb R)$ into Lemma \ref{thm:C.1} (1), where $\phi^{(m)}_{t,x} := \varphi_{(x-\mathrm vt)m}, (t,x)\in \mathbb R_+\times \mathbb R, m\in \mathbb N$ and
	\[
	\quad \varphi_x :=   \begin{cases} 
		0, & x\in [1,\infty)
		\\{\rm smooth}, & x \in [0,1],
		\\ 1+2x+x^2, & x \in [-1,0],
		\\ 0, & x \in (-\infty,-1].
	\end{cases}
	\]
	Then we get for each $m\in \mathbb N$ and $t\geq 0$, almost surely
	$
	A_t = \mathrm I^{(m)}_t + \mathrm {II}^{(m)}_t
	$
	where
	\[
	\mathrm I^{(m)}_t:=\iint_0^t \phi_{s,y}^{(m)} M^v(\mathrm ds\mathrm dy) - \int \phi_{t,y}^{(m)}v_{t,y}\mathrm dy
	\]
	and
	\[
		\mathrm {II}^{(m)}_t:=\iint_0^t v_{s,y}(\partial_s \phi_{s,y}^{(m)} + \partial^2_y\phi_{s,y}^{(m)}) \mathrm ds \mathrm dy. 
	\]
	Observe that $\phi_{s,y}^{(m)} \downarrow 0$ as $m\uparrow \infty$ on $\{(s,y)\in \mathbb R_+\times \mathbb R: x < \mathrm vs\}$.
	This allows us to use monotone convergence theorem and \cite[Proposition 17.6]{Kallenberg2002Foundations} to get that for each $t\geq 0$, $I_t^{(m)}$ convergence to $0$ in probability as $m\to \infty$.
	Fix two arbitrary $r< t$ in $\mathbb R_+$. 
	\cite[Lemma 4.2]{Kallenberg2002Foundations} allows us to choose an unbounded $N\subset \mathbb N$ so that $\mathrm I_t^{(m)} - \mathrm I_r^{(m)}$ convergence to $0$ almost surely as $m\to \infty, m\in N$.
	Now we have almost surely
	\begin{align}
		&A_t - A_r = \lim_{m\to \infty, m\in N} (\mathrm {II}_t^{(m)} - \mathrm {II}_r^{(m)})
		\\&\label{eq:W.61}= \lim_{m\to \infty, m\in N} \int_r^t  \mathrm ds \int_{\mathrm vs-\frac{1}{m}}^{\mathrm vs} v_{s,y}\Big(-2\mathrm vm \big(1+(y-\mathrm vs)m\big) + 2m^2\Big) \mathrm dy
		\\&\geq \lim_{m\to \infty, m\in N} \int_r^t  \mathrm ds \int_{\mathrm vs-\frac{1}{m}}^{\mathrm vs} v_{s,y}(-2\mathrm vm + 2m^2) \mathrm dy
		\geq 0.
	\end{align}
	From this and the fact that $(A_t)_{t\geq 0}$ has continuous sample path, we have almost surely $t\mapsto A_t$ is non-decreasing. 
	
	Denote by $\mathrm b\mathscr B(\mathbb R_+)$ the space of bounded Borel functions on $\mathbb R_+$. 
	Fix a time $t\geq 0$ and define $\mathscr H:= \{\psi \in \mathrm b \mathscr B(\mathbb R_+):  \eqref{eq:C.2} \text{ holds for }\psi \}$.
	From the definition of $(A_t)_{t\geq 0}$ and the fact that it almost surely has non-decreasing sample path, we can verify $\mathscr K \subset \mathscr H$ where
	\begin{equation}\label{eq:T.63}
		\begin{multlined}
			\mathscr K := \Big\{ \sum_{k\in \mathbb N} n_k \mathbf 1_{(t_k,t_{k+1}]}:  (n_k)_{k\in \mathbb N} \subset \mathbb R \text{ is bounded}, 
			\\ (t_k)_{k\in \mathbb N} \subset \mathbb R_+\text{ is unbounded and strictly increasing}\Big\}.
		\end{multlined}
	\end{equation}
	One can verify from monotone convergence theorem and \cite[Proposition 17.6]{Kallenberg2002Foundations} that $\mathscr H$ is a monotone vector space in the sense of \cite[p.~364]{Sharpe1988General}. 
	Also observe that $\mathscr K$ is closed under multiplication.
	Therefore using monotone class theorem (\cite[Theorem A0.6]{Sharpe1988General}) we get $\mathrm b \mathscr B(\mathbb R_+) = \sigma(\mathscr K) \subset \mathscr H$. 
\end{proof}
\end{note}
\section{proof of Proposition \ref{thm:W}} \label{sec:W}
	Let us write \eqref{eq:H.3} in the following short form:
\begin{equation}
\label{eq:06_07_2}
	\partial_t w = \partial^2_x w + f^w + \sigma^w \dot W^w + \delta_{\mathrm vt}(x) \dot A_t.
\end{equation}
	The first step of the proof is to remove the drift term $f^w$ using  Dawson's  Girsanov transformation. 
	We summarize this transformation in the following lemma.  
	We refer the reader to \cite[Section 10.2.1]{PratoZabczyk2013Stochastic} for its proof.
	Notice that in this section, since we are dealing with more than one probability measure, we sometimes write ``$\mathrm P\text{-}\iint$'' for the stochastic integral to emphasize the underlying probability measure $\mathrm P$. 

\begin{lemma} \label{thm:W.1}
	Suppose that $W$ is a white noise defined 
	on a filtered probability space $(\Omega, \mathcal G, (\mathcal F_t)_{t\geq 0}, \mathrm P)$.
	Suppose that $h$ is a real-valued predictable random field satisfying 
\[ 
	\mathbb E\Big[\exp\Big\{\frac{1}{2} \iint_0^\infty h_{s,y}^2\mathrm ds\mathrm dy\Big\}\Big]
	< \infty. 
\]
	Then under the probability measure $\mathrm Q$ given by
\[
	\mathrm d\mathrm Q
	:= \exp\Big\{\iint_0^\infty h_{s,y}W(\mathrm ds\mathrm dy) - \frac{1}{2} \iint_0^\infty h_{s,y}^2 \mathrm ds\mathrm dy\Big\} \mathrm d\mathrm P,
\]
	there exists a white noise $\tilde W$ satisfying that for each $g\in \mathscr L_{\mathrm{loc}}^2$ almost surely
	\begin{equation} \label{eq:intQ} 
	\mathrm Q\text{-}\iint_0^t g_{s,y} \tilde W(\mathrm ds\mathrm dy) 
	=\mathrm P\text{-}\iint_0^t g_{s,y}W(\mathrm ds\mathrm dy)  - \iint_0^t h_{s,y} g_{s,y} \mathrm ds\mathrm dy.
\end{equation}
\begin{note}
	$\tilde W$ is an $\mathcal F$-adapted \footnote{change the filtration notation here} space-time white noise where for each $t\geq 0$ and each Borel subset $A$ of $\mathbb R$ with finite Lebesgue measure,
\begin{align} \label{eq:WSI.62}
	& \tilde W_t(A):=W_t(A) - \iint_0^t \mathbf 1_{y\in A}h_{s,y}\mathrm ds\mathrm dy.  
\end{align} 
	Furthermore for any real-valued predictable random field $g$ satisfying
\[
	\iint_0^t g_{s,y}^2 \mathrm ds\mathrm dy < \infty, \quad t\geq 0, \quad \text{a.s. }
\] 
	it holds for each $t\geq 0$ that,  almost surely
\begin{align} \label{eq:WSI.61}
	\iint_0^t g_{s,y} \tilde W(\mathrm ds,\mathrm dy) 
	=\iint_0^t g_{s,y}W(\mathrm ds\mathrm dy)  - \iint_0^t h_{s,y} g_{s,y} \mathrm ds\mathrm dy.
\end{align}
\end{note}
\end{lemma}

\begin{remark*}
	Let $\mathrm Q$ and $\mathrm P$ be the probability measure in Lemma \ref{thm:W.1}. 
	One can verify that $\mathrm Q$ and $\mathrm P$ are mutually absolute continuous. 
	In other word, $A\subset \Omega$ is a $\mathrm Q$-null set if and only if $A$ is a 
	$\mathrm P$-null set.
	Therefore, the filtered probability space $(\Omega, \mathcal G, (\mathcal F_t)_{t\geq 0}, \mathrm Q)$ also satisfies the usual hypotheses; 
	and there is no need to distinguish between ``$\mathrm P$-a.s.'' and ``$\mathrm Q$-a.s.''.  
\end{remark*}

	Later in the proof of Proposition \ref{thm:W}, we will construct a new probability measure $\mathrm Q$, using Lemma \ref{thm:W.1}, under which $w$ will satisfy
\begin{equation}
\label{eq:06_07_1}
	\partial_t w = \partial^2_x w  + \sigma^w \dot {\tilde W}^w+ \delta_{\mathrm vt}(x) \dot A_t, \quad t\in [0,T], x\in \mathbb R
\end{equation}
	where $\tilde W^w$ is a white noise under $\mathrm Q$.
	In order to study the support of $w$ under this new probability, we will need the following   proposition. 
	In what follows, we say that a random measure $\mu$ on a Polish space $S$ has finite mean if its mean measure 
	$
	(\mathbb E\mu) (\cdot):= \mathbb E[\mu(\cdot)]
	$
	is a finite measure on $S$. 
	For more on random measures see~ \cite{Kallenberg2017Random}. 
\begin{proposition}\label{thm:S}
	Let $\tilde T>0$ be arbitrary.
	Suppose that $\tilde w$ is an adapted non-negative continuous random field, defined on a filtered probability space $(\Omega, \mathcal G,  (\mathcal F_t)_{t\geq 0},  \mathrm Q)$, such that $(\tilde w_{t,\cdot}: t\geq 0)$ is a $\mathcal C_{\mathrm {tem}}$-valued continuous process, and  for each $t\in [0,\tilde T]$ and $\phi \in C_\mathrm c^\infty ([0,\tilde T]\times \mathbb R)$, 
	\begin{equation}\label{eq:W.1}
			\int \phi_{t,x}\tilde w_{t,x} \mathrm dx 
			= \iint_0^t {\tilde w}_{s,y}(\partial_s \phi_{s,y} + \partial_y^2 \phi_{s,y}) \mathrm ds\mathrm dy  + \iint_0^t \phi_{s,y}  \big( \tilde \sigma_{s,y} W(\mathrm ds\mathrm dy)+\mu(\mathrm ds\mathrm dy)\big), \;\rm a.s. 
\end{equation}
	Here  $\tilde \sigma$ is a predictable random field, $W$ is a white noise, and $\mu$ is a random measure on   $[0,\tilde T] \times\mathbb R$  with finite mean. 
	Suppose that 
	there exist 
	deterministic $\tilde \vartheta\geq \vartheta > 0$ satisfying that almost surely $ \tilde \vartheta \sqrt{{\tilde w}}\geq \tilde \sigma  \geq \vartheta \sqrt{{\tilde w}}$ on $\mathbb R_+ \times \mathbb R$.
	Then for each $-\infty \leq a< b \leq \infty$ it holds that
	\begin{align}
		\mathrm Q\Big( \int_0^{\tilde T} \mathrm ds\int_{[a,b]^c}  {\tilde w}_{s,y}\mathrm dy > 0\Big) 
		\leq \mathbb E^\mathrm Q \Big[\iint_0^{\tilde T} (\zeta_{\tilde T-s,b-y}^{\vartheta}+ \zeta^\vartheta_{\tilde T-s,y-a}) \mu (\mathrm ds\mathrm dy) \Big]
	\end{align}
	where
	\begin{equation} \label{eq:W.15}
		\zeta_{s,y}^{\vartheta} := 
		\begin{cases}
			0, & s\geq 0, y =\infty;
			\\ \frac{2^8 \sqrt{s}}{\vartheta^2 y^3} e^{- \frac{y^2}{2^4s}}, & s\geq 0,y>0;
			\\ \infty, & s\geq 0, y\leq 0.
		\end{cases}
	\end{equation}
\end{proposition}
	The proof of Proposition \ref{thm:S} will be given in Section \ref{sec:S}.
	In order to control the support of $w$ using the above proposition, we will investigate the expectation of $A_t$  
	under the new  
	probability measure  $\mathrm Q$ which is absolutely continuous with respect to the original probability measure.
	Recall  that $A_t$ is given in
	Lemma~\ref{thm:C.1}(2) and  can be considered as the amount of 
	mass
	of $v$ killed at the line $\{(s,y) \in \mathbb R_+\times \mathbb R: y = \mathrm vs, s\leq t\}$. 
	We will show that  
	  under the new probability $\mathrm Q$,  $v$ is still a weak solution to the SPDE \eqref{eq:H.1}, 
	 and, in fact,  for any such weak solution, we can derive  the  upper bound on the expectation of $A_t$ 
	using the following lemma. 
	  
\begin{lemma} \label{thm:W.3}
	Suppose that
$v$ is a $\Cunit$-valued weak solution to the SPDE \eqref{eq:H.1} defined on 
a filtered probability space 
	$(\Omega, \mathcal G, (\mathcal F_t)_{t\geq 0}, \mathrm Q)$ with $v_{0,\cdot} =   \tilde F $. 
	Let $(A_t)_{t\geq 0}$ be given as in
	Lemma~\ref{thm:C.1} (2).
	Then, 
	\[
		\mathbb E^\mathrm Q[A_t - A_r] \leq \varepsilon (t-r), \quad 0\leq r\leq t<\infty.
	\]
\end{lemma}

\begin{proof}
	\emph{Step 1.} 
	Let $\varrho$ be given as in \eqref{eq:H.04}.
	Note that from Lemma \ref{thm:H.05}, $\varrho$ is a solution to PDE \eqref{eq:H.06}. 
	We define 
	\[
	A^\varrho_t
	:= \iint_0^t \Pi_{s,y}[\rho \leq t] M^\varrho(\mathrm ds \mathrm dy), \quad t\geq 0,
	\] 
	the killing process of $\varrho$ at its boundary, where $M^\varrho(\mathrm ds\mathrm dy) := \varrho_{0,y} \delta_{0}(\mathrm ds)\mathrm dy+ \bar f (\varrho_{s,y})\mathrm ds \mathrm dy$.
	Similar to Lemma \ref{thm:C.1}, we can verify that $t\mapsto A^\varrho_t$ is a real-valued non-decreasing continuous function on $\mathbb R_+$, and for each $t\geq 0$ and $\phi \in \cC_\mathrm c^{1,2}(\mathbb R_+\times \mathbb R)$, it holds that
	\begin{equation}\label{eq:E.04}
		\begin{multlined}
			\int\phi_{t,x} \varrho_{t,x} \mathrm dx
			=  \iint_0^t \phi_{s,y} M^\varrho(\mathrm ds\mathrm dy) + \iint_0^t \varrho_{s,y}(\partial_s \phi_{s,y} + \partial^2_y\phi_{s,y}) \mathrm ds \mathrm dy   - \int_0^t  \phi_{s,\mathrm vs} \mathrm dA^\varrho_s.
		\end{multlined}
	\end{equation}
	
	\emph{Step 2.} We show that $A^\varrho_t =\varepsilon t$ for each $t\in \mathbb R_+$.
	To do this, we use an argument similar to the one we used for \eqref{eq:W.61}, and obtain from \eqref{eq:E.04} that
	\begin{align}
		&A^{\varrho}_t = \lim_{m\to \infty} \int_0^t  \mathrm ds \int_{\mathrm vs-\frac{1}{m}}^{\mathrm vs} \varrho_{s,y}
		\cdot	3\big(1+(y-\mathrm vs)m\big)\Big(-\mathrm vm \big(1+(y-\mathrm vs)m\big) + 2m^2\Big) \mathrm dy. 
	\end{align}
	Now we can verify from bounded convergence theorem that
	\begin{align}
		& A^\varrho_t = \lim_{m\to \infty} 
		\int_0^t \mathrm ds \int_{- 1}^0 F(u/m) \cdot 3(1+u)\big(-\mathrm v(1+u)+2m\big) \mathrm du 
		\\&  
		= \int_0^t\mathrm ds \int_{- 1}^0  F'(0-) \cdot 6 (1+u)u \mathrm du 
		= \varepsilon t.
	\end{align}
	
	For the following Steps 3-5, we fix an arbitrary $t\in \mathbb R_+$ and $x\in (-\infty, \mathrm vt]$.
	
	\emph{Step 3.}
	It holds that 
	$\mathbb E^\mathrm Q[v_{t,x}] \leq \mathbb E^\mathrm Q[\mathrm I]$ where
	\[
	\mathrm I
	:= \iint_0^t G^{\mathrm v}_{s,y;t,x} \big( v_{0,y} \delta_0(\mathrm ds) \mathrm dy + f(v_{s,y})\mathrm ds\mathrm dy\big).
	\]
	In fact note that almost surely $v_{t,x} = \mathrm{I} + \mathrm{II}_t$ where
	\[
	\mathrm {II}_u := \epsilon \iint_0^u G^{\mathrm v}_{s,y;t,x} \sigma(v_{s,y})W^v(\mathrm ds\mathrm dy), \quad u\geq 0
	\]
	is a local martingale.
	Therefore, we can choose a sequence of stopping time $(\rho_n)_{n\in \mathbb N}$ so that for each $n\in \mathbb N$, $(\mathrm {II}_{u\wedge \rho_n})_{u\geq 0}$ is a martingale; and almost surely $\rho_n \uparrow \infty$ when $n\uparrow \infty$.
	Now from the fact that $v_{t,x}$ is non-negative, we can verify from Fatou's lemma that  
	$
	\mathbb E^\mathrm Q[v_{t,x}] \leq \liminf_{n\to \infty} \mathbb E^\mathrm Q[\mathrm I + \mathrm {II}_{t\wedge \rho_n}] = \mathbb E^\mathrm Q[\mathrm I].
	$
	
	\emph{Step 4.} 
	We show that $\mathbb E^\mathrm Q[v_{t,x}] \leq \tilde v_{t,x}$ where
	\begin{align}
		&\tilde v_{t,x}: = \iint_0^t G^{\mathrm v}_{s,y;t,x} \big( v_{0,y} \delta_0(\mathrm ds) \mathrm dy +
		\bar f(\mathbb E^\mathrm Q[v_{s,y}])\mathrm ds\mathrm dy\big).
	\end{align}
	In fact, noticing from Lemma \ref{thm:H.05} that $\bar f \geq f$, we have  
	\[
	\tilde v_{t,x}= 
	\mathbb E^\mathrm Q\Big[\iint_0^t G^{\mathrm v}_{s,y;t,x} \big( v_{0,y} \delta_0(\mathrm ds) \mathrm dy + \bar f(v_{s,y})\mathrm ds\mathrm dy\big)\Big]\geq \mathbb E^\mathrm Q[\mathrm I]
	\]
	where $\mathrm I$ is given as in Step 3. 
	Now the desired result in this step follows from Step 3.
	
	\emph{Step 5.}
	It holds that $\mathbb E^\mathrm Q[v_{t,x}] \leq \varrho_{t,x}$.
	To see this, we first observe from Lemma \ref{thm:H.05} that $\varrho$ admits the following mild form
	\[
	\varrho_{t,x} = \iint_0^t G_{s,y;t,x}^{(\mathrm v)} \big(\varrho_{0,y} \delta_0(\mathrm ds) \mathrm dy + (\alpha \varrho_{s,y} + \beta) \mathrm ds\mathrm dy\big)
	\]
	where $\alpha := \theta(1-\theta)\mathrm v^2$ and $\beta := (1-\theta)\mathrm v \varepsilon$.
	Using Feynman-Kac formula (c.f. \cite[Lemma 1.5. on p.~1211]{Dynkin1993Superprocesses}) we have that 
	\[
	\varrho_{t,x} = e^{\alpha t} \iint_0^t G^{(\mathrm v)}_{s,y;t,x} e^{-\alpha s}\big(\varrho_{0,y} \delta_0(\mathrm ds) \mathrm dy +\beta \mathrm ds\mathrm dy\big).
	\]
	Similarly, using Feynman-Kac formula for $\widetilde v$, we get
	\[
	\tilde v_{t,x}
	:= e^{\alpha t}\iint_0^t G^{(\mathrm v)}_{s,y;t,x} e^{-\alpha s} \big( v_{0,y} \delta_0(\mathrm ds) \mathrm dy + ( -\alpha \tilde v_{s,y}+ \alpha \mathbb E^\mathrm Q[v_{s,y}] + \beta)\mathrm ds\mathrm dy\big).
	\]
	Observing from the above two equations and Step 4, we have that $\widetilde v_{t,x} \leq \varrho_{t,x}$.
	Using Step 4 again, we get the desired result in this step.
	
	\emph{Step 6.} We show that for any $0\leq r< t<\infty$, it holds that
	$\mathbb E^\mathrm Q[A_t-A_r]\leq A^\varrho_t-A^\varrho_r$.
	To do this, note that almost surely $0\leq A_t-A_r = \mathrm {III} + \mathrm {IV}_t$ where
	\begin{align}
		&\mathrm {III} := \iint_r^t \Pi_{s,y}[\rho \leq t] \big(v_{0,y} \delta_0(\mathrm ds) \mathrm dy + f(v_{s,y})\mathrm ds\mathrm dy\big);
		\\&\mathrm {IV}_u :=  \iint_r^{u\wedge t} \Pi_{s,y}[\rho \leq t] \sigma(v_{s,y})W^v(\mathrm ds\mathrm dy), \quad u\geq r.
	\end{align}
	Since $(\mathrm {IV}_u)_{u\geq r}$ is a local martingale, we can choose a sequence of stopping time $(\tilde \rho_n)_{n\in \mathbb N}$ so that for each $n\in \mathbb N$, $(\mathrm {IV}_{u\wedge \tilde \rho_n})_{u\geq r}$ is a martingale; and almost surely $\tilde \rho_n \uparrow \infty$ when $n\uparrow \infty$.
	From Fatou's Lemma we have $ \mathbb E^\mathrm Q[A_t - A_r] \leq \liminf_{n\to \infty} \mathbb E^\mathrm Q[\mathrm {III}+\mathrm {IV}_{t\wedge \tilde \rho_n}] = \mathbb E^\mathrm Q[\mathrm {III}].$
	From Lemma \ref{thm:H.05} that $\bar f \geq f$, Steps 1 and 5, we can verify that 
	\begin{align}
		\mathbb E^\mathrm Q[\mathrm {III}]
		\leq \iint_r^t \Pi_{s,y}[\rho \leq t] \big(\varrho_{0,y} \delta_0(\mathrm ds) \mathrm dy + \bar f(\varrho_{s,y})\mathrm ds\mathrm dy\big) 
		= A^\varrho_t - A^\varrho_r.
	\end{align}	
	The desired result in this step then follows.
	
	\emph{Final Step.}
	The desired result in this lemma follows from Steps 2 and 6.
\end{proof}

	As for showing that  $v$ is a weak solution to the SPDE \eqref{eq:H.1},
	 under  the new probability $\mathrm Q$, this will be done  with the help of the following lemma whose proof is standard and therefore is omitted (one can replicate the analogous classical proof for Brownian motions).
\begin{lemma} \label{thm:W.2}
	Suppose the conditions of Lemma \ref{thm:W.1} hold. 
		Further	suppose that there exists another $(\mathcal F_t)_{t\geq 0}$-adapted space-time white noise $W'$ which, under the probability $\mathrm P$, is independent of $W$.
	Then $W'$ is still a white noise under the probability $\mathrm Q$.  
	Moreover, for each $t\geq 0$ and $g\in \mathscr L_{\mathrm{loc}}^2$, it holds that
	\begin{equation}
		\mathrm Q\text{-}\iint_0^t g_{s,y} W'(\mathrm ds\mathrm dy)
		= \mathrm P\text{-}\iint_0^t g_{s,y} W'(\mathrm ds\mathrm dy)
	\quad \text{a.s.}
	\end{equation}
\end{lemma}
	
	We are now ready to give the proof of Proposition \ref{thm:W}.

\begin{proof}[Proof of Proposition \ref{thm:W}]
\emph{Step 1.} 
	Noticing from \eqref{eq:U.4} that $\nu \varepsilon L \leq 1/4$, and the fact that for any $x,y \in [0,1/4]$, 
\[ 
	|f(y) - f(x)|
	= |y^p - x^p|
	= \Big|\int_x^y p z^{p-1}\mathrm dz \Big| 
	\leq \int_0^{|y-x|} p z^{p-1}\mathrm dz 
	= | y - x |^p,
\]
	we have almost surely for each $(t,x)\in \mathbb R_+\times \mathbb R$,
\begin{align}
	& h_{t,x}
	:= \frac{f^w_{t,x}}{\epsilon \sigma^w_{t,x}} \mathbf 1_{\sigma^w_{t,x}>0, t\leq T}
	\\& \leq \frac{|f(v_{t,x}+w_{t,x}) - f(v_{t,x})|\mathbf 1_{x \in (-L, \mathrm v T + L), v_{t,x} + w_{t,x} \leq \nu \varepsilon L} }{\epsilon \sqrt{w_{t,x} / 2 }} \mathbf 1_{w_{t,x}>0, t\leq T}
	\\ &\leq \sqrt{2}\epsilon^{-1}w_{t,x}^{p-\frac{1}{2}} \mathbf 1_{x\in (-L, \mathrm vT + L), w_{t,x} \leq \nu \varepsilon L, t\leq T}
	\leq \sqrt{2} \epsilon^{-1} (\nu \varepsilon L)^{p - \frac{1}{2}} \mathbf 1_{x\in (-L, \mathrm vT + L), t\leq T}. 
\end{align}
	
\emph{Step 2.} 
We construct a probability measure $\mathrm Q$ on $(\Omega, \mathcal G)$ such that 
\begin{align} \label{eq:W.4}
	\mathrm d\mathrm Q 
	= \exp\Big\{-\iint_0^\infty h_{s,y} W^w(\mathrm ds\mathrm dy) - \frac{1}{2} \iint_0^\infty h^2_{s,y} \mathrm ds\mathrm dy\Big\} \mathrm d\mathrm P.
\end{align}
We can do this thanks to Step 1 that gives
\begin{equation} 
	\mathbb E^{\mathrm P}\Big[\exp\Big\{\frac{1}{2} \iint_0^\infty h_{s,y}^2\mathrm ds\mathrm dy\Big\}\Big]
	<\infty.
\end{equation}
	
\emph{Step 3.} 
	We verify 
that for any $\phi\in C^\infty_{\rm c}(\mathbb R_+\times \mathbb R)$ and $t\in [0,T]$, almost surely 
\begin{equation}
\begin{multlined}
	\int \phi_{t,x}w_{t,x} \mathrm dx 
= \iint_0^t 
w_{s,y}
(\partial_s \phi_{s,y} + \partial_y^2 \phi_{s,y}) \mathrm ds\mathrm dy  + \int_0^t \phi_{s,\mathrm vs}\mathrm dA_s + {}
\\ \mathrm Q\text{-}\iint_0^t \phi_{s,y}\epsilon \sigma^w_{s,y} \tilde W^w(\mathrm ds\mathrm dy)
\end{multlined}
\end{equation}
where $\tilde W^w$ is a white noise on the filtered probability space $(\Omega, \mathcal G, (\mathcal F_t)_{t\geq 0}, \mathrm Q)$ given as in Lemma \ref{thm:W.1} so that 
\[
	\mathrm Q\text{-}\iint_0^\cdot g_{s,y} \tilde W^w(\mathrm ds\mathrm dy) 
	=\mathrm P\text{-}\iint_0^\cdot g_{s,y}W^w(\mathrm ds\mathrm dy) + \iint_0^\cdot h_{s,y} g_{s,y} \mathrm ds\mathrm dy,
	\quad \text{a.s.}
	\quad g\in \mathscr L_\mathrm{loc}^2.
\]
\begin{note}
for each $t\geq 0$ and each Borel subset $A$ of $\mathbb R$ with finite Lebesgue measure,
\begin{align} 
	& \tilde W^w_t(A):=W^w_t(A) + \iint_0^t \mathbf 1_{y\in A}h_{s,y}\mathrm ds\mathrm dy.  
\end{align} 
In fact, from Lemma \ref{thm:W.1} we know that $\tilde W^w$ is indeed an $\mathcal F$-white noise \footnote{change the filtration notation here} under probability $\mathrm Q$.
	Also from Lemma \ref{thm:W.1} we have for each $\phi\in C^\infty_{\rm c}(\mathbb R_+\times \mathbb R)$ and $t\in [0,T]$ almost surely
\begin{align}
	&\epsilon \iint_0^t \phi_{s,x}\sigma^w_{s,x} \tilde W^w(\mathrm ds\mathrm dx)
	= \epsilon \iint_0^t \phi_{s,x} \sigma^w_{s,x} W^w(\mathrm ds\mathrm dx) + \epsilon \iint_0^t \phi_{s,x} \sigma^w_{s,x} h_{s,x}\mathrm ds\mathrm dx
	\\&	= \epsilon \iint_0^t \phi_{s,x} \sigma^w_{s,x} W^w(\mathrm ds\mathrm dx) + \iint_0^t \phi_{s,x} f^w_{s,x} \mathbf 1_{\sigma^w_{s,x}>0, s\leq  T}\mathrm ds\mathrm dx
	\\& = \epsilon \iint_0^t \phi_{s,x} \sigma^w_{s,x} W^w(\mathrm ds\mathrm dx) + \iint_0^t \phi_{s,x} f^w_{s,x} \mathrm ds\mathrm dx.
\end{align}
	From this and \eqref{eq:H.3} we get the desired result in this step.
\end{note}

	\emph{Step 4.} We will show that for each $t\geq 0$ and non-negative continuous function $\psi$ on $\mathbb R_+$ the following holds:
\begin{equation} \label{eq:W.5}
	\mathbb E^\mathrm Q\Big[\int_0^t \psi_s \mathrm dA_s\Big]
	\leq \varepsilon \int_0^t \psi_s \mathrm ds. 
\end{equation}
To see this, we verify from
Lemma~\ref{thm:W.2} that with respect to the filtered probability space $(\Omega, \mathcal G, (\mathcal F_t)_{t\geq 0},\mathrm Q)$:
\begin{itemize}
	\item $W^v$ is still a white noise;
	\item
		$v$ is still a weak solution to the SPDE \eqref{eq:H.1} with $v_{0,\cdot} =  \tilde F $;
	\item $(A_t)_{t\geq 0}$ is still the killing process of $v$; see
	Lemma~\ref{thm:C.1} (2).
\end{itemize}
	Therefore from
	Lemma~\ref{thm:W.3}, we have 
	$\mathbb E^\mathrm Q[A_t-A_r]\leq \varepsilon(t-r)$ 
	for each $0\leq r\leq t<\infty$. 
	From this we can verify that \eqref{eq:W.5} holds for each $t\geq 0$ and each non-negative $\psi\in\mathscr K$ where 
		\begin{equation}\label{eq:T.63}
		\begin{multlined}
			\mathscr K := \Big\{ \sum_{k\in \mathbb N} n_k \mathbf 1_{(t_k,t_{k+1}]}:  (n_k)_{k\in \mathbb N} \subset \mathbb R \text{ is bounded}, 
			\\ (t_k)_{k\in \mathbb N} \subset \mathbb R_+\text{ is unbounded and strictly increasing}\Big\}.
		\end{multlined}
	\end{equation}
	Now the desired result in this step follows from monotone convergence theorem and the fact that for any non-negative continuous function $\psi$ on $\mathbb R_+$ there exists a non-negative sequence $(\psi^{(n)})_{n\in  \mathbb N}\subset \mathscr K$ such that $\psi^{(n)}\uparrow \psi$ pointwise as $n\uparrow \infty$.
	
	\emph{Step 5.} We will show that $\mathrm Q(\tau_1 < T) \leq 2^{14}\gamma$.
	Note that almost surely
\begin{align}
	\epsilon \sqrt{w_{t,x}}
	\geq \epsilon \sigma^w_{t,x}
	\geq \frac{\epsilon}{\sqrt{2}} \sqrt{w_{t,x}},
	\quad t\geq 0, x\in \mathbb R.
\end{align}
	So from Step 3, Step 4, and Proposition \ref{thm:S}, we get that 
\begin{align}
		& \mathrm Q(\tau_1 < T) 
		\leq
		\mathbb E^\mathrm Q
		\Big[ \iint_0^T \big(\zeta_{T - s,(\mathrm vT+L)-x}^{\epsilon/\sqrt{2}} + \zeta^{\epsilon/\sqrt{2}}_{T -s, x - (-L)}\big)  \delta_{\mathrm vs}(\mathrm dx) \mathrm dA_s\Big]
		\\ &\leq \varepsilon  \int_0^T \big(\zeta_{T - s,(\mathrm vT+L) -\mathrm vs}^{\epsilon/\sqrt{2}} + \zeta_{T -s, \mathrm vs -  (-L)}^{\epsilon/\sqrt{2}}\big)  \mathrm ds
		\leq 2\varepsilon \int_0^T \zeta_{T -s, L}^{\epsilon/\sqrt{2}} \mathrm ds.
\end{align}
	Here in the last inequality, we used the fact that for any given $s\geq 0$,  the map $x\mapsto \zeta^{\epsilon/\sqrt{2}}_{s,x}$ is non-increasing on $\mathbb R$.
	Now we have
\begin{align}	
		&\mathrm Q(\tau_1 < T) 
		\leq  2\varepsilon \int_0^T \frac{2^9 {(T-s)}^{1/2}}{\epsilon^2  L^3} e^{- \frac{L^2}{2^4(T-s)}} \mathrm ds 
		=   \frac{2^{10}\gamma}{L^3}  \int_0^T s^{1/2}e^{- \frac{L^2}{8s}} \mathrm ds 
	      \\& \leq  \frac{2^{10}\gamma T^{5/2}}{L^3} \frac{2^4}{L^2}\int_{s=0}^{s=T} e^{- \frac{L^2}{2^4s}} \mathrm d(-\frac{L^2}{2^4s}) 
		\leq  \frac{2^{14}\gamma T^{5/2}}{L^5}  e^{- \frac{L^2}{2^4T}} 
		\leq  2^{14}\gamma. 
\end{align}	
	
\emph{Final Step.} 
	Noticing that $\tilde W^w$ is a white noise under $\mathrm Q$, we can verify that for each $q \in [1,\infty)$ the expectation of
\begin{equation}
	m^{(q)}
	:=\exp\Big\{q\iint_0^\infty h_{s,y} \tilde W^w (\mathrm ds\mathrm dy)- \frac{q^2}{2}\iint_0^\infty h_{s,y}^2 \mathrm ds\mathrm dy\Big\}
\end{equation}
	under $\mathrm Q$ equals to $1$. 
Also note from \eqref{eq:W.4} and Lemma \ref{thm:W.1} we have that
\begin{align}
	&\frac{\mathrm d \mathrm P}{\mathrm d \mathrm Q}
	= \exp\Big\{\iint_0^\infty h_{s,y}W^w(\mathrm ds\mathrm dy)+ \frac{1}{2} \iint_0^\infty h_{s,y}^2 \mathrm ds\mathrm dy\Big\}
	\\& = \exp\Big\{\iint_0^\infty h_{s,y}\tilde W^w(\mathrm ds\mathrm dy)- \frac{1}{2} \iint_0^\infty h_{s,y}^2 \mathrm ds\mathrm dy\Big\}
	= m^{(1)}.
\end{align}
	Now we can verify using Cauchy–Schwartz inequality that
\begin{align} 
	& \mathrm P(\tau_1 < T) 
	= \mathbb E^\mathrm Q[\mathbf 1_{\{\tau_1 < T\}} m^{(1)}] 
	\leq \mathrm Q(\tau_1 < T)^{\frac{1}{2}} \mathbb E^\mathrm Q \big[\big(m^{(1)}\big)^2 \big]^{\frac{1}{2}}
	\\&= \mathrm Q(\tau_1 < T)^{\frac{1}{2}} \mathbb E^\mathrm Q\Big[ m^{(2)} \exp\Big\{\iint_0^\infty h_{s,y}^2 \mathrm ds\mathrm dy\Big\}\Big]^{\frac{1}{2}}.
\end{align}
Finally, using \eqref{eq:U.1}, \eqref{eq:U.2}, \eqref{eq:U.3} and Steps 1, 5 we have that 
\begin{align}
	&\mathrm P(\tau_1 < T)
	\leq \mathrm Q(\tau_1 < T)^{\frac{1}{2}} \exp\{ (\mathrm vT+2L)T \epsilon^{-2} (\nu \varepsilon L)^{2p-1}\}
      \\&= \mathrm Q(\tau_1 < T)^{\frac{1}{2}} \exp\{ 3\gamma \kappa^{2p-2} \nu^{2p-1} \}
\leq 2^{7} \sqrt{\gamma}  \exp\{ 3\gamma \nu\}
	\leq 1/8.
	\qedhere
\end{align}
\end{proof}
\begin{note}
	We now give the proof of Lemma \ref{thm:W.3}.
\begin{proof}[Proof of Lemma \ref{thm:W.3}]
	\emph{Step 1.} 
	Let $\varrho$ be given as in \eqref{eq:H.04}.
	Note that from Lemma \ref{thm:H.05}, $\varrho$ is a solution to PDE \eqref{eq:H.06}. 
	We define 
\[
	A^\varrho_t
	:= \iint_0^t \Pi_{s,y}[\rho \leq t] M^\varrho(\mathrm ds \mathrm dy), \quad t\geq 0,
\] 
	the killing process of $\varrho$ at its boundary, where $M^\varrho(\mathrm ds\mathrm dy) := \varrho_{0,y} \delta_{0}(\mathrm ds)\mathrm dy+ \bar f (\varrho_{s,y})\mathrm ds \mathrm dy$.
	Similar to Lemma \ref{thm:C.1}, we can verify that $t\mapsto A^\varrho_t$ is a real-valued non-decreasing continuous function on $\mathbb R_+$; and for each $t\geq 0$ and $\phi \in C_\mathrm c^{1,2}(\mathbb R_+\times \mathbb R)$, it holds that
\begin{equation}\label{eq:E.04}
\begin{multlined}
	\int\phi_{t,x} \varrho_{t,x} \mathrm dx
	=  \iint_0^t \phi_{s,y} M^\varrho(\mathrm ds\mathrm dy) + \iint_0^t \varrho_{s,y}(\partial_s \phi_{s,y} + \partial^2_y\phi_{s,y}) \mathrm ds \mathrm dy   - \int_0^t  \phi_{s,\mathrm vs} \mathrm dA^\varrho_s.
\end{multlined}
\end{equation}
	
	\emph{Step 2.} We show that $A^\varrho_t =\varepsilon t$ for each $t\in \mathbb R_+$.
	To do this, we use an argument similar to the one we used for \eqref{eq:W.61}, and obtain from \eqref{eq:E.04} that
\begin{align}
		&A_t = \lim_{m\to \infty} \int_0^t  \mathrm ds \int_{\mathrm vs-\frac{1}{m}}^{\mathrm vs} \varrho_{s,y}\Big(-2\mathrm vm \big(1+(y-\mathrm vs)m\big) + 2m^2\Big) \mathrm dy.
\end{align}
	Now we can verify from bounded convergence theorem that
\begin{align}
	& A^\varrho_t = \lim_{m\to \infty} \int_0^t \mathrm ds \int_{- \frac{1}{m}}^0 \big(2m^2 - 2\mathrm v m(1+ ym) \big) F(y) \mathrm dy
	\\&  = \lim_{m\to \infty} \int_0^t \mathrm ds \int_{- 1}^0 \big(2m - 2\mathrm v \big(1+ z\big) \big) F\Big(\frac{z}{m}\Big) \mathrm dz 
	= \int_0^t\mathrm ds \int_{- 1}^0 2z F'(0-) \mathrm dz 
	= \varepsilon t.
\end{align}

For the following Steps 3-5, we fix an arbitrary $t\in \mathbb R_+$ and $x\in (-\infty, \mathrm vt]$.

\emph{Step 3.}
	It holds that 
	$\mathbb E^\mathrm Q[v_{t,x}] \leq \mathbb E^\mathrm Q[\mathrm I]$ where
	\[
	\mathrm I
	:= \iint_0^t G^{\mathrm v}_{s,y;t,x} \big( v_{0,y} \delta_0(\mathrm ds) \mathrm dy + f(v_{s,y})\mathrm ds\mathrm dy\big).
\]
	In fact note that almost surely $v_{t,x} = \mathrm{I} + \mathrm{II}_t$ where
\[
	\mathrm {II}_u := \epsilon \iint_0^u G^{\mathrm v}_{s,y;t,x} \sigma(v_{s,y})W^v(\mathrm ds\mathrm dy), \quad u\geq 0
\]
	is a local martingale.
	Therefore, we can choose a sequence of stopping time $(\rho_n)_{n\in \mathbb N}$ so that for each $n\in \mathbb N$, $(\mathrm {II}_{u\wedge \rho_n})_{u\geq 0}$ is a martingale; and almost surely $\rho_n \uparrow \infty$ when $n\uparrow \infty$.
	Now from the fact that $v_{t,x}$ is non-negative, we can verify from Fatou's lemma that  
$
	\mathbb E^\mathrm Q[v_{t,x}] \leq \liminf_{n\to \infty} \mathbb E^\mathrm Q[\mathrm I + \mathrm {II}_{t\wedge \rho_n}] = \mathbb E^\mathrm Q[\mathrm I].
$
	
	\emph{Step 4.} 
	It holds that $\mathbb E^\mathrm Q[v_{t,x}] \leq \tilde v_{t,x}$ where
	\begin{align}
		&\tilde v_{t,x}: = \iint_0^t G^{\mathrm v}_{s,y;t,x} \big( v_{0,y} \delta_0(\mathrm ds) \mathrm dy +
		\bar f(\mathbb E^\mathrm Q[v_{s,y}])\mathrm ds\mathrm dy\big).
	\end{align}
	In fact, noticing from Lemma \ref{thm:H.05} that $\bar f \geq f$, we have  
\[
	\tilde v_{t,x}= 
	\mathbb E^\mathrm Q\Big[\iint_0^t G^{\mathrm v}_{s,y;t,x} \big( v_{0,y} \delta_0(\mathrm ds) \mathrm dy + \bar f(v_{s,y})\mathrm ds\mathrm dy\big)\Big]\geq \mathbb E^\mathrm Q[\mathrm I]
\]
	where $\mathrm I$ is given as in Step 3. 
	Now the desired result in this step follows from Step 3.
	
	\emph{Step 5.}
	It holds that $\mathbb E^\mathrm Q[v_{t,x}] \leq \varrho_{t,x}$.
	To see this, we first observe from Lemma \ref{thm:H.05} that $\varrho$ admits the following mild form
\[
	\varrho_{t,x} = \iint_0^t G_{s,y;t,x}^{(\mathrm v)} \big(\varrho_{0,y} \delta_0(\mathrm ds) \mathrm dy + (\alpha \varrho_{s,y} + \beta) \mathrm ds\mathrm dy\big)
\]
	where $\alpha := \theta(1-\theta)\mathrm v^2$ and $\beta := (1-\theta)\mathrm v \varepsilon$.
	Using Feynman-Kac formula (c.f. \cite[Lemma 1.5. on p.~1211]{Dynkin1993Superprocesses}) we have that 
\[
	\varrho_{t,x} = e^{\alpha t} \iint_0^t G^{(\mathrm v)}_{s,y;t,x} e^{-\alpha s}\big(\varrho_{0,y} \delta_0(\mathrm ds) \mathrm dy +\beta \mathrm ds\mathrm dy\big).
\]
	Similarly, using Feynman-Kac formula for $\widetilde v$, we get
\[
	\tilde v_{t,x}
	:= e^{\alpha t}\iint_0^t G^{(\mathrm v)}_{s,y;t,x} e^{-\alpha s} \big( v_{0,y} \delta_0(\mathrm ds) \mathrm dy + ( -\alpha \tilde v_{s,y}+ \alpha \mathbb E^\mathrm Q[v_{s,y}] + \beta)\mathrm ds\mathrm dy\big).
\]
	Observing from the above two equations and Step 4, we have that $\widetilde v_{t,x} \leq \varrho_{t,x}$.
	Using Step 4 again, we get the desired result in this step.
	
	\emph{Step 6.} We show that for any $0\leq r< t<\infty$, it holds that
	$\mathbb E^\mathrm Q[A_t-A_r]\leq A^\varrho_t-A^\varrho_r$.
	To do this, note that almost surely $0\leq A_t-A_r = \mathrm {III} + \mathrm {IV}_t$ where
\begin{align}
	&\mathrm {III} := \iint_r^t \Pi_{s,y}[\rho \leq t] \big(v_{0,y} \delta_0(\mathrm ds) \mathrm dy + f(v_{s,y})\mathrm ds\mathrm dy\big);
	\\&\mathrm {IV}_u :=  \iint_r^{u\wedge t} \Pi_{s,y}[\rho \leq t] \sigma(v_{s,y})W^v(\mathrm ds\mathrm dy), \quad u\geq r.
\end{align}
	Since $(\mathrm {IV}_u)_{u\geq r}$ is a local martingale, we can choose a sequence of stopping time $(\tilde \rho_n)_{n\in \mathbb N}$ so that for each $n\in \mathbb N$, $(\mathrm {IV}_{u\wedge \tilde \rho_n})_{u\geq r}$ is a martingale; and almost surely $\tilde \rho_n \uparrow \infty$ when $n\uparrow \infty$.
	From Fatou's Lemma we have $ \mathbb E^\mathrm Q[A_t - A_r] \leq \liminf_{n\to \infty} \mathbb E^\mathrm Q[\mathrm {III}+\mathrm {IV}_{t\wedge \tilde \rho_n}] = \mathbb E^\mathrm Q[\mathrm {III}].$
	From Lemma \ref{thm:H.05} that $\bar f \geq f$, Steps 1 and 5, we can verify that 
\begin{align}
	\mathbb E^\mathrm Q[\mathrm {III}]
	\leq \iint_r^t \Pi_{s,y}[\rho \leq t] \big(\varrho_{0,y} \delta_0(\mathrm ds) \mathrm dy + \bar f(\varrho_{s,y})\mathrm ds\mathrm dy\big) 
	= A^\varrho_t - A^\varrho_r.
\end{align}	
	The desired result in this step then follows.
	
\emph{Final Step.}
	The desired result in this lemma follows from Steps 2 and 6.
\end{proof}
\end{note}

\section{proof of Proposition \ref{thm:S}} \label{sec:S}
	In this section we will give the proof of Proposition \ref{thm:S} following a strategy similar to that used in \cite[Proof of Proposition 3.2]{Tribe1995Large}.
	Notice that, in the special case when $\sigma = \sqrt{\tilde w}$, the solution $\tilde w$ to the SPDE \eqref{eq:W.1} can be considered as the density of a super-Brownian motion with space-time immigration $\mu$. 
	Next lemma deals with properties of the solutions to the so-called log-Laplace equations which play very important role 
	in studying properties of superprocesses (see e.g.~\cite{Iscoe1986AWeighted}).
	In the general case when the noise coefficient $\sigma$ is comparable to $\sqrt{\tilde w}$, we can still use this log-Laplace equation to obtain properties of the random field $\tilde w$.  
\begin{lemma}\label{thm:S.1}
	Let $\tilde T>0, \vartheta> 0$ and $\psi\in \cC_0^2(\mathbb R)$ be non-negative.
	There exists a unique non-negative $\phi \in \cC_{\mathrm b}^{1,2}([0,\tilde T] \times \mathbb R)$ such that
	\begin{equation} \label{eq:S.2}
		\begin{cases}
			\partial_t \phi_{t,x} =- \partial_x^2 \phi_{t,x} +  \frac{1}{2}(\vartheta \phi_{t,x})^2  - \psi_{x}, &\quad (t,x)\in [0,\tilde T]\times \mathbb R; 
			\\ \phi_{\tilde T,x} = 0,&\quad x\in \mathbb R.
		\end{cases}
	\end{equation}
	Furthermore, if $b\in \mathbb R$ and $\psi = 0$ on $(-\infty,b]$, then
\begin{equation}
\label{eq:S.2aa}
\phi_{t,x} \leq \zeta^\vartheta_{\tilde T-t,b-x}, \quad (t,x)\in [0,\tilde T]\times \mathbb R
\end{equation}
		where
	\begin{equation} \label{eq:W.15}
		\zeta_{s,y}^{\vartheta} =
		\begin{cases}
			0, & s\geq 0, y =\infty;
			\\ \frac{2^8 \sqrt{s}}{\vartheta^2 y^3} e^{- \frac{y^2}{2^4s}}, & s\geq 0,y>0;
			\\ \infty, & s\geq 0, y\leq 0.
		\end{cases}
	\end{equation}
\end{lemma}

\begin{proof}
	The existence and uniqueness for~\eqref{eq:S.2} is given in~\cite{Iscoe1986AWeighted}.
	\begin{note}
		\emph{Step 1.} We show the existence.
		Note that the one dimensional Brownian motion $B$ is an $\mathbb R$-valued Feller process with generator $\partial_x^2$, a closed operator on the Banach space $\cC_0(\mathbb R)$ with domain $\cC_0^2(\mathbb R)$; See \cite[Theorem 31.5]{Sato1999Levy} and \cite[Excisece 19.13]{Kallenberg2002Foundations} for example.   
		From $\psi \in \cC_0^2(\mathbb R)$ and \cite[Theorem 3.1]{Iscoe1986AWeighted}, we know that there exists a continuous map $t\mapsto \phi_t$ from $[0,T]$ to $\cC_0(\mathbb R)$ satisfying  
		\begin{itemize}
			\item it is continuous differentiable, i.e. there exists a continuous map $t\mapsto \dot \phi_t$ from $[0,T]$ to $\cC_0(\mathbb R)$ so that 
			$
			\lim_{t'\to t, t'\in [0,T]}\|\frac{\phi_{t'}-\phi_{t}}{t'-t}-\dot \phi_t\|_{\infty} = 0;
			$
			\item for each $t\in [0,T]$, it holds that $\phi_t\in C^2_0(\mathbb R)$ and that $\dot \phi_t = -\partial_x^2\phi_t +\frac{1}{2}(\theta \phi_t)^2 -\psi$.
		\end{itemize}	
		It can be verified that the map $\phi:(t,x)\mapsto \phi_t(x)$ belongs to $\cC_{\mathrm b}^{1,2}([0,T]\times\mathbb R)$ and solves the PDE \eqref{eq:S.2}.
		
		To see $\phi$ is non-negative we need a probabilistic representation of $\phi$ using theory of super Brownian motion.
		Let $M_F(\mathbb R)$ be the space of finite Borel measures on $\mathbb R$ equipped with the weak topology.
		It is known (see \cite[Section 2]{Li2010Measure-valued} for example) that there exists an $M_F(\mathbb R)$-valued Markov process $\{(X_t)_{t\geq 0}; (\mathbb P_{\mu})_{\mu \in M_F(\mathbb R)}\}$ whose transition kernel is characterized by
		\[
		\mathbb E_\mu\Big[\exp\Big\{\int g_x X_t(\mathrm dx)\Big\}\Big]
		= \exp\Big\{-\int V^g_{t,x} \mu(\mathrm dx)\Big\}, 
		\quad t\geq 0, g\in \mathrm b\mathscr B(\mathbb R),
		\]
		where for each $T \geq 0$ and $g\in \mathrm b\mathscr B(\mathbb R)$, $(V^g_{t,x})_{(t,x)\in [0,T]\times \mathbb R}$ is the unique non-negative bounded Borel function on $[0,T]\times \mathbb R$ satisfying that
		\[
		V^g_{t,x}
		= \int G_{0,y;t,x} g_y \mathrm dy -  \iint_0^t G_{s,y;t,x} ( V^g_{s, y})^2 \mathrm ds \mathrm dy, \quad (t,x)\in [0,T]\times \mathbb R.
		\]	
		Such process $X$ is known as the super-Brownian motion. 
		Now, according to \cite[Theorem 3.1]{Iscoe1986AWeighted}, $\phi$ has the representation
		\begin{equation} \label{eq:W.091}
			\frac{1}{2}\theta^2\phi_{s,y} = - \log \mathbb E_{s,\delta_y} \Big[\exp\Big\{- \int_0^{T} \mathrm dt \int \frac{1}{2}\theta^2\psi_x X_t(\mathrm dx)\Big\}\Big], \quad s\geq 0, y\in \mathbb R.
		\end{equation}
		It is clear from this representation that $\phi$ is non-negative.
		
		\emph{Step 2.} We show the uniqueness.
		Let $\phi$ be any non-negative $\cC_{\mathrm b}^{1,2}([0,T]\times \mathbb R)$ solution to \eqref{eq:S.2}.
		Fix $(s,y)\in [0,T]\times \mathbb R$.
		Let $(B_t)_{t\geq s}$ be a one-dimensional Brownian motion with generator $\partial_x^2$ initiated at time $s$ and position $y$ with its probability denoted as $\Pi_{s,y}$.
		Using Ito's formula \cite[Theorem 3.3 and Remark 1 on p.~147]{RevuzYor1999Continuous} we can derive that almost surely,
		\begin{equation}
			\phi_{t, B_t} - \phi_{s,B_s} = \int_s^t \Big( \frac{1}{2} \theta \phi^2_{r,B_r} - \psi_{B_r} \Big) \mathrm dr + \int_s^t \partial_x\phi_{r,x}|_{x=B_r}\mathrm d B_r,
			\quad t\in [s,T].
		\end{equation}
		From the fact that $\phi\in \cC_\mathrm b^{1,2}([0,T]\times \mathbb R)$, the second integral on the right hand side of the above equation is a martingale on $[s,T]$. 
		Using the optional sampling theorem \cite[Corollary 17.8]{Kallenberg2002Foundations}, for each optional time $\tilde \rho\in [s,T]$, defined on the probability space where $B$ is defined, we have
		\begin{equation} \label{eq:W.10}
			\phi_{s,y} = \Pi_{s,y} \Big[\phi_{\tilde \rho,B_{\tilde \rho}} - \int_s^{\tilde \rho} \Big( \frac{1}{2} \theta \phi^2_{r,B_r} - \psi_{B_r} \Big) \mathrm dr \Big].
		\end{equation}
		From \cite[Theorem 5.15]{Li2010Measure-valued} the uniqueness of $\phi$ follows.
		
	\end{note}
 Note that although the  proof of the upper bound~\eqref{eq:S.2aa} is also pretty standard (see e.g. derivation of (5) in the proof of Proposition 3.2 in \cite{Tribe1995Large}, or the relevant steps in 
	the proof of Lemma~2.6 in~\cite{MuellerMytnikRyzhik2019TheSpeed}),  we decided to include it for the sake of completeness.
	
	We give the upper bound for $\phi$ in~\eqref{eq:S.2aa} provided $\psi = 0$ on $(-\infty, b]$ for an arbitrary $b\in \mathbb R$.
	First, using the connection between solutions to \eqref{eq:S.2aa} and super-Brownian motion and  due to \cite[Theorem 1]{Iscoe1988OnTheSupports}
	we can derive 
	\begin{equation}\label{eq:W.11}
		\frac{1}{2}\vartheta^2\phi_{s,y} 
		\leq \frac{9}{(y-b)^2},
		\quad (s,y)\in [0,T]\times (-\infty, b).
	\end{equation}
	Now, let $(B_t)_{t\geq s}$ be a one-dimensional Brownian motion with generator $\partial_x^2$ initiated at time $s$ and position $y$; it  induces the probability measure $\Pi_{s,y}$  on the canonical path space. 
	Then, from the fact that $\phi\in \cC_\mathrm b^{1,2}([0,T]\times \mathbb R)$, we can use Ito's formula and the optional sampling theorem to get 
	\begin{equation} \label{eq:W.10a}
		\phi_{s,y} = \Pi_{s,y} \Big[\phi_{\tilde \rho,B_{\tilde \rho}} - \int_s^{\tilde \rho} \Big( \frac{1}{2} \vartheta \phi^2_{r,B_r} - \psi_{B_r} \Big) \mathrm dr \Big]
	\end{equation}
	for each optional time $\tilde \rho\in [s,T]$, defined on the probability space where $B$ is defined.
	Choose an arbitrary $z\in (y, b)$. 
	Denote by $\rho_z$ the first time for the Brownian motion $B$ hitting $\{z\}$. 
	Replacing $\tilde \rho$ in \eqref{eq:W.10a} by $T\wedge \rho_z$, we get from \eqref{eq:W.11} that
	\begin{equation}\label{eq:W.077}
		\phi_{s,y} 
		\leq \Pi_{s,y} [\phi_{T\wedge \rho_z,B_{T\wedge \rho_z}}]
		\leq \frac{18}{\vartheta^2(z - b)^2} \Pi_{s,y} (\rho_z<T).
	\end{equation}
	From the reflecting principle we have
	\begin{align} 
		&\Pi_{s,y} (\rho_z<T) 
		=2\Pi_{0,0}[B_{T-s} \geq z-y]
		=2\int_{z-y}^\infty \frac{1}{\sqrt{4\pi(T-s)}} e^{- \frac{u^2}{4(T-s)}} \mathrm du
		\\& \label{eq:W.078}\leq 2\int_{z-y}^\infty \frac{1}{\sqrt{4\pi(T-s)}} \frac{u}{z-y}e^{- \frac{u^2}{4(T-s)}} \mathrm du
		\leq \frac{2}{\sqrt{\pi}}\frac{\sqrt{T-s}}{z-y} e^{ - \frac{(z-y)^2}{4(T-s)}}.
	\end{align}
	Note that $z\in (y, b)$ is chosen arbitrarily. 
	So taking $z = \frac{y+b}{2}$ in \eqref{eq:W.077} and \eqref{eq:W.078}, we get 
	\begin{align}
		&\phi_{s,y}  
		\leq \frac{18}{\vartheta^2(z - b)^2}\frac{2}{\sqrt{\pi}}\frac{\sqrt{T-s}}{z-y} e^{ - \frac{(z-y)^2}{4(T-s)}}
		\leq \frac{2^8}{\vartheta^2}\frac{\sqrt{T-s}}{(b-y)^3} e^{ - \frac{(b-y)^2}{2^4(T-s)}}.
		\qedhere
	\end{align}
\end{proof}
	In order to study the property of ${\tilde w}$ using the above testing function $\phi$, we need the following 
	%lemma whose prove is postponed later in this section.
	lemma. %

\begin{lemma} \label{thm:S.3}
	Under the conditions of Proposition \ref{thm:S}, it holds that 
\[
	\sup_{t\in [0,\tilde T]}\mathbb E^\mathrm Q\Big[\int {\tilde w}_{t,x} \mathrm dx\Big]< \infty.
\]
	Furthermore, \eqref{eq:W.1} holds almost surely for each $t\in [0,\tilde T]$ and $\phi\in \cC_{\mathrm b}^{1,2}([0,\tilde T]\times \mathbb R)$.
\end{lemma}
\begin{proof}
	\emph{Step 1.}
	It is routine (c.f. \cite[Theorem 2.1]{Shiga1994Two}) to verify that 
	\[
	{\tilde w}_{t,x} 
	= \iint_0^t G_{s,y;t,x} \big(\tilde \sigma_{s,y}W(\mathrm ds\mathrm dy) + \mu(\mathrm ds\mathrm dy)\big),
	\quad \text{a.s.} 
	\quad (t,x)\in [0,\tilde T]\times \mathbb R.
	\]
	
	\emph{Step 2.}
	For an arbitrary fixed $(t,x)\in [0,\tilde T]\times\mathbb R$, we will show that
	\[
	\mathbb E^\mathrm Q[{\tilde w}_{t,x}]
	\leq \mathbb E^\mathrm Q\Big[\iint_0^t G_{s,y;t,x} \mu(\mathrm ds\mathrm dy)\Big].
	\]
	To do this, for each $r \geq 0$, define $\mathrm I_r:= \mathrm {II}_r + \mathrm {III}_r$ where
	\[ 
	\mathrm {II}_r
	:= 	\iint_0^r G_{s,y;t,x} \mu(\mathrm ds\mathrm dy); 
	\qquad \mathrm {III}_r 
	:= \iint_0^r G_{s,y;t,x} \tilde \sigma_{s,y} W(\mathrm ds\mathrm dy).
	\]
	We can verify that if $r\geq t$, then $\mathrm I_r = {\tilde w}_{t,x}$, and if $r\in [0,t)$, then from stochastic Fubini theorem we get 
	\[
	\mathrm I_r
	= \int G_{t,x;r,z}{\tilde w}_{r,z}\mathrm dz, \quad \text{a.s.}
	\]
	In particular, $(\mathrm I_r)_{r\geq 0}$ is a non-negative process.
	Note that $(\mathrm {III}_r)_{r\geq 0}$ is a local martingale.
	So there exists a sequence of stopping time $(\rho_n)_{n\in \mathbb N}$ such that for each $n\in \mathbb N$, $(\mathrm {III}_{r\wedge \rho_n})_{r\geq 0}$ is a martingale, and $\rho_n\uparrow \infty$  almost surely as $n \uparrow \infty$.
	Now for any fixed $r\geq 0$ we can verify from Fatou's lemma that $\mathrm Q[\mathrm I_r]\leq \liminf_{n\to \infty} \mathrm Q[\mathrm I_{r\wedge \rho_n}]\leq \mathrm Q[\mathrm {II}_r].$
	In particular $\mathrm Q[{\tilde w}_{t,x}] = \mathrm Q[\mathrm I_t] \leq \mathrm Q[\mathrm {II}_t]$ as desired.
	
	\emph{Step 3.}
	From Fubini's theorem we can verify from Step 2 that for each $t\in [0,\tilde T]$,
	\[
	\mathbb E^\mathrm Q \Big[\int {\tilde w}_{t,x}\mathrm dx \Big]
	\leq \mathbb E^\mathrm Q \Big[\iint_0^{t} \mu(\mathrm ds\mathrm dy) \int G_{s,y;t,x} \mathrm dx\Big]
	\leq \mathbb E^\mathrm Q\Big[\iint_0^{\tilde T} \mu(\mathrm ds\mathrm dy)\Big]<\infty.
	\]
	This proves the first part of the lemma.
	
	\emph{Step 4.} Let $g$ and sequence $(g_n)_{n\in \mathbb N}$ be $\mathbb R$-valued Borel functions on a Polish space $S$.
	We say $(g_n)_{n\in \mathbb N}$ converges to $g$ bounded pointwise if $(g_n)_{n\in \mathbb N}$ converges to $g$ pointwise, and $\sup_{n\in \mathbb N, s\in S}|g_n(s)|<\infty$. 
	Fix any $\phi\in \cC_{\mathrm b}^{1,2}([0,\tilde T]\times \mathbb R)$. Then it is easy to get that 
	there exists a sequence of $(\phi^{(n)}:n\in \mathbb N)$ in $\cC_\mathrm c^\infty([0,\tilde T]\times \mathbb R)$ 
	such that, $(\phi^{(n)})_{n\in \mathbb N}$, $(\partial_t\phi^{(n)})_{n\in \mathbb N}$, $(\partial_x\phi^{(n)})_{n\in \mathbb N}$ and $(\partial_x^2\phi^{(n)})_{n\in \mathbb N}$ converges bounded pointwise to $\phi$, $\partial_t \phi$, $\partial_x\phi$ and $\partial_x^2\phi$,  respectively.
	\begin{note}(To see this, one can take
		\[
		\phi^{(n)}_{t,x}:= \iint_0^{\tilde T} \phi_{s,y} a^{(n)}_{s,y} b^{(n)}_{t-s,x-y} \mathrm ds\mathrm dy, \quad (t,x)\in [0,\tilde T]\times \mathbb R, n\in \mathbb N
		\]
		where $(a^{(n)}:n\in \mathbb N)$ is a suitable sequence of smooth cut-off functions on $\mathbb R^2$ converging to $1$, and $(b^{(n)}: n \in \mathbb N)$ is a suitable sequence of mollifiers on $\mathbb R^2$ converging to the Dirac delta function.)
	\end{note}
	
	\emph{Final Step.} 
	From Steps 3, 4, bounded convergence theorem, \cite[Proposition 17.6]{Kallenberg2002Foundations} and the fact that $\tilde \sigma^2\leq \tilde \vartheta {\tilde w}$ on $[0,\tilde T]\times \mathbb R$, we can verify that \eqref{eq:W.1} holds almost surely for each $t\in [0,\tilde T]$ and $\phi\in \cC_{\mathrm b}^{1,2}([0,\tilde T]\times \mathbb R)$. 
\end{proof}
	
	We are now ready to give the proof of Proposition \ref{thm:S}.
	
\begin{proof}[Proof of Proposition \ref{thm:S}]
	\emph{Step 1.} We only need to prove the desired result for the case $-\infty=a<b< \infty $.
	In fact, in the case of $ a=-\infty, b = \infty$, nothing needs to be proved. 
	And if the desired result holds for the case $-\infty=a< b<\infty$, then by symmetry, it also holds for the case $-\infty < a< b= \infty$.
	  For the only remaining case  $-\infty < a < b < \infty$, we use
\[
	\mathrm Q\Big( \int_0^{\tilde T} \mathrm ds\int_{(a,b)^c}  {\tilde w}_{s,y}\mathrm dy > 0\Big) 
	\leq \mathrm Q\Big( \int_0^{\tilde T} \mathrm ds\int_{-\infty}^a  {\tilde w}_{s,y}\mathrm dy > 0\Big) +\mathrm Q\Big( \int_0^{\tilde T} \mathrm ds\int_{b}^\infty  {\tilde w}_{s,y}\mathrm dy > 0\Big).
\]

	\emph{Step 2.} Fix $b\in \mathbb R$ and a non-negative $\psi\in \cC_0^2(\mathbb R)$ with support $\{x\in \mathbb R: \psi_x>0\}=(b, \infty)$. 
	For each $n > 0$, let $\phi^{(n)} \in \cC_{\mathrm b}^{1,2}([0,\tilde T] \times \mathbb R)$ be given by Lemma \ref{thm:S.1} with $\psi$ replaced by $n \psi$ and $\vartheta$ from Proposition~\ref{thm:S}. 
	For any $n > 0$, define process
\[
		M^{(n)}_t
		:=  n \iint_0^t {\tilde w}_{s,y} \psi_{y}\mathrm ds\mathrm dy + \int {\tilde w}_{t,x} \phi_{t,x}^{(n)}\mathrm dx, 
	\quad t \in [0,\tilde T].
\]
	We note that
	\begin{align}
	 &\mathrm Q\Big(\int_0^{\tilde T}\mathrm ds\int_b^\infty {\tilde w}_{s,y}\mathrm dy > 0 \Big)
		= \mathrm Q\Big(\iint_0^{\tilde T} {\tilde w}_{s,y} \psi_{y}\mathrm ds\mathrm dy> 0 \Big)
		\\ &= \lim_{n \to \infty} \mathbb E^\mathrm Q \Big( 1 - \exp\Big\{- n\iint_0^{\tilde T} {\tilde w}_{s,y} \psi_{y}\mathrm ds\mathrm dy \Big\}\Big) 
		= \lim_{n \to \infty} \mathbb E^\mathrm Q( 1 - e^{- M_{\tilde T}^{(n)} }).
	\end{align}
	
	 \emph{Step 3.} 
	We will verify that 
	\[\mathbb E^\mathrm Q( 1 - e^{- M^{(n)}_{\tilde T}}) \leq \mathbb E^\mathrm Q \Big[\iint_0^{\tilde T} \phi^{(n)}_{s,y}\mu(\mathrm ds\mathrm dy)\Big], \quad n > 0.\]
	In fact, from Lemma \ref{thm:S.3} we have for each $t\in [0,\tilde T]$ almost surely
	\begin{equation} 
	\begin{multlined}
			M_t^{(n)} 
			\overset{\eqref{eq:W.1}}= n \iint_0^t {\tilde w}_{s,y} \psi_{y} \mathrm ds\mathrm dy + \iint_0^t {\tilde w}_{s,y}(\partial_y^2 \phi^{(n)}_{s,y} + \partial_s \phi_{s,y}^{(n)}) \mathrm ds\mathrm dy 
			\\ + \iint_0^t \tilde \sigma_{s,y}\phi_{s,y}^{(n)} W(\mathrm ds\mathrm dy) + \iint_0^t \phi_{s,y}^{(n)} \mu(\mathrm ds \mathrm dy). 
		\end{multlined}
	\end{equation}
	Therefore, we have almost surely
\[
		\langle M^{(n)} \rangle_t
		= \iint_0^t (\tilde \sigma_{s,y} \phi_{s,y}^{(n)} )^2 \mathrm ds\mathrm dy,
		\quad t\in [0,\tilde T].
\]
	Now, we use It\^o's formula and get that for any $t\in [0,\tilde T]$ almost surely,
	\begin{align}
		& e^{-M^{(n)}_t} -1
		=  \int_0^t(- e^{-M^{(n)}_s})\mathrm d M^{(n)}_s + \frac{1}{2} \int_0^t  e^{-M^{(n)}_s} \mathrm d\langle M^{(n)}\rangle_s
		\\&\begin{multlined}
			= \iint_0^t (-e^{- M^{(n)}_s}) \big(n {\tilde w}_{s,y} \psi_{y} +{\tilde w}_{s,y} (\partial_y^2\phi^{(n)}_{s,y} +\partial_s \phi^{(n)}_{s,y}) \big) \mathrm ds\mathrm dy 
			\\ + \iint_0^t (-e^{-M_s^{(n)} }\phi^{(n)}_{s,y}) \big(\tilde \sigma_{s,y}W(\mathrm ds\mathrm dy)+\mu(\mathrm ds\mathrm dy) \big) 
			 +  \frac{1}{2}\iint_0^t e^{-M_s^{(n)}}( \tilde \sigma_{s,y}\phi_{s,y}^{(n)} )^2\mathrm ds\mathrm dy
		\end{multlined}
		\\&\label{eq:S.4}\begin{multlined}
			\overset{\eqref{eq:S.2}}= \frac{1}{2}\iint_0^t e^{-M_s^{(n)}} (\phi^{(n)}_{s,y})^2 \big(\tilde \sigma_{s,y}^2 - \vartheta^2 {\tilde w}_{s,y}\big) \mathrm ds\mathrm dy
			\\+ \iint_0^t (-e^{-M_s^{(n)} }\phi^{(n)}_{s,y}) \tilde \sigma_{s,y}W(\mathrm ds\mathrm dy) +\iint_0^t (-e^{-M_s^{(n)} }\phi^{(n)}_{s,y}) \mu(\mathrm ds\mathrm dy).
		\end{multlined}
	\end{align}
	Note that the second integral on the right hand side of \eqref{eq:S.4} is a $L^2$-bounded martingale on $[0,\tilde T]$ since from Lemma \ref{thm:S.3},
\[
	\mathbb E^\mathrm Q\Big[\iint_0^{\tilde T} (-e^{-M_s^{(n)} }\phi^{(n)}_{s,y})^2 (\tilde \sigma_{s,y})^2 \mathrm ds\mathrm dy\Big]
	\leq \|\phi^{(n)}\|_\infty^2 {\tilde \vartheta}^2 \mathbb E^\mathrm Q\Big[\iint_0^{\tilde T} {\tilde w}_{s,y}\mathrm ds\mathrm dy\Big]<\infty.
\]
Noticing that $\tilde \sigma^2\geq \vartheta^2 {\tilde w}$ on $\mathbb R_+\times \mathbb R$, we can take expectation on \eqref{eq:S.4} and get that 
\begin{align}
		&\mathbb E^\mathrm Q[ 1 - e^{- M^{(n)}_{\tilde T} }] 
	      \\&	= \mathbb E^\mathrm Q \Big[ \frac{1}{2}\iint_0^{\tilde T} e^{-M_s^{(n)}} (\phi^{(n)}_{s,y})^2 \big( \vartheta^2 {\tilde w}_{s,y}-(\tilde \sigma_{s,y})^2\big) \mathrm ds\mathrm dy +  \iint_0^{\tilde T} e^{-M_s^{(n)} }\phi^{(n)}_{s,y} \mu(\mathrm ds\mathrm dy)\Big]
		\\& \leq \mathbb E^\mathrm Q \Big[ \iint_0^{\tilde T} e^{-M_s^{(n)} }\phi^{(n)}_{s,y} \mu(\mathrm ds\mathrm dy)\Big]
		\leq \mathbb E^\mathrm Q \Big[ \iint_0^{\tilde T} \phi^{(n)}_{s,y} \mu(\mathrm ds\mathrm dy)\Big].
	\end{align}
	\emph{Final step.} The desired result now follows from Steps 3, 4 and Lemma \ref{thm:S.1}. 
\end{proof}

\section{proof of Proposition \ref{thm:V}} \label{sec:V}

We first need the following lemma to control the small time fluctuation of certain random fields.  
This lemma is modified from \cite[Lemma 6.1]{MuellerMytnikQuastel2011Effect} in order to incorporate the small time intervals.
Its proof follows the lines of the proof of  \cite[Lemma 6.1]{MuellerMytnikQuastel2011Effect} and therefore is omitted.

\begin{lemma} \label{thm:V.1}
	Suppose that 
	\begin{enumerate}
		\item 
			$\tilde T >0$, $\tilde L >0$, $a\in \mathbb R$ are arbitrary and
		$
		\mathbf H: = [0, \tilde T] \times [a,a+\tilde L];
		$
		\item
			$(g_{s,y;t,x}:(s,y), (t,x)\in \mathbb R_+\times \mathbb R )$ and $(\eta_{t,x}: (t,x)\in \mathbb R_+\times \mathbb R)$ are deterministic non-negative functions satisfying
		\[
		B
		:= \sup_{(t',x'),(t,x)\in \mathbf H} \frac{\iint_0^\infty (g_{s,y; t',x'} - g_{s,y;t,x})^2 \eta_{s,y}\mathrm ds\mathrm dy}{|\frac{x'-x}{\tilde L}| + |\frac{t'-t}{\tilde T}|^{1/2}} < \infty;
		\]
		\item
			$W$ is a white noise defined on a filtered probability space $(\Omega, \mathcal G, (\mathcal F_t)_{t\geq 0}, \mathrm P)$;
\item
	$\tilde \sigma$ is a predictable random field on $\Omega$
  such that almost surely $\tilde \sigma^2 \leq \eta$ on $\mathbb R_+\times \mathbb R$;
\item
	$Z$ is a continuous random field on $\Omega$
 such that for all  $ (t,x)\in \mathbf H$,
\[
	Z_{t,x}
	= \iint_0^\infty g_{s,y;t,x} \tilde \sigma_{s,y}W(\mathrm ds\mathrm dy)
	\quad \text{a.s.}
\]
	\end{enumerate}
Then for each $z\geq 0$,
	\[
	\mathrm P\Big( \sup_{(t,x), (t',x')\in \mathbf H} |Z_{t',x'}-Z_{t,x}|>  z \sqrt{B}\Big)
	\leq 2^5 e^{- z^2/2^{12}}.
	\]
\end{lemma}

Next result is a simple corollary of  the above lemma.
\begin{corollary} \label{thm:V.2}
	Lemma \ref{thm:V.1} still holds if its conditions (1) and (2) are replaced by:
	\begin{itemize}
		\item [(1')] 
			$\tilde {\mathrm v} > 0$ and $a \in \mathbb R$ are arbitrary and
		\[
			\mathbf H:= \{(t,x)\in \mathbb R_+\times \mathbb R: t\in [0, \tilde{\mathrm v}^{-2}], x-{\tilde{\mathrm v} }t\in [a,a+ \tilde{\mathrm v}^{-1}]\};
		\]
	\item [(2')]
			$(g_{s,y;t,x}:(s,y), (t,x)\in \mathbb R_+\times \mathbb R )$ and $(\eta_{t,x}: (t,x)\in \mathbb R_+\times \mathbb R)$ are (deterministic) non-negative functions satisfying
		\[
		B
		:= \sup_{(t',x'),(t,x)\in \mathbf H} \frac{\tilde{\mathrm v}^{-1}\iint_0^\infty (g_{s,y; t',x'} - g_{s,y;t,x})^2 \eta_{s,y}\mathrm ds\mathrm dy}{|(x'-\tilde{\mathrm v}t')-(x- \tilde{\mathrm v}t)| + |t'-t|^{1/2}} < \infty.
		\]
	\end{itemize}
\end{corollary}
	In order to control the quantity $B$ in
	Lemma~\ref{thm:V.1} and Corollary~\ref{thm:V.2}
	we will be using the following analytical lemma.
\begin{lemma}\label{thm:V.3}
For any $\tilde{\mathrm v}>0$ and $(t,x), (t',x')\in \mathbb R_+ \times \mathbb R$ satisfying
\begin{equation}\label{eq:V.4} 
	t,t'
	\in [0,\tilde{\mathrm v}^{-2}]; 
	\quad x -\tilde{\mathrm v}t, x'-\tilde{\mathrm v}t' 
	\in (-\infty, 0]; 
	\quad |(x'-\tilde{\mathrm v}t')- (x - \tilde{\mathrm v}t)| 
	\leq \tilde{\mathrm v}^{-1},
\end{equation}
it holds that
\begin{equation}
\begin{multlined}
	\iint_0^\infty (G^{(\tilde {\mathrm v})}_{s,y;t',x'} - G^{(\tilde{\mathrm v})}_{s,y;t,x})^2 e^{-\tilde{\mathrm v}(y - \tilde{\mathrm v}s)} \mathrm ds\mathrm dy
	\\ \leq 2^9 e^{-\tilde{\mathrm v} (x-\tilde{\mathrm v}t)}\big(|(x'-\tilde{\mathrm v}t')-(x - \tilde{\mathrm v}t)|+|t'-t|^{1/2}\big).
\end{multlined}
\end{equation}
\end{lemma}
\begin{proof}
	Let us fix an arbitrary $\tilde{\mathrm v}>0$ and arbitrary $(t,x), (t',x')\in \mathbb R_+\times \mathbb R$ satisfying \eqref{eq:V.4}. 
	Define $z := x -\tilde{\mathrm v}t$ and $z' := x'- \tilde{\mathrm v} t'$.
	By the symmetry between $(t,x)$ and $(t',x')$, we can assume without loss of generality that $\frac{\tilde{\mathrm v}}{2}(z'-z) + \frac{\tilde{\mathrm v}^2}{4}(t'-t) \geq 0$.
	
	\emph{Step 1.} 
	Note that one can give the precise expression of $G^{(\tilde{\mathrm v})}$ using the reflection principle and Girsanov transformation for the Brownian motion (see \cite[Proof of Lemma 6.2]{MuellerMytnikQuastel2011Effect}).
	In fact, for each $(s,y)\in \mathbb R_+\times \mathbb R$, we have 
	\[
	G^{( \tilde{\mathrm v})}_{s,y;t,x} 
	= \rho^{(1)}_{s,y-\tilde{\mathrm v} s;t,z} - \rho^{(-1)}_{s,y - \tilde{\mathrm v} s; t, z}
	\]
	where \[
	\rho^{(i)}_{s,y;t,z} 
	:= e^{-\frac{\tilde{\mathrm v}}{2}(z - y) - \frac{\tilde{\mathrm v}^2}{4}(t-s)} G_{s,y;t,i z}\mathbf 1_{y,z\leq 0}, \quad i\in \{1,-1\}.
	\]
	Now from the fact that the squares of the sum of two numbers is bounded by twice the sum of the squares of those two numbers, we have
	\begin{equation}
		(G^{(  \tilde{\mathrm v} ) }_{s,y;t',x'} - G^{( \tilde{\mathrm v} )}_{s,y;t,x})^2
		\leq 2\sum_{i = 1,-1}(\rho^{(i)}_{s,y-\tilde{\mathrm v} s;t',z'} - \rho^{(i)}_{s,y-\tilde{\mathrm v}s;t,z})^2,
		\quad (s,y)\in \mathbb R_+\times \mathbb R.
	\end{equation}

	\emph{Step 2.}
	We show that for each $i\in \{-1,1\}$ we have
	\begin{align}
		& \mathrm I_i :=
		\iint_0^\infty   (\rho^{(i)}_{s,y;t',z'} - \rho^{(i)}_{s,y;t,z})^2 e^{-\tilde{\mathrm v} (y-z)} \mathrm ds\mathrm dy 
		\leq  4 (\mathrm {II}_i + \mathrm {III}_i) 
	\end{align}
	where
	\begin{equation}
		\mathrm  {II}_i
		:= \iint_0^\infty  \big((\gamma - 1) G_{s,y;t',i z'}\big)^2 \mathrm ds\mathrm dy;
		\quad \mathrm {III}_i 
		:= \iint_0^\infty  (G_{s,y;t',i z'}- G_{s,y;t,i z})^2 \mathrm ds\mathrm dy
	\end{equation}
	and	
	$
	\gamma 
	:= e^{-\frac{\tilde{\mathrm v}}{2}(z'-z) - \frac{\tilde{\mathrm v}^2}{4}(t'-t)}.
	$
	In fact we can verify that
	\begin{align}
		&\mathrm I_i 
		= \iint_0^\infty   e^{-\tilde{\mathrm v}(z-y)- \frac{\tilde{\mathrm v}^2}{2}(t-s)} \mathbf 1_{y\leq 0}(\gamma G_{s,y;t',i z'}- G_{s,y;t,i z})^2 e^{-\tilde{\mathrm v}(y-z)} \mathrm ds\mathrm dy 
		\\& \leq e^{\frac{\tilde{\mathrm v}^2}{2}|t'-t|} \iint_0^\infty  (\gamma G_{s,y;t',i z'}- G_{s,y;t,i z})^2 \mathrm ds\mathrm dy.
	\end{align}
	The desired result in this step then follows from \eqref{eq:V.4} that $\tilde{\mathrm v}^2|t-t'|\leq 1$.
	
	\emph{Step 3.} 
	We show that for each $i\in \{-1,1\}$ we have 
	\[
	\mathrm {II}_i 
	\leq ( |z'-z|+ |t'-t|^{1/2}) /4
	\]
	where $\mathrm {II}_i$ is given in Step 2.
	In fact,
	\begin{align}
		&\mathrm {II}_i		
		= (\gamma - 1)^2 \iint_0^{t'} \frac{e^{-\frac{(i z'-y)^2}{2(t'-s)}}}{4\pi (t'-s)} \mathrm ds\mathrm dy
		= (\gamma - 1)^2 \frac{\sqrt{t'}}{\sqrt{2\pi}}
		\\&\leq \Big(\frac{\tilde{\mathrm v}}{2}(z'-z) +\frac{\tilde{\mathrm v}^2}{4}(t'-t)\Big)^2 \frac{\tilde{\mathrm v}^{-1}}{\sqrt{2\pi}}
		\leq 2\Big(\frac{\tilde{\mathrm v}}{2}(z'-z)\Big)^2 \frac{\tilde{\mathrm v}^{-1}}{\sqrt{2\pi}}+2\Big(\frac{\tilde{\mathrm v}^2}{4}(t'-t)\Big)^2 \frac{\tilde{\mathrm v}^{-1}}{\sqrt{2\pi}}
		\\ &=  \frac{\tilde{\mathrm v}|z'-z|}{2\sqrt{2\pi}} |z'-z| + \frac{\tilde{\mathrm v}^3|t'-t|^{3/2} }{8\sqrt{2\pi}}  |t'-t|^{1/2}.
	\end{align}
	Here, in the first inequality, we used the fact that $\frac{\tilde{\mathrm v}}{2}(z'-z) + \frac{\tilde{\mathrm v}^2}{4}(t'-t) \geq 0$.
	The desired result in this step then follows from \eqref{eq:V.4} that $\tilde{\mathrm v}|z-z'|\leq 1$ and $\tilde{\mathrm v}^2|t-t'|\leq 1$.
	
	\emph{Step 4.}
	We note from \cite[Lemma 6.2(1)]{Shiga1994Two} that there exists a universal constant $\tilde C>0$, independent of our choice of $(t,x), (t',x')$ and $\tilde{\mathrm v}$,  such that 	$
	\mathrm {III}_i
	\leq \tilde C ( |z'-z|+ |t'-t|^{1/2})
	$ for each $i\in \{-1,1\}$.
	In fact, one can take $\tilde C = 2^7$ (c.f. Lemma \ref{thm:B.3}).
	
	\emph{Final Step.}
	From Step 1, we know that
	\begin{align} 
		&\iint_0^\infty (G^{(\tilde{\mathrm v})}_{s,y;t',x'} - G^{(\tilde{\mathrm v})}_{s,y;t,x})^2 \frac{e^{-\tilde{\mathrm v}(y - \tilde{\mathrm v}s)} }{e^{-\tilde{\mathrm v} (x-\tilde{\mathrm v}t)}} \mathrm ds\mathrm dy
		\\& \leq \sum_{i\in \{-1,1\}} \iint_0^\infty (\rho^{(i)}_{s,y-\tilde{\mathrm v} s;t',z'} - \rho^{(i)}_{s,y-\tilde{\mathrm v} s;t,z})^2\frac{e^{-\tilde{\mathrm v}(y - \tilde{\mathrm v}s)} }{e^{-\tilde{\mathrm v} z}} \mathrm ds\mathrm dy
		= \sum_{i\in \{-1,1\}} \mathrm I_i.
	\end{align}
	The desired result in this Lemma then follows from Steps 2, 3 and 4.
\end{proof}
	We are now ready to give the proof of Proposition \ref{thm:V}.
\begin{proof}[Proof of Proposition \ref{thm:V}]
	\emph{Step 1.}
Define 
\[\mathrm I_{t,x} := \varepsilon kL e^{-\theta \mathrm v(x - \mathrm vt)} \mathbf 1_{x\leq \mathrm vt}, \quad(t,x) \in \mathbb R_+\times \mathbb R.\]
Let $\bar v$ be as in  Proposition \ref{thm:C}. Then by part (4) of that proposition we have that $\bar v \geq v$ on $\mathbb R_+\times \mathbb R$ almost surely.
	Therefore, in order to prove Proposition \ref{thm:V}, we only have to show that $\mathrm P(\tilde \tau_3 < T)< 1/8$ holds with 
\begin{equation}
\tilde \tau_3 
:= \inf\{t\in [0,T]: \bar v_{t,x}\geq F(x - \mathrm vt)+ \mathrm I_{t,x} \text{ for some } x\in (-\infty, \mathrm vt]\}.
\end{equation}

\emph{Step 2.}
	Define $\tilde Z_{t,x} := e^{-\alpha t}(\bar v_{t,x} - \varrho_{t,x})$ for each $(t,x)\in \mathbb R_+\times \mathbb R$ where $\varrho$ is given in \eqref{eq:H.04} and $\alpha := \theta (1-\theta)\mathrm v^2$.
	Then it can be verified from Lemma \ref{thm:H.05} that 
\begin{equation} 
	\tilde Z_{t,x} 
	=  \iint_0^t G_{s,y;t,x}^{(\mathrm v)}\epsilon e^{-\alpha s} \sigma(\bar v_{s,y}) W^{\bar v}( \mathrm ds\mathrm dy),
	\quad \text{a.s.}
	\quad (t,x)\in \mathbb R_+\times \mathbb R.
\end{equation}
	From this we immediately get that 
\[
\tilde \tau_3 
= \inf\{t\in [0,T]: \tilde Z_{t,x} \geq e^{-\alpha t}\mathrm {I}_{t,x} \text{ for some } x\in (-\infty, \mathrm vt]\}.
\]

\emph{Step 3.}
We show that almost surely 
\[
	\epsilon e^{-\alpha t} \sigma(\bar v_{t,x}) \leq \epsilon \sigma(\bar v_{t,x})\leq  \sqrt{\eta_{t,x}}, \quad (t,x) \in [0,\tilde \tau_3]\times \mathbb R
\]
	where \[\eta_{t,x} := 2kL \epsilon^2 \varepsilon e^{-\theta \mathrm v(x - \mathrm vt)}\mathbf 1_{x\leq \mathrm vt}, \quad (t,x)\in \mathbb R_+\times \mathbb R.\] 
	In fact, almost surely for each $(t,x) \in [0,\tilde \tau_3]\times \mathbb R$,
\begin{align}
	&\sigma(\bar v_{t,x})^2 
	\leq \bar v_{t,x} 
	\leq F(x - \mathrm vt)+\mathrm I_{t,x}
      \\& \leq \frac{\varepsilon}{\theta \mathrm v}(e^{-\theta \mathrm v(x-\mathrm v t)}-1) \mathbf 1_{x\leq \mathrm vt} + \varepsilon kL e^{-\theta \mathrm v(x-\mathrm v t)}  \mathbf 1_{x\leq \mathrm vt}
	\\&\leq (2 + k)L  \varepsilon e^{-\theta \mathrm v(x-\mathrm v t)}\mathbf 1_{x\leq \mathrm vt}.
\end{align}
Note from \eqref{eq:U.2} and \eqref{eq:U.6} that $k\geq 2$.
	The desired result in this step follows.

\emph{Step 4.}
From Step 3 we can verify that almost surely $Z = \tilde Z$ on $[0,\tilde \tau_3]\times \mathbb R$ where $Z$ is a continuous random field so that 
\begin{equation} 
	Z_{t,x}   
	= \iint_0^t G_{s,y;t,x}^{(\mathrm v)} \Big( \sqrt{\eta_{s,y}} \wedge \big( \epsilon e^{-\alpha s} \sigma(\bar v_{s,y})\big) \Big) W^{\bar v}( \mathrm ds\mathrm dy)
	\quad \text{a.s.} 
	\quad (t,x)\in \mathbb R_+\times \mathbb R.
\end{equation}
Thus from Step 2, we get that 
\[\tilde \tau_3 = \inf \{t\in [0,T]: Z_{t,x} \geq e^{- \alpha t} \mathrm I_{t,x} \text{ for some } x\leq \mathrm vt\}, 
\quad 
	{\rm a.s.}
\] 
\emph{Step 5.}
Define \[
	\Gamma_n:= \{(t,x)\in [0,T]\times \mathbb R: x- \mathrm vt \in (-nL, -(n-1)L]\}, \quad  n\in \mathbb N.
	\]
	We can verify from \eqref{eq:U.3} that for each $n\in \mathbb N$ and $(t,x)\in \Gamma_n$, 
\[
	e^{-\alpha t}\mathrm I_{t,x}
	\geq e^{-\alpha T} \varepsilon kL e^{\theta \mathrm v(n-1)L} 
	= C_\theta \varepsilon k L e^{\theta n} 
	=: \mathrm {II}_n 
\] 
where $C_\theta := e^{-\theta(2-\theta)}$.

\emph{Step 6.}
For each $n\in \mathbb N$, we can
get  from Lemma \ref{thm:V.3} and \eqref{eq:U.3} that
\begin{align} 
	B_n
	= \sup_{(t,x),(t',x')\in \Gamma_n} \frac{\mathrm v^{-1}\iint_0^\infty (G_{s,y;t,x}^{(\mathrm v)} - G_{s,y;t',x'}^{(\mathrm v)})^2 \eta_{s,y} \mathrm ds\mathrm dy}{ |(x'-\mathrm v t') - (x-\mathrm vt)| + |t'-t|^{1/2}}  
	\leq 2^{10} k L^2 \epsilon^2 \varepsilon e^n.
\end{align}
	In fact, for $(t,x), (t',x') \in \Gamma_n$, since \eqref{eq:V.4} holds, we have from Lemma \ref{thm:V.3} that 
\begin{align} 
 &\iint_0^\infty (G_{s,y;t,x}^{(\mathrm v)} - G_{s,y;t',x'}^{(\mathrm v)})^2 \eta_{s,y} \mathrm ds\mathrm dy
 \leq 2^2kL \epsilon^2 \varepsilon \iint_0^\infty (G_{s,y;t,x}^{(\mathrm v)} - G_{s,y;t',x'}^{(\mathrm v)})^2 e^{- \mathrm v(y - \mathrm vs)}\mathrm ds\mathrm dy
\\& \leq 2kL \epsilon^2 \varepsilon 2^9 e^{-\mathrm v (x-\mathrm vt)}(|(x'-\mathrm vt')-(x - \mathrm vt)|+|t'-t|^{1/2})
\\& \leq 2^{10} kL \epsilon^2 \varepsilon  e^{n\mathrm v L}(|(x'-\mathrm vt')-(x - \mathrm vt)|+|t'-t|^{1/2}).
\end{align}
Noting  from \eqref{eq:U.3}  that $\mathrm vL = 1$, the desired  result in this step follows.

\emph{Step 7.}
From Step 6, \eqref{eq:U.3} that $\varepsilon = \gamma \epsilon^2$, \eqref{eq:U.6} that $\gamma k = \mathcal K$, and Corollary \ref{thm:V.2} we can
obtain
\begin{align} 
	&\mathrm P \Big( \sup_{\Gamma_n} Z \geq \mathrm {II}_n \Big) 
      \leq \mathrm P \Big( \sup_{\Gamma_n} Z \geq 2^{-5}C_\theta \sqrt{\varepsilon/\epsilon^2} \sqrt{k}  e^{(\theta-1/2) n} \sqrt{B_n}\Big) 
      \\&\leq \mathrm P \Big( \sup_{\Gamma_n} Z \geq 2^{-5}C_\theta \sqrt{\mathcal K}   e^{(\theta-1/2) n} \sqrt{B_n}\Big) 
      \\&\leq 2^5 \exp( -2^{-22} C_\theta^2 \mathcal K   e^{(2\theta-1) n} ).
\end{align}

\emph{Final Step.}
Using Steps 4, 5 and 7, we can verify that
\begin{align} 
&\mathrm P(\tilde \tau_3 < T) 
\leq \mathrm P\Big(\exists (t,x) \in \bigcup_{n=1}^\infty \Gamma_n: Z_{t,x} \geq e^{-\alpha t}\mathrm I_{t,x}\Big) 
	\\  &\leq \sum_{n=1}^\infty \mathrm P\Big(\sup_{\Gamma_n} Z \geq \mathrm {II}_n \Big) 
	\leq 2^5 \sum_{n=1}^\infty \exp( -2^{-22} C_\theta^2 \mathcal K e^{(2\theta-1) n} )
			\leq 1/8
\end{align}
where we used \eqref{eq:U.13} in the last inequality.
			\begin{note}
			where the constants $C_{\ref{thm:V.1}}$ and
$
\tilde C 
:=C_\theta^2 \gamma (\theta - \frac{1}{2})^2 /(2^2 C_{\ref{thm:V.3}}C_{\ref{thm:V.1}} )
>0
$
are independent of the choice of $\epsilon \in (0,\epsilon_0)$. 
Note that the fixed $k$ is chosen arbitrarily, so we can chose $k$ large enough, independent of the choice of $\epsilon \in (0,\epsilon_0)$, so that the right hand side of \eqref{eq:V.5} is strictly less than $1/8$.
\end{note}
\end{proof}

\begin{note}
\begin{proof}
\emph{Step 1.} Let us introduce the parabolic dyadic lattices of $\mathbf H$.
	For each $(t,x)$ and $(t',x')$ in $\mathbb R^+ \times \mathbb R$, define their parabolic distance as
	\[
	d_{t',x'; t,x} = |x'-x| + |t-t'|^{1/2}.
	\]
	For each $n\in \mathbb Z$, $i\in \mathbb Z_+ \cap [0,4^n]$ and $j\in \mathbb Z_+ \cap [0,2^n]$ define $t_{n;i,j}$ and $x_{n;i,j}$ as the solution of 
	\[
	\frac{t}{T} = \frac{i}{4^n},\quad \frac{x-\mathrm vt-a}{L} = \frac{j}{2^n}.
	\]
	Define a lattice on $\mathbf H$ by
	\[
	\mathbf G_n 
	:= \{(t_{n;i,j},x_{n;i,j}): i\in \mathbb Z_+ \cap [0,4^n), j\in \mathbb Z_+ \cap [0,2^n)\},
	\]
	and its edges $\mathbf E_n = \mathbf E_n^I \cup \mathbf E_n^J$ where 
	\begin{align}
		&\mathbf E_n^I
		:= \Big\{\big((t_{n;i,j},x_{n;i,j}), (t_{n;i+1,j},x_{n;i+1,j})\big); i\in \mathbb Z_+ \cap [0,4^n), j\in \mathbb Z_+ \cap [0,2^n)\Big\};
		\\&\mathbf E_n^J
		:= \Big\{\big((t_{n;i,j},x_{n;i,j}), (t_{n;i,j+1},x_{n;i,j+1})\big); i\in \mathbb Z_+ \cap [0,4^n), j\in \mathbb Z_+ \cap [0,2^n)\Big\}.
	\end{align}
	For any pair $e=((t,x),(t',x'))\in \mathbf H$, define
	\[
	d_e := |x'-x| + |t-t'|^{1/2},\quad Z_e := Z_{t',x'} - Z_{t,x}.
	\]
	Then for any $e\in \mathbf E_n^I$, we have 
	\begin{align}
		&d_e = |x_{n;i+1,j}-x_{n;i,j}| + |t_{n;i+1,j}-t_{n;i,j}|^{1/2}
		\\&=  \Big|\Big(L\frac{j}{2^n}+a+\mathrm vT\frac{i+1}{4^n}\Big) - \Big(L\frac{j}{2^n}+a+\mathrm vT\frac{i}{4^n}\Big)\Big| + \Big|T\frac{i+1}{4^n}-T\frac{i}{4^n}\Big|^{1/2}
		\\& = L(\frac{1}{4^n} + \frac{1}{2^n}) \leq L/2^{n-1}.
	\end{align}
	And for any $e\in \mathbf E_n^J$, we have 
	\begin{align}
		&d_e = |x_{n;i,j+1}-x_{n;i,j}| + |t_{n;i,j+1}-t_{n;i,j}|^{1/2}
		\\&=  \Big|\Big(L\frac{j+1}{2^n}+a+\mathrm vT\frac{i}{4^n}\Big) - \Big(L\frac{j}{2^n}+a+\mathrm vT\frac{i}{4^n}\Big)\Big| + \Big|T\frac{i}{4^n}-T\frac{i}{4^n}\Big|^{1/2}
		\\& = L/2^n \leq L/2^{n-1}.
	\end{align}
	Define
	\[
	\mathbf G = \bigcup_{n=1}^\infty \mathbf G_n; \quad \mathbf E = \bigcup_{n=1}^\infty \mathbf E_n.
	\]
	It is elementary to see that for any $(t,x),(t',x')\in \mathbf G$ there exists $\mathbf E_{t,x;t',x'}$, a finite subset of $\mathbf E$, such that (1) the edges in  $\mathbf E_{t,x;t',x'}$ connects $(t,x)$ and $(t',x')$; and (2) 
	\begin{align}
		\# (\mathbf E_{t,x;t',x'} \cap \mathbf E_n) \leq 8, \quad n \in \mathbb N.
	\end{align}  
	
	\emph{Step 2.}
		We show that for any $(t,x), (t',x')\in \mathbf H$ and $c\geq 0$, 
		\begin{align}
			\mathrm P(Z_{t',x'} - Z_{t,x} \geq c)  \leq \exp\Big\{ -\frac{c^2}{2} B^{-1} d_{t,x;t',x'}^{-1}\Big\}.
		\end{align}
	\begin{align} 
		& 1 = E\Big[\exp\Big\{ \theta  (Z_{t',x'} - Z_{t,x}) - \frac{1}{2} \theta^2 \iint_0^\infty ( g_{s,y; t,x} -  g_{s,y; t',x'} )^2 \sigma_{s,y}^2 \mathrm ds\mathrm dy\Big\}\Big] 
		\\& \geq E\Big[\exp\Big\{ \theta  (Z_{t',x'} - Z_{t,x}) - \frac{1}{2} \theta^2 \iint_0^\infty ( g_{s,y; t,x} -  g_{s,y; t',x'} )^2 \eta_{s,y} \mathrm ds\mathrm dy\Big\}\Big].
	\end{align}
	Therefore
	\begin{align} 
		&  E[\exp \{ \theta  (Z_{t',x'} - Z_{t,x}) \} ] 
		\leq \exp\Big\{ \frac{1}{2} \theta^2 \iint_0^\infty ( g_{s,y; t,x} -  g_{s,y; t',x'} )^2 \eta_{s,y} \mathrm ds\mathrm dy\Big\}
		\\& \leq \exp\Big\{ \frac{1}{2} \theta^2 Bd_{t',x';t,x}\Big\}.
	\end{align}
	By Chebyshev's inequality we have for any $\theta \geq 0$
	\begin{align}
		& P(Z_{t',x'} - Z_{t,x} \geq c) = P[\exp\{\theta (Z_{t',x'} - Z_{t,x} )\} \geq e^{\theta c}]
		\leq P[\exp\{\theta (Z_{t',x'} - Z_{t,x} )\}] / e^{\theta c}
		\\ &\leq \exp\Big\{ \frac{1}{2} \theta^2 Bd_{t',x';t,x} - \theta c\Big\}
	\end{align}
	Optimizing on $\theta$ we get the desired result.
	
	\emph{Step 3.}
		We show that for any $\rho > 8$, we have
		\[
		P(A^c) \leq  \frac{2}{\frac{\rho}{8}-1}
		\]
		where
		\[
		A:= \bigcap_{n\in \mathbb N}\bigcap_{e\in \mathbf E_n}\{Z_{e}\leq c_n\}; 
		\quad c_n = \sqrt{2(\log \rho) n BL/2^{n-1}}, \quad n\in \mathbb N.
		\]
	\begin{align}
		&P(A^c) \leq \sum_{n\in \mathbb N} \sum_{e\in \mathbf E_n} P(Z_e>c_n) 
		\leq \sum_{n\in \mathbb N}  \sum_{e\in \mathbf E_n}  \exp\Big\{ -\frac{c_n^2}{2} B^{-1} d_{e}^{-1}\Big\} 
		\\&\leq \sum_{n\in\mathbb N} (\# \mathbf E_n) \exp\Big\{ -\frac{c_n^2}{2} B^{-1} L^{-1}2^{n-1}\Big\} 
		\leq 2 \sum_{n\in \mathbb N} 2^{3n}\exp\Big\{ -\frac{c_n^2}{2} B^{-1} L^{-1}2^{n-1}\Big\} 
		\\&\leq 2 \sum_{n\in \mathbb N} 8^n\rho^{-n} 
		= \frac{2}{\frac{\rho}{8}-1}.
	\end{align}
	
	\emph{Step 4.}
	We show that almost surely on event $A$, we have 
		\[
		\sup_{(t,x),(t',x')\in \mathbf H} |Z_{t',x'} - Z_{t,x}| \leq 8 \sqrt{2(\log \rho) B} \sum_{n\in \mathbb N} \sqrt{n2^{-n}}.
		\]
	\[
	\sup_{(t,x),(t',x')\in \mathbf H} |Z_{t',x'} - Z_{t,x}| = \sup_{(t,x),(t',x')\in \mathbf H} Z_{t',x'} - Z_{t,x} = \sup_{(t,x),(t',x')\in \mathbf G} Z_{t',x'} - Z_{t,x},
	\]
	so we only need to verify that for any $(t,x)$ and $(t',x')$ in $\mathbf G$,
	\begin{align}
		&Z_{t',x'} - Z_{t,x} = \sum_{e\in \mathbf E_{t,x; t',x'}} Z_e
		= \sum_{n\in \mathbb N} \sum_{e\in \mathbf E_{t,x;t',x'}\cap \mathbf E_n} Z_e
		\leq 8\sum_{n\in \mathbb N} c_n
		\\& \leq  8 \sqrt{2(\log \rho) BL} \sum_{n\in \mathbb N} \sqrt{n/2^{n-1}}.
	\end{align}
	
	\emph{Final Step.} From Step 3 and Step 4 we know that for any $e^{k^2}:=\rho\geq 16$, we have 
	\begin{align}
		&\mathbf P\Big(\sup_{(t,x),(t',x')\in \mathbf H} |Z_{t',x'} - Z_{t,x}| > c^{-1/2} k \sqrt{BL}\Big)
		\leq \mathbf P(A^c) 
		\\&\leq \frac{16}{\rho-8}
		= 16\rho^{-1} (1+\frac{8}{\rho-8})
		\\&\label{eq:T.951}\leq 32 \rho^{-1},
	\end{align}
	where $c^{-1/2}:= 8 \sqrt{2}\sum_{n\in \mathbb N} \sqrt{n/2^{n-1}}$.
	Note that $32\rho^{-1}\geq 1$ when $\rho \in (1,16)$, therefore \eqref{eq:T.951} is actually true for all $\rho \geq 1$.
	Finally, taking $l = c^{-1/2}k$, we get the desired result.
	
\end{proof}
\end{note}

\begin{note}
\begin{proof}[Proof of Lemma \ref{thm:V.1}]
\emph{Step 1.} 
	For each $(t,x)$ and $(t',x')$ in $\mathbb R^+ \times \mathbb R$, define their parabolic distance as
\[
	d_{t',x'; t,x} 
	= \Big|\frac{x'-x}{\tilde L}\Big| + \Big|\frac{t-t'}{\tilde T}\Big|^{1/2}.
\]
	For each $n\in \mathbb Z$, define the parabolic dyadic lattice on $\mathbf H$, 
	\[
	\mathbf G_n 
	:= \Big\{(t,x)\in \mathbf H: \exists i\in \mathbb Z_+, j \in \mathbb Z_+, \text{ s.t. } \frac{t}{\tilde T} = \frac{i}{4^n}, \frac{x-a}{\tilde L} = \frac{j}{2^n}\Big\},
	\]
	and its edges
\[
	\mathbf E_n
	:= \Big\{e=\big((t',x'),(t,x)\big): (t',x'),(t,x)\in \mathbf G_n \text{ s.t. } d_e:=d_{t',x';t,x} = \frac{1}{2^n}\Big\}.
\]
	Define
\[
	\mathbf G
	= \bigcup_{n=1}^\infty \mathbf G_n; \quad \mathbf E = \bigcup_{n=1}^\infty \mathbf E_n.
\]
	It is elementary to see that for any $(t,x),(t',x')\in \mathbf G$ there exists $\mathbf E_{t,x;t',x'}$, a finite subset of $\mathbf E$, such that
\begin{statement}
\item \label{eq:V.51} 
	the edges in  $\mathbf E_{t,x;t',x'}$ connects $(t,x)$ and $(t',x')$; and
\item \label{eq:V.52}
	for each $n\in \mathbb N$, $ \# (\mathbf E_{t,x;t',x'} \cap \mathbf E_n) \leq 8.$  
\end{statement}
	Also for each $e\in \mathbf E$ with $e =\big ((t,x), (t',x')\big)$, we define $ Z_e := Z_{t',x'} - Z_{t,x}.$ 
	
	\emph{Step 2.}
		We show that for any $(t,x), (t',x')\in \mathbf H$ and $c\geq 0$, 
		\begin{align}
			\mathrm P(Z_{t',x'} - Z_{t,x} \geq c)  \leq \exp\Big\{ -\frac{c^2}{2} B^{-1} d_{t,x;t',x'}^{-1}\Big\}.
		\end{align}
	Note that from the exponential martingale, for any $\lambda \geq 0$, 
\begin{align} 
	& 1 = \mathbb E\Big[\exp\Big\{ \lambda (Z_{t',x'} - Z_{t,x}) - \frac{1}{2} \lambda^2 \iint_0^\infty ( g_{s,y; t,x} -  g_{s,y; t',x'} )^2 \sigma_{s,y}^2 \mathrm ds\mathrm dy\Big\}\Big] 
	\\& \geq \mathbb E\Big[\exp\Big\{ \lambda (Z_{t',x'} - Z_{t,x}) - \frac{1}{2} \lambda^2 \iint_0^\infty ( g_{s,y; t,x} -  g_{s,y; t',x'} )^2 \eta_{s,y} \mathrm ds\mathrm dy\Big\}\Big].
\end{align}
	Therefore, for any $\lambda \geq 0$,
\begin{align} 
	& \mathbb E[\exp \{ \lambda (Z_{t',x'} - Z_{t,x}) \} ] 
	\leq \exp\Big\{ \frac{1}{2} \lambda^2 \iint_0^\infty ( g_{s,y; t,x} -  g_{s,y; t',x'} )^2 \eta_{s,y} \mathrm ds\mathrm dy\Big\}
	\\& \leq \exp\Big\{ \frac{1}{2} \lambda^2 Bd_{t',x';t,x}\Big\}.
\end{align}
	By Chebyshev's inequality we have for any $\lambda \geq 0$
\begin{align}
	& \mathrm P(Z_{t',x'} - Z_{t,x} \geq c) 
	= \mathrm P[\exp\{\lambda (Z_{t',x'} - Z_{t,x} )\} \geq e^{\lambda c}]
	\\& \leq \mathbb E[\exp\{\lambda (Z_{t',x'} - Z_{t,x} )\}] / e^{\lambda c}
	\leq \exp\Big\{ \frac{1}{2} \lambda^2 Bd_{t',x';t,x} - \lambda c\Big\}.
\end{align}
	Optimizing on $\lambda$ we get the desired result.
	
\emph{Step 3.}
	We show that for any $\beta > 8$, we have
\[
	\mathrm P(A^c)
	\leq  \frac{16}{\beta-8}
\]
	where
\[
	A
	:= \bigcap_{n\in \mathbb N}\bigcap_{e\in \mathbf E_n}\{Z_{e}\leq c_n\}; 
	\quad c_n
	= \sqrt{2(\log \beta) n B2^{-n}}, \quad n\in \mathbb N.
\]
	In fact, 
\begin{align}
	&\mathrm P(A^c) \leq \sum_{n\in \mathbb N} \sum_{e\in \mathbf E_n} P(Z_e>c_n) 
	\leq \sum_{n\in \mathbb N} (\# \mathbf E_n) \exp\Big\{ -\frac{c_n^2}{2} B^{-1} d_{e}^{-1}\Big\} 
	\\&\leq 2 \sum_{n\in \mathbb N} 2^{3n}\exp\Big\{ -\frac{c_n^2}{2}  B^{-1} 2^n\Big\} 
	= 2 \sum_{n\in \mathbb N} 8^n\beta^{-n} 
	= \frac{2}{\frac{\beta}{8}-1}.
\end{align}
	
	\emph{Step 4.} We show that almost surely on event $A$, we have 
	\begin{equation}
		\sup_{(t,x),(t',x')\in \mathbf H} |Z_{t',x'} - Z_{t,x}| \leq 8 \sqrt{2(\log \beta) B} \sum_{n\in \mathbb N} \sqrt{n2^{-n}}.
		\end{equation}
	In fact, from the continuity of $Z$,
\[
	\sup_{(t,x),(t',x')\in \mathbf H} |Z_{t',x'} - Z_{t,x}|
	= \sup_{(t,x),(t',x')\in \mathbf H} Z_{t',x'} - Z_{t,x}
	= \sup_{(t,x),(t',x')\in \mathbf G} Z_{t',x'} - Z_{t,x},
\]
	Note from \eqref{eq:V.51} and \eqref{eq:V.52}, almost surely on event $A$, for any $(t,x)$ and $(t',x')$ in $\mathbf G$,
\begin{align}
	&Z_{t',x'} - Z_{t,x}
	= \sum_{e\in \mathbf E_{t,x; t',x'}} Z_e
	= \sum_{n\in \mathbb N} \sum_{e\in \mathbf E_{t,x;t',x'}\cap \mathbf E_n} Z_e
	\leq 8\sum_{n\in \mathbb N} c_n
	\\& \leq  8 \sqrt{2(\log \beta) B} \sum_{n\in \mathbb N} \sqrt{n2^{-n}}.
\end{align}
	Thus the desired result in this step follows.
	
\emph{Final Step.}
	Define $\exp(k^2):= \beta$. 
	From Steps 3 and 4, for any $\beta \geq 16$, we get 
\begin{align}
	&\mathrm P\Big(\sup_{(t,x),(t',x')\in \mathbf H} |Z_{t',x'} - Z_{t,x}| > c^{-1/2} k \sqrt{B}\Big)
	\\&\leq \mathrm P(A^c) 
	\leq \frac{16}{\beta-8}
	= 16\beta^{-1} (1+\frac{8}{\beta-8})
	\\&\label{eq:T.951}\leq 32 \beta^{-1},
\end{align}
	where $c^{-1/2}:= 8 \sqrt{2}\sum_{n\in \mathbb N} \sqrt{n2^{-n}}$.
	Note that $32\beta^{-1}\geq 1$ when $\beta \in (1,16)$, therefore \eqref{eq:T.951} actually holds for all $\beta \geq 1$.
	Finally take $z: = c^{-1/2}k$, we get the desired result since
\[
	32 \beta^{-1} = 32 e^{-k^2} = 32 e^{- c z^2} \leq 32 e^{- z^2/2^{12}}.
	\qedhere
\]
\end{proof}
\end{note}

\section{Proof of Proposition \ref{thm:B}} \label{sec:B} 
We will need the following analytical lemma.
\begin{lemma} \label{thm:B.3}
	For any $\tilde {\mathrm v}>0$ and $(t,x), (t',x')\in \mathbb R_+\times \mathbb R$ satisfying 
	\[t,t' \in [0, \tilde{\mathrm v}^{-2}]; \quad x, x' \in [-2 \tilde {\mathrm v}^{-1}, 2 \tilde {\mathrm v}^{-1}]\]
	it holds that
\[
	\iint_0^\infty (G_{s, y;t', x'} - G_{s, y;t, x})^2e^{-\tilde {\mathrm v} y}\md s\md y
	\leq 2^{7} (|x'-x| + |t'-t|^{1/2}).
\]
\end{lemma}
Note that the upper bound in the above lemma is uniform in $\tilde {\mathrm v}$.
\begin{proof}
	Let us fix an arbitrary $\tilde {\mathrm v}>0$.
	First note that 
	\begin{align} 
		&\iint_0^\infty (G_{s, y;t', x'} - G_{s, y;t, x})^2e^{-\tilde {\mathrm v} y}\md s\md y
		\\&\leq 2\iint_0^\infty (G_{s, y;t, x'} - G_{s, y;t, x})^2e^{-\tilde {\mathrm v} y}\md s\md y+ 2\iint(G_{s, y;t', x'} - G_{s, y;t, x'})^2e^{-\tilde {\mathrm v} y}\md s\md y 
		\\&=:  
		2\mathrm{I}+2\mathrm{II}, \quad (t,x), (t',x')\in \mathbb R_+\times \mathbb R. 
	\end{align}
	To finish the proof it is sufficient to show that 
	\begin{align}
		\label{wq:10_1a}
		\mathrm I
		&\leq 2^{6} |x'-x|, \quad t\in [0, \tilde{\mathrm v}^{-2}],\; x, x' \in [-2 \tilde {\mathrm v}^{-1}, 2 \tilde {\mathrm v}^{-1}],
		\\
		\label{wq:10_1b}
		\mathrm {II}
		&\leq 2^{6} |t'-t|^{1/2}, \quad t, t'\in [0, \tilde{\mathrm v}^{-2}], \;
		x' 
		\in [-2 \tilde {\mathrm v}^{-1}, 2 \tilde {\mathrm v}^{-1}].
	\end{align}
	We will prove only~\eqref{wq:10_1a}, and leave the proof of~\eqref{wq:10_1b}, which is tedious but not much different, to the reader. 
	
	To prove~\eqref{wq:10_1a} we assume without loss of generality that $z:=x' - x\geq 0$.
	Note that
	\begin{align}
		&2^{-3}\mathrm I
		\leq e^{\tilde{\mathrm v}x}\mathrm I 
		= \iint(G_{s, y;t, x'} - G_{s, y;t, x})^2e^{-\tilde{\mathrm v} (y - x)}\md s \md y
		= \iint(G_{s, y;t, z} - G_{s, y;t, 0})^2e^{-\tilde{\mathrm v} y}\md s \md y. 
	\end{align}
	From the expression of $G$ in \eqref{eq:P.3}, we have
	\begin{align}
		& 2^{-3}\mathrm {I} 
		\leq \iint_0^t \frac{1}{4\pi s}(e^{-\frac{(y-z)^2}{4s}} - e^{-\frac{y^2}{4s}})^2e^{-\mathrm vy}\md s\md y\\
		& = \iint_0^t \frac{1}{4\pi s}\big( e^{-\frac{y^2}{2s}} - 2e^{-\frac{y^2+ (y-z)^2}{4s}} + e^{-\frac{(y-z)^2}{2s}} )e^{-\mathrm vy}\md s\md y\\
		& = \int_0^t \frac{\mathrm ds}{4\pi s} \int \big( e^{-\frac{y^2}{2s} - \mathrm vy} - 2e^{-\frac{y^2}{2s} + (\frac{z}{2s} - \mathrm v) y - \frac{z^2}{4s}} + e^{-\frac{y^2}{2s} + (\frac{z}{s} - \mathrm v) y - \frac{z^2}{2s}} )\md y.
	\end{align}
	From the fact that
	\begin{equation} \label{eq:TEMP}
		\int e^{-ay^2 + by} \mathrm dy 
		=  \int e^{-a(y - \frac{b}{2a})^2 + \frac{b^2}{4a}} \mathrm dy 
		=  e^{\frac{b^2}{4a}}\int e^{-ay^2 } \mathrm dy 
		= \sqrt{\frac{\pi}{a}}e^{\frac{b^2}{4a}},
		\quad a> 0, b\in \mathbb R,
	\end{equation}
	we can get
	\begin{align}
		&2^{-3}\mathrm {I}
		\leq \int_0^t \frac{\mathrm ds}{4\pi s}  \Big( \sqrt{2s\pi} e^{\frac{s}{2} \mathrm v^2 }- 2 e^{- \frac{z^2}{4s}} \sqrt{2s\pi}e^{ \frac{s}{2}(\frac{z}{2s} - \mathrm v)^2} + e^{- \frac{z^2}{2s}}\sqrt{2s\pi}e^{ \frac{s}{2}(\frac{z}{s} - \mathrm v)^2} \Big)
		\\& \leq  \frac{1}{2}\int_0^t \frac{e^{\frac{\mathrm v^2}{2}s}}{\sqrt{2\pi}} ( 1 - 2 e^{- \frac{z^2}{8s}  - \frac{z\mathrm v}{2} } + e^{- z \mathrm v } )\frac{\mathrm ds}{\sqrt{s} } 
		\\& \leq \frac{1}{2}\int_0^{z^2 \wedge t}  ( 1 + e^{- z \mathrm v } )\frac{\mathrm ds}{\sqrt{s} }  + \frac{1}{2}\int_{z^2 \wedge t}^t  ( 2|1 -  e^{- \frac{z^2}{8s}  - \frac{z\mathrm v}{2} }| + |e^{- z \mathrm v } - 1| )\frac{\mathrm ds}{ \sqrt{s} }. 
	\end{align}
	Now using the fact that $|1- e^{-z}| \leq z$ for $z\in \mathbb R_+$, we have
	\begin{align}
		&2^{-3}\mathrm {I}
		\leq ( 1 + e^{- z \mathrm v } ) z  + \mathbf 1_{z^2 \leq t} \int_{z^2}^t  \Big(\frac{z^2}{8s}  +\mathrm vz\Big)\frac{\mathrm ds}{ \sqrt{s} } 
		\\&\leq 2z  + \int_{z^2}^\infty  \frac{z^2}{8s}  \frac{\mathrm ds}{ \sqrt{s} }+ \int_0^t  \mathrm vz\frac{\mathrm ds}{ \sqrt{s} } 
		= (2+\frac{1}{4} + 2\mathrm v\sqrt{t})z    
		\leq 2^3z.    
	\end{align}
	This gives us~\eqref{wq:10_1a}. As we have mentioned we omit the proof of~\eqref{wq:10_1b} and thus we are done.
	\begin{note}
		\emph{Step 2.} 
		show that for any $t, t'\in [0, \tilde{\mathrm v}^{-2}]$ and $x \in [-2 \tilde {\mathrm v}^{-1}, 2 \tilde {\mathrm v}^{-1}]$, it holds that 
		\[
		\mathrm {II}
		:=\iint(G_{s, y;t', x} - G_{s, y;t, x})^2e^{-\tilde {\mathrm v} y}\md s\md y
		\leq 2^{5} |t'-t|.
		\]
		To prove the result in this step, we assume without loss of generality that $h:=t' - t\geq 0$.
		Note that
		\begin{align} 
			&2^{-3}\mathrm {II}
			\leq e^{\tilde{\mathrm v}x}\mathrm {II}
			= \iint_0^\infty (G_{s, y;t', x} - G_{s, y;t, x})^2e^{-\tilde {\mathrm v} (y-x)}\md s\md y
			\\&= \iint_0^\infty (G_{s, y;t', 0} - G_{s, y;t, 0})^2e^{-\tilde {\mathrm v} y}\md s\md y.
		\end{align}
		From the expression of $G$ in \eqref{eq:P.3}, we have
		\begin{align}
			&2^{-3}\mathrm {II}\leq \iint_0^{t'} \Big( \frac{e^{-\frac{y^2}{4(t'-s)}}}{\sqrt{4\pi (t'-s)}} - \frac{e^{-\frac{y^2}{4(t-s)}}}{\sqrt{4\pi (t-s)}}\mathbf 1_{s<t} \Big)^2 e^{-\tilde {\mathrm v} y}\md s\md y
			\\&= \iint_{-h}^{t} \Big( \frac{e^{-\frac{y^2}{4(s+h)}}}{\sqrt{4\pi (s+h)}} - \frac{e^{-\frac{y^2}{4s}}}{\sqrt{4\pi s}}\mathbf 1_{0<s} \Big)^2 e^{-\tilde {\mathrm v} y}\md s\md y
			= \mathrm {III} + \mathrm {IV}
		\end{align}
		where
		\[
		\mathrm {III}
		:= \iint_{0}^{h} \Big( \frac{e^{-\frac{y^2}{4 s}}}{\sqrt{4\pi s}} \Big)^2 e^{-\tilde {\mathrm v} y} \mathrm ds \mathrm dy;
		\quad \mathrm {IV}
		:= \iint_0^t\Big( \frac{e^{-\frac{y^2}{4(s+h)}}}{\sqrt{4\pi (s+h)}} - \frac{e^{-\frac{y^2}{4s}}}{\sqrt{4\pi s}}\Big)^2 e^{-\tilde {\mathrm v} y} \mathrm ds \mathrm dy.
		\]
		Note from \eqref{eq:TEMP} we have
		\[
		\mathrm {III}
		= \int_{0}^{h} \frac{\mathrm ds}{4\pi s}\int e^{-\frac{y^2}{2 s} - \tilde{\mathrm v}y}    \mathrm dy
		= \int_{0}^{h} \frac{e^{\frac{1}{2}s \tilde {\mathrm v}^2}\mathrm ds}{2 \sqrt{2\pi s}}
		\leq \frac{e^{\frac{1}{2}h \tilde {\mathrm v}^2}}{\sqrt{2\pi}}\sqrt{h}	
		\leq \sqrt{h}.
		\]
		To bound $\mathrm {IV}$ we note that
		\begin{align} 
			& \mathrm {IV}
			= \frac{1}{4\pi}\iint_0^t    \Big( \frac{ e^{-\frac{y^2}{2(s+h)}-\tilde {\mathrm v} y} }{s+h} - \frac{ 2e^{-\frac{(2s+h)y^2}{4s(s+h)}-\tilde {\mathrm v} y} }{\sqrt{s(s+h)}  }+ \frac{   e^{-\frac{y^2}{2s}-\tilde {\mathrm v} y}}{s  }\Big)  \mathrm ds \mathrm dy
			\\&= \frac{1}{2\sqrt{2\pi}}\int_0^t    \Big( \frac{e^{\frac{1}{2}(s+h)\tilde{\mathrm v}^2}}{\sqrt{s+h}} - \frac{2 e^{\frac{s(s+h)}{2s+h} \tilde{\mathrm v}^2 }}{\sqrt{s+h/2}  }+ \frac{ e^{\frac{1}{2}s \tilde {\mathrm v}^2} }{\sqrt{s}  }\Big)  \mathrm ds 
			\\&\leq \frac{1}{2\sqrt{2\pi}}\int_0^t   e^{\frac{1}{2}s \tilde {\mathrm v}^2} \Big( \frac{e^{\frac{1}{2}h\tilde{\mathrm v}^2}}{\sqrt{s+h}} - \frac{2 }{\sqrt{s+h}  }+ \frac{ 1 }{\sqrt{s}  }\Big)  \mathrm ds 
			= \mathrm {V} + \mathrm {VI}
		\end{align}
		where
		\begin{align} 
			&\mathrm V:= \frac{1}{2\sqrt{2\pi}}\int_0^t   e^{\frac{1}{2}s \tilde {\mathrm v}^2} \frac{e^{\frac{1}{2}h\tilde{\mathrm v}^2} - 1}{\sqrt{s+h}}  \mathrm ds,
			\quad \mathrm {VI} :=\frac{1}{2\sqrt{2\pi}}\int_0^t   e^{\frac{1}{2}s \tilde {\mathrm v}^2} \Big( \frac{ 1 }{\sqrt{s}  } - \frac{1 }{\sqrt{s+h} } \Big)  \mathrm ds. 
		\end{align}
		From the fact that $e^x - 1 \leq x e^x$ for each $x\geq 0$, we get 
		\begin{align} 
			&\mathrm V \leq \frac{e^{\frac{1}{2}h\tilde{\mathrm v}^2} - 1}{2}\int_0^t   \frac{1}{\sqrt{s+h}}  \mathrm ds
			\leq \frac{1}{2}h\tilde{\mathrm v}^2e^{\frac{1}{2}h\tilde{\mathrm v}^2}  \sqrt{t'}
			\leq \sqrt{h}.
		\end{align}
		We can also verify that
		\begin{align}
			&\mathrm {VI} 
			\leq  \frac{1}{2} \int_0^t\Big( \frac{1}{\sqrt{s}} - \frac{1}{\sqrt{s+h}} \Big)\mathrm ds 
			= \frac{1}{2} \Big(\int_0^t \frac{1}{\sqrt{s}} \mathrm ds - \int_h^{t+h} \frac{1}{\sqrt{s}} \mathrm ds \Big) 
			\leq \frac{1}{2} \int_0^h \frac{1}{\sqrt{s}} \mathrm ds 
			= \sqrt{h}. 
		\end{align}
		To sum up, we have
		\begin{equation} 
			2^{-3}\mathrm {II}
			\leq \mathrm {III} + \mathrm {IV}
			\leq \mathrm {III} + \mathrm {V} + \mathrm {VI}
			\leq 3 \sqrt{h}
		\end{equation}
		which implies the desired result in this step.
	\end{note}
	\begin{note}
		\emph{Final Step.} From Steps 1 and 2, we have
		\begin{align} 
			&\iint_0^\infty (G_{s, y;t', x'} - G_{s, y;t, x})^2e^{-\tilde {\mathrm v} y}\md s\md y
			\\&\leq 2\iint_0^\infty (G_{s, y;t', x'} - G_{s, y;t', x})^2e^{-\tilde {\mathrm v} y}\md s\md y+ 2\iint(G_{s, y;t', x} - G_{s, y;t, x})^2e^{-\tilde {\mathrm v} y}\md s\md y
			\\&\leq 2^7 (|x'-x|+ |t'-t|^{1/2}).
			\qedhere
		\end{align}
	\end{note}
\end{proof}
Let us now give the proof of Proposition \ref{thm:B}.
\begin{proof}[Proof of Proposition \ref{thm:B}]
\emph{Step 1.}
it is easy to see that
on the event $\{\tau_1 \geq T, \tau_3 \geq T\}$, 
the following holds almost surely: for each $(s,y)\in [0, \tau_2 \wedge T]\times \mathbb R$,
\begin{align}
	&v_{s,y} 
	\leq F(s - \mathrm vy) +  \varepsilon k L e^{-\theta \mathrm v (x - \mathrm vt)} \mathbf 1_{x\leq \mathrm vt}
	\leq 2k \varepsilon L e^{-\theta \mathrm v (y - \mathrm vs)}, 
      \\& f^w_{s,y} 
      \leq w_{s,y}^p \mathbf 1_{y\in [ - L,  \mathrm vT + L], w_{s,y}\leq \nu \varepsilon L}
      \leq  (\nu \varepsilon L)^p,
	\\&\sigma^w_{s,y} 
      = \sigma(v_{s,y}+w_{s,y})^2 - \sigma(v_{s,y})^2 
      \leq w_{s,y}
      \leq \nu \varepsilon L \mathbf 1_{y\in [-L, \mathrm vT + L]}.
\end{align}

\emph{Step 2.}
Note that $v+w$ admits the following mild form (c.f. \cite[Theorem 2.1]{Shiga1994Two}): 
\begin{equation}
\begin{multlined}
	v_{t,x} + w_{t,x} 
	= \iint_0^t  G_{s,y;t,x} \Big( v_{0,y}  \mathrm dy\delta_0(\mathrm ds)+ \big(f(v_{s,y}) + f^w_{s,y}\big)\mathrm ds\mathrm dy  + {}
	\\ \epsilon \sigma(v_{s,y}) W^v(\mathrm ds\mathrm dy) + \epsilon \sigma^w_{s,y} W^w(\mathrm ds\mathrm dy) \Big), 
	\quad {\rm a.s.} \quad \forall (t,x)\in (0,\infty)\times \mathbb R. 
\end{multlined}	
\end{equation}
Therefore, almost surely on the event $\{\tau_1 \geq T, \tau_3 \geq T\}$, we have $v + w = \tilde u $ on $[0,\tau_2\wedge T]\times \mathbb R$. 
	  Here, $\tilde u := \sum_{i=1}^5 Z^{(i)}$ where $\{Z^{(i)}: i = 1,\dots, 5\}$ is a list of continuous random fields defined 
	  so that for each $(t,x)\in \mathbb R_+\times \mathbb R$, 
\begin{align}
	& Z^{(1)}_{t,x}
	= 
	\mathbf 1_{t = 0}\tilde F(x)+ \mathbf 1_{t>0}\int  G_{0,y;t,x}  \tilde F(y)  \mathrm dy,
     \\ & Z^{(2)}_{t,x} 
	= 	
     \iint_0^t G_{s,y;t,x} f\big(v_{s,y} \wedge (2k \varepsilon L e^{-\theta \mathrm v (y - \mathrm vs)} ) \big) \mathrm ds\mathrm dy,
	\quad {\rm a.s.} 
     \\ & Z^{(3)}_{t,x}
     = 
     \iint_0^tG_{s, y;t, x} \Big(f^w_{s,y} \wedge \big((\nu \varepsilon L)^p \mathbf 1_{y \in [-L, \mathrm v T + L]}\big)\Big)\md s\md y,
	\quad {\rm a.s.} 
     \\ & Z^{(4)}_{t,x} 
     = 
     \epsilon \iint_0^t G_{s,y; t,x} \sigma\big(v_{s,x}\wedge (2k \varepsilon L e^{-\theta \mathrm v (y - \mathrm vs)} )\big) W^v(\mathrm ds\mathrm dx),
	\quad {\rm a.s.} 
      \\& Z^{(5)}_{t,x}
      = 
      \epsilon\iint_0^tG_{s,y; t,x} \big(\sigma^w_{s,y} \wedge (\nu \varepsilon L \mathbf 1_{y\in [-L, \mathrm vT + L]}) \big) 
      W^w(\mathrm ds\mathrm dx), 
	\quad {\rm a.s.} 
\end{align}

\emph{Step 3.}
Clearly, $\tau_2 = \tilde \tau_2$ holds almost surely on the event $\{\tau_1\geq T, \tau_3 \geq T\}$ where
$
	\tilde \tau_2
	:= \inf \{t \in [0,T]: \tilde u_{t,x} \geq \nu \varepsilon L \text{ for some } x\in [-L, \mathrm vT + L]\}.
$

\emph{Step 4.}
	We will show that
\[
	\sup_{(t,x)\in \mathbf H} Z^{(1)}_{t,x} \leq \nu_1\varepsilon L.
\]
	where $\nu_1 := 2^3$ and  $\mathbf H :=[0,T]\times [-L, \mathrm vT + L].$ 
Note that from \eqref{eq:U.5}, \eqref{eq:P.3} and \eqref{eq:U.3}, for any $(t,x)\in \mathbf H$, we have
\begin{align} 
	&Z^{(1)}_{t,x} 
	\leq \int G_{0,y;t,x} \frac{\varepsilon}{\theta \mathrm v} e^{-\theta \mathrm v y} \mathrm dy
	= \frac{\varepsilon}{\theta \mathrm v} e^{-\theta \mathrm v x}\int \frac{e^{-\frac{y^2}{4t} + \theta \mathrm vy}}{\sqrt{4\pi t}}  \mathrm dy
	= \frac{\varepsilon}{\theta \mathrm v} e^{-\theta \mathrm v x}e^{\theta^2\mathrm v^2 t}
	\\&\leq \theta^{-1}e^{\theta^2+\theta} \varepsilon L 
	\leq \nu_1 \varepsilon L.
\end{align}

\emph{Step 5.}
	We will show that
\[
\sup_{(t,x)\in \mathbf H} Z^{(2)}_{t,x} \leq \nu_2 \varepsilon 
L,  
\quad \text{a.s.}
\]
where $ \nu_2 := 2^4k.$
Note that from \eqref{eq:U.1} that $\kappa^{p-1}\leq 1$, \eqref{eq:P.3} and \eqref{eq:U.3} we can verify that,
\begin{align} 
	&Z^{(2)}_{t,x}
	\leq \iint_0^t G_{s,y;t,x} (2k \varepsilon L e^{-\theta \mathrm v (y - \mathrm vs)} )^p  \mathrm ds\mathrm dy 
      \leq (2k \varepsilon L)^p\int_0^t e^{p \theta \mathrm v^2 s} \mathrm ds \int G_{s,y;t,x} e^{-p \theta \mathrm v y}   \mathrm dy 
      \\&=(2k \varepsilon L)^p e^{-p\theta \mathrm v x}e^{p^2\theta^2 \mathrm v^2 t}\int_0^t e^{(p\theta \mathrm v^2- p^2\theta^2 \mathrm v^2) s} \mathrm ds  
      \leq (2k \varepsilon L)^p e^{-p\theta \mathrm v x} t e^{p\theta \mathrm v^2 t}   
      \\&\leq (2k \varepsilon L)^p e^{p\theta \mathrm v L} T e^{p\theta \mathrm v^2 T}   
      = 2^p k^p e^{2 p\theta } \kappa^{p-1} \varepsilon L    
      \leq 2^4 k \varepsilon L    
      = \nu_2 \varepsilon L, \quad \forall (t,x) \in
      \mathbf H, \quad {\rm a.s.}     
\end{align}

\emph{Step 6.}
	We will show that
\[
\sup_{(t,x)\in \mathbf H} Z^{(3)}_{t,x} \leq \nu_3\varepsilon 
L,  
\quad \text{a.s.}
\]
where $\nu_3 := \nu^p$.
In fact, from \eqref{eq:P.3} and \eqref{eq:U.3} we can verify that 
\begin{align} 
	&Z^{(3)}_{t,x}
	\leq \iint_0^t G_{s,y;t,x} (\nu \varepsilon L)^p \mathrm ds\mathrm dy 
      = (\nu \varepsilon L)^p\int_0^t \mathrm ds \int G_{s,y;t,x} \mathrm dy 
      \\& 
      \leq T(\nu \varepsilon L)^p
      = \nu^p \kappa^{p-1} \varepsilon L
      \leq \nu_3 \varepsilon L,  \quad \forall (t,x) \in 
      \mathbf H, \quad {\rm a.s.}
\end{align}

\emph{Step 7.}
	We will show that
\[
	\mathrm P\Big( \sup_{(t,x)\in \mathbf H} Z^{(4)}_{t,x} > \nu_4\varepsilon L \Big) \leq 2^{-4}
\]
where $\nu_4:=2^{13} \mathcal K^{1/2} \gamma^{-1}$. 
First note that almost surely for each $(s,y)\in \mathbb R_+\times \mathbb R$, 
\[
	\epsilon^2 \sigma\big(v_{s,y}\wedge (2k \varepsilon L e^{-\theta \mathrm v (y - \mathrm vs)} )\big)^2 
	\leq 2 \epsilon^2 k \varepsilon L e^{-\theta \mathrm v (y - \mathrm vs)}\mathbf 1_{y\leq \mathrm vs} 
	=: \eta^{(4)}_{s,y}.
\]
Then note that for each $(t,x), (t',x') \in \mathbf H$, using Lemma \ref{thm:B.3}, 
\begin{align}
	&\frac{1}{2\epsilon^2 k \varepsilon L}\iint_0^\infty (G_{s,y;t',x'} -G_{s,y;t,x})^2 \eta^{(4)}_{s,y} \mathrm ds\mathrm dy	
	\\&\leq \iint_0^\infty (G_{s,y;t',x'} -G_{s,y;t,x})^2 e^{-\mathrm v(y - \mathrm vs)} \mathrm ds\mathrm dy
	\leq 2^7 (|x'-x| + |t'-t|^{1/2}). 
\end{align}
Therefore,
\begin{align} 
	B^{(4)} 
	:= \sup_{(t,x),(t',x')\in \mathbf H} \frac{\iint_0^\infty (G_{s,y;t',x'} - G_{s,y;t,x})^2 \eta^{(4)}_{s,y} \mathrm ds\mathrm dy}{\big| \frac{x'-x}{3L}\big| + \big|\frac{t'-t}{T}\big|^{1/2}}  
	\leq 2^{10} \epsilon^2 k \varepsilon L^2
	=: \tilde B^{(4)}.
\end{align}
Taking $z = 2^8$, we get from Lemma \ref{thm:V.1} that
\begin{equation} 
	\mathrm P\Big(\sup_{(t,x)\in \mathbf H} Z^{(4)}_{t,x} > z \sqrt{\tilde B^{(4)}}\Big) 
	\leq \mathrm P\Big(\sup_{(t,x),(t',x')\in \mathbf H} |H^{(4)}_{t',x'}- H^{(4)}_{t,x}| > z \sqrt{B^{(4)}}\Big) 
	\leq 2^5 e^{- z^2/2^{12}}
	\leq 2^{-4}.
\end{equation}
To finish this step we note that
\begin{equation} 
	z \sqrt{\tilde B^{(4)}}
	= 2^{13} \sqrt{\epsilon^2 k \varepsilon L^2}
	= \nu_4 \varepsilon L.
\end{equation}

\emph{Step 8.}
	We will show that
\[
	\mathrm P\Big( \sup_{(t,x)\in \mathbf H} Z^{(5)}_{t,x} > \nu_5\varepsilon L \Big) \leq 2^{-4}
\]
with $\nu_5 = \nu/4$.
First note that almost surely for each $(s,y)\in \mathbb R_+\times \mathbb R$, 
\[
	\epsilon^2 \big(\sigma^w_{s,y} \wedge (\nu \varepsilon L \mathbf 1_{y\in [-L, \mathrm vT + L]}) \big)^2 
	\leq \epsilon^2 \nu^2 \varepsilon^2 L^2
	=: \eta^{(5)}_{s,y}.
\]
Then note that for each $(t,x), (t',x') \in \mathbf H$, using \cite[Lemma 6.2(1)]{Shiga1994Two} (c.f. Lemma \ref{thm:B.3}), 
\begin{align}
	\frac{1}{\epsilon^2 \nu^2 \varepsilon^2 L^2}\iint_0^\infty (G_{s,y;t',x'} -G_{s,y;t,x})^2 \eta^{(5)}_{s,y} \mathrm ds\mathrm dy	
	\leq 2^7 (|x'-x| + |t'-t|^{1/2}). 
\end{align}
Therefore,
\begin{align} 
	&B^{(5)} 
	:= \sup_{(t,x),(t',x')\in \mathbf H} \frac{\iint_0^\infty (G_{s,y;t',x'} - G_{s,y;t,x})^2 \eta^{(5)}_{s,y} \mathrm ds\mathrm dy}{\big| \frac{x'-x}{3L}\big| + \big|\frac{t'-t}{T}\big|^{1/2}}  
	\leq 2^9 \epsilon^2 \nu^2 \varepsilon^2 L^3
	=: \tilde B^{(5)}.
\end{align}
Taking $z = 2^8$, we get from Lemma \ref{thm:V.1} that
\begin{equation} 
	\mathrm P\Big(\sup_{(t,x)\in \mathbf H} Z^{(5)}_{t,x} > z \sqrt{\tilde B^{(5)}}\Big) 
	\leq \mathrm P\Big(\sup_{(t,x),(t',x')\in \mathbf H} |Z^{(5)}_{t',x'}- Z^{(5)}_{t,x}| > z \sqrt{B^{(5)}}\Big) 
	\leq 2^5 e^{- z^2/2^{12}}
	\leq 2^{-4}.
\end{equation}
To finish this step we note from \eqref{eq:U.4} that
\begin{equation} 
	z \sqrt{\tilde B^{(5)}}
	= 2^{13} \sqrt{\epsilon^2 \nu^2 \varepsilon^2 L^3}
	= 2^{13} \sqrt{\epsilon^2 L}\nu  \varepsilon L
	\leq \nu_5 \varepsilon L.
\end{equation}

\emph{Final step.}
	We note from \eqref{eq:U.2} and \eqref{eq:U.6} that
\begin{equation} 
	 \sum_{i=1}^5 \nu_i = 2^3 + 2^4 k + \nu^p + 2^{13} \mathcal K^{1/2} \gamma^{-1} + \nu/4 
	 \leq \nu.
\end{equation}
	Also note from Steps 2 and 8 that
\begin{align} 
	&\mathrm P(\tau_2 < T, \tau_1\geq T, \tau_3 \geq T) 
	= \mathrm P(\tilde \tau_2 < T, \tau_1 \geq T, \tau_3 \geq T)
	\leq \mathrm P(\tilde \tau_2 < T)
      \\&= \mathrm P(\{\tilde u \leq \nu \varepsilon L \text{ on }[0,T]\times [-L, \mathrm vT + L]\}^c)
      \\&\leq \mathrm P\Big(\bigcup_{i=1}^5 \{ Z^{(i)} \leq  \nu_i \varepsilon L \text{ on }[0,T]\times [-L, \mathrm vT + L]\}^c\Big).
\end{align}
	Now from Steps 3-7, we have
\begin{align} 
	&\mathrm P(\tau_2 < T, \tau_1\geq T, \tau_3 \geq T) 
      \\&\leq \sum_{i=1}^5 \mathrm P( \{ Z^{(i)} \leq  \nu_i \varepsilon L \text{ on }[0,T]\times [-L, \mathrm vT + L]\}^c)
      \leq 2^{-3}.
      \qedhere
\end{align}
\end{proof}

\subsection*{Acknowledgments} 
The work of the authors  was supported in part by ISF grants No.~1704/18 and No.~1985/22.  The first author is a Zuckerman Postdoctoral Scholar, and this work was supported in part by the Zuckerman STEM Leadership Program.
%The third author is supported in part at the Technion by a fellowship of the Israel Council for Higher Education. 
Most of this research was done while the third author was a Postdoc at the Technion—Israel Institute of Technology, supported in part by a fellowship of the Israel Council for Higher Education. %
We thank Eyal Neuman, Zenghu Li and Hugo Panzo for very helpful conversations.
We are grateful to Lenya Ryzhik for
generously sharing his deep understanding of various aspects of the FKPP equations. 
We also thank the referees for the helpful comments and suggestions. %

\begin{note}
Let $(\Omega, \mathcal G, \mathcal F = (\mathcal F_t)_{t\geq 0}, \mathrm P)$ be a \emph{usual probability space}, i.e. (1) $(\Omega, \mathcal G, P)$ is a complete probability space; (2) $(\mathcal F_t)_{t\geq 0}$ is a family of $\sigma$-field on $\Omega$  such that for each $0\leq s\leq t<\infty$, $\mathcal N:= \{A\in \mathcal G: P(A) = 0\} \subset \mathcal F_s\subset \mathcal F_t \subset \mathcal G$.
Denote by $\mathscr M_{\rm loc}$ the space of continuous local martingales.
For any continuous semi-martingale $M$, denote by $\langle M \rangle$ its quadratic variation. 
For any continuous semi-martingale $M, N$, denote by $\langle M, N\rangle$ their quadratic covariation.
For more details about those concept, we refer our readers to \cite[Chapter 17]{Kallenberg2002Foundations}.

We say $f$ is a \emph{random field} if it is a real-valued stochastic process indexed by $\mathbb R_+\times \mathbb R$.
We say a random field $f$ is progressive if it is $\mathcal P \times \mathcal B(\mathbb R)$-measurable where $\mathcal P$ is the progressive $\sigma$-field on $\Omega \times [0,\infty)$ with respect to the filtration $\mathcal F$.
We say $f\in \mathscr L^2_{\rm loc}$ if it is a progressive random field on $\Omega$ such that
\[
\iint_0^t f^2_{s,y} \mathrm ds\mathrm dy < \infty, \quad t\geq 0, {\rm a.s.}
\]

Let $\mathcal B_F(\mathbb R)$ be the collection of Borel subsets of $\mathbb R$ whose Lebesgue measure are finite.
We say $\{(W_s(A))_{A \in \mathcal B_F(\mathbb R),s\in \mathbb R_+}; \mathrm P\}$ is an $\mathcal F$-white noise if it is a stochastic process indexed by $\mathcal B_F(\mathbb R)\times \mathbb R_+$ such that 
(1)  for disjoint $A,B\in \mathcal B_F(\mathbb R^d)$, almost surely, for each $t\geq 0$, $W_t(A\cup B) = W_t(A) + W_t(B)$; and
(2) for any $A \in \mathcal B_F(\mathbb R)$, $t\mapsto W_t(A)$ is a Brownian motion such that for each $t\geq 0$, $\langle W(A) \rangle_t = t \operatorname{Leb}(A)$. 

For an $\mathcal F$-white noise $W$, we say
\begin{align} \label{eq:WSI.4}
 f \mapsto \iint_0^\cdot f_{s,y}W(\mathrm ds\mathrm dy) 
\end{align}
 is \emph{Walsh's stochastic integral for $W$} if it is a map from $\mathscr L_{\rm loc}^2$ to $\mathscr M_{\rm loc}$ such that (1) for any $f, g \in \mathscr L_{{\rm loc}}^2$ we have
 \begin{align} 
 &  \iint_0^t (f_{s,y} + g_{s,y}) W(\mathrm ds\mathrm dy) 
 = \iint_0^t f_{s,y} W(\mathrm ds\mathrm dy) +  \iint_0^t g_{s,y} W(\mathrm ds\mathrm dy) , 
 \quad t\geq 0, {\rm a.s.};
 \end{align}  
(2) for any $f\in \mathscr L_{{\rm loc}}^2$,  
\begin{align}
& \Big\langle \iint_0^\cdot f_{s,y}W(\mathrm ds\mathrm dy) \Big\rangle_t = \iint_0^t f^2_{s,y}\mathrm ds\mathrm dy, 
\quad t\geq 0, {\rm a.s.};
\end{align}
and (3) for any $A \in \mathcal B_F(\mathbb R)$ and real-valued progressive process $h$ with 
\begin{align} 
& \int_0^t h_s^2 \mathrm ds < \infty, \quad t\geq 0, {\rm a.s.} 
\end{align}
it holds that 
\begin{equation} \label{eq:WSI.5}
\iint_0^t h_s \mathbf 1_{y\in A} W(\mathrm ds\mathrm dy) = \int_0^t h_s \mathrm d W_s(A), \quad t\geq 0, {\rm a.s.}
\end{equation}
where the right hand side of \eqref{eq:WSI.5} is Ito's stochastic integral for Brownian motion. 
For any $\mathcal F$-white noise $W$, using the argument in \cite[Section II.5]{Perkins2002Dawson-Watanabe}, it can be verified that this map \eqref{eq:WSI.4}, the Walsh's stochastic integral for $W$, exists and is unique.

\begin{lemma} 
	Suppose that $W$ is a $\mathcal F$-white noise and
\[ f\in \mathscr D^\mathrm P:= \bigg\{h\in \mathscr L_{\rm loc}^2: \mathrm P\Big[\exp\Big\{\frac{1}{2} \iint_0^\infty h_{s,y}^2\mathrm ds\mathrm dy\Big\}\Big]< \infty\bigg\}
\]\
	Then 
\[
	\mathcal E_t^f
	:= \exp\Big\{\iint_0^t f_{s,y}W(\mathrm ds\mathrm dy) - \frac{1}{2} \iint_0^t f_{s,y}^2 \mathrm ds\mathrm dy\Big\}, 
	\quad t\geq 0
\]
	is a uniformly integrable martingale. 
	Moreover, under probability $\mathrm Q$, which is given by $\mathrm d\mathrm Q =  \mathcal E_\infty^f \mathrm d \mathrm P$, the map 
\begin{align} \label{eq:WSI.61}
& g\in \mathscr L_{\rm loc}^2 \mapsto \iint_0^\cdot g_{s,y}W(\mathrm ds\mathrm dy)  - \iint_0^\cdot f_{s,y} g_{s,y} \mathrm ds\mathrm dy
\end{align}
	is a Walsh's stochastic integral for the $\mathcal F$-white noise 
\begin{align} \label{eq:WSI.62}
& W_t(A) - \iint_0^t \mathbf 1_{y\in A}f_{s,y}\mathrm ds\mathrm dy, \quad t\geq 0, A \in \mathcal B_F(\mathbb R).  
\end{align} 
\end{lemma}
\begin{proof}
	According to \cite[Theorem 18.23]{Kallenberg2002Foundations}, we know that $(\mathcal E^f_t)_{t\geq 0}$ is a uniformly integrable martingale (under probability $\mathrm P$).
	Therefore the probability $\mathrm Q$ is well-defined.
	Fix an arbitrary $g\in \mathscr L_{\mathrm{loc}}^2$. 
	Let us write
	\begin{equation} 
	 \widetilde W_t(g):= \iint_0^t g_{s,y} W(\mathrm ds\mathrm dy) - \iint_0^t f_{s,y}g_{s,y} \mathrm ds\mathrm dy, \quad t\geq 0.
	\end{equation} 
	Then, it can be verified from \cite[Theorem 18.19~\& Lemma 18.21]{Kallenberg2002Foundations} that 
\begin{equation} \label{eq:WSI.63}
	\widetilde W(g) \in \mathscr M^\mathrm Q_{\rm loc},
\end{equation} 
	where $\mathscr M_{\rm loc}^\mathrm Q$ is the space of continuous local martingales with respect to probability $\mathrm Q$. 
	It can also be verified from \cite[Theorem 18.20]{Kallenberg2002Foundations} that 
\begin{equation}\label{eq:WSI.64}
	\langle \widetilde W(g) \rangle_t^\mathrm Q = \iint_0^t  g_{s,y}\mathrm ds\mathrm dy, \quad t\geq 0,
\end{equation}
	where $\langle \cdot \rangle^\mathrm Q$ stands for the qudratic variation of semimartingales with respect to probability $\mathrm Q$.
	Now with the help of \eqref{eq:WSI.63} and \eqref{eq:WSI.64}, we can verify that \eqref{eq:WSI.62} is indeed an $\mathcal F$-white noise (under probability $\mathrm Q$), and \eqref{eq:WSI.61} is the corresponding Walsh's stochastic integral.
	We omit the details.
\end{proof}
\end{note}
\begin{note}
For $s\geq 0$ and $y\in \mathbb R$, let $\{(B_t)_{t\geq s};\Pi_{s,y}\}$ be a Brownian motion with generator $\Delta$ initiated at time $s$ and position $y$.
In other word
\[
\int G_{s,y;t,x} \varphi_x\mathrm dx = \Pi_{s,y}[\varphi_{B_t}], \quad 0\leq s<t<\infty, y \in \mathbb R, \varphi \in \mathcal B_b(\mathbb R),
\] 
where
\begin{align} \label{eq:SPDE.02}
	G_{s,y;t,x} := \frac{e^{-\frac{(x-y)^2}{4(t-s)}}}{\sqrt{4\pi (t-s)}}\mathbf 1_{0\leq s < t< \infty}, \quad (s,y),(t,x) \in \mathbb R_+\times \mathbb R.
\end{align}
Let $(\Omega, \mathcal G, \mathcal F = (\mathcal F_t)_{t\geq 0}, \mathrm P)$ be a usual probability space.
We say $\sigma \in \mathscr P_\mathrm e$ if $\sigma$ is a real-valued progressive random field on $\Omega$ such that, almost surely, for each $T\geq 0$, there exists $c>0$ such that for each $0\leq t\leq T$ and $x\in \mathbb R$, $|\sigma_{t,x}| < ce^{c|x|}$.
According to \cite[Lemma 6.2(ii)]{Shiga1994Two}, we have for any $c \in \mathbb R$ and $T>0$,
\begin{align} \label{eq:SPDE.03}
	\sup_{0\leq s<t\leq T} \sup_{x\in \mathbb R} e^{-c |x|}\int_\mathbb R G_{s,y;t,x} e^{c |y|}\mathrm dy
	< \infty.
\end{align}

\begin{lemma} \label{thm:SC.4}
	Suppose that $\sigma \in \mathscr P_\mathrm e$.
	Then, the following three random fields
\begin{align} 
	&\Big(\int G_{0,y;t,x}\sigma_{0,y}\mathrm dy\Big)_{(t,x)\in \mathbb R_+\times \mathbb R}, 
	\qquad \Big( \iint_0^t G_{s,y;t,x} \sigma_{s,y}\mathrm ds\mathrm dy\Big)_{(t,x)\in \mathbb R_+\times \mathbb R}, 
	\\\text{and}\quad 
	&\Big(\iint_0^t  G_{s,y;t,x}^2 \sigma_{s,y} \mathrm ds \mathrm dy\Big)_{(t,x)\in \mathbb R_+\times \mathbb R}.
\end{align}
	are all in $\mathscr P_{\mathrm e}$.
\end{lemma}
\begin{proof}
	Fix $T>0$.
	Since $\sigma \in \mathscr P_\mathrm e$ we know that almost surely there exists $c>0$ such that $|\sigma_{s,y}| < ce^{c|y|}$ for $s\in [0,T]$ and $y\in \mathbb R$.
	Therefore, almost surely there exists a $C_0>0$ such that for all $0\leq s< t\leq T$ and $x \in \mathbb R$,
\begin{align}
	\int G_{s,y;t,x}\sigma_{s,y}\mathrm dy
	\leq c\int G_{s,y;t,x} e^{c|y|}\mathrm dy
	\overset{\eqref{eq:SPDE.03}}\leq C_0 e^{c|x|}.
\end{align}
	Also, almost surely for all $t\in [0,T]$ and $x\in \mathbb R$,
\begin{align}
	\iint_0^t G_{s,y;t,x} \sigma_{s,y}\mathrm ds\mathrm dy 
	\leq \int_0^t C_0 e^{c|x|} \mathrm ds \leq T C_0 e^{c|x|}
\end{align}
	and
\begin{align}
	&\iint_0^t G_{s,y;t,x}^2 \sigma_{s,y}\mathrm ds\mathrm dy 
	\overset{\eqref{eq:SPDE.02}}\leq \iint_0^t \frac{G_{s,y;t,x}}{\sqrt{4\pi (t-s)}} \sigma_{s,y}\mathrm ds\mathrm dy 
	\\&\leq \int_0^t \frac{C_0 e^{c|x|}}{\sqrt{4\pi(t-s)}} \mathrm ds 
	\leq \sqrt{\frac{T}{\pi}} C_0 e^{c|x|}.
	\qedhere
\end{align}
\end{proof}

Now let $W$ be an $\mathcal F$-white noise. 
Let $f, \sigma\in \mathscr P_\mathrm e$. 
We say $u$ is a \emph{(mild) solution to SPDE}
\begin{equation} \label{eq:SPDE.1}
	\partial_t u_{t,x} = \Delta_x u_{t,x} + f_{t,x} + \sigma_{t,x} \dot W_{t,x}, \quad t\geq 0, x\in \mathbb R,
 \end{equation}
if $u \in \mathscr P_\mathrm e$ is continuous and satisfies that for each $t> 0$ and $x\in \mathbb R$, almost surely
\begin{align} \label{eq:SPDE.15}
u_{t,x} 
= \iint_0^t G_{s,y;t,x} M^u(\mathrm ds\mathrm dy)
\end{align}
where
\[
	M^u(\mathrm ds\mathrm dy) 
	= u_{0,y} \delta_0(\mathrm ds) \mathrm dy + f_{s,y}\mathrm ds\mathrm dy + \sigma_{s,y}W(\mathrm ds\mathrm dy).
\]
Note that from Lemma \ref{thm:SC.4} the right hand side of \eqref{eq:SPDE.15} is always well-defined.

\begin{lemma} \label{thm:SC.45}
	Suppose that $u$ is a solution to SPDE \eqref{eq:SPDE.1}. 
	Then for each $\phi \in \cC_{\rm c}^{1,2}([0,\infty)\times \mathbb R)$, almost surely we have
\begin{align}
	\int \phi_{t,y} u_{t,y} \mathrm dy
	= \iint_0^t (\partial_s \phi_{s,x} +\partial_x^2 \phi_{s,x})u_{s,x}\mathrm ds\mathrm dx + \iint_0^t \phi_{s,x} M^u(\mathrm ds\mathrm dx).
\end{align}
\end{lemma}

We omit the proof of Lemma \ref{thm:SC.45} since it is similar to the Lemma \ref{thm:SC.5} below.

	Let us fix a $\mathrm v > 0$.
	Let $\tau_s := \inf\{t \geq s : B_t \geq \mathrm vt \}$. 
	It can be verified that for each $0\leq s< t< \infty$ and $y < \mathrm vs$ there exists a (unique) continuous map $x \mapsto G^{(\mathrm v)}_{s,y;t,x}$ from $(-\infty,\mathrm v t)$ to $(0,\infty)$ such that
\begin{equation} \label{eq:SPDE.155}
	\int_{-\infty}^{\mathrm vt} G^{(\mathrm v)}_{s,y;t,x} \varphi(x)\mathrm dx = \Pi_{s,y}[\varphi( B_t); t < \tau_s], \quad \varphi \in \mathcal B_b(\mathbb R).
\end{equation}
	Let us also define that $G^{\mathrm v}_{s,y;t,x} = 0$ on $\{(s,y;t,x): 0\leq s< t, y < \mathrm vs, x < \mathrm vt\}^c$.
	It can be verified that
\begin{align} \label{eq:SPDE.16}
	0 \leq G^{\mathrm v}_{s,y;t,x} \leq G_{s,y;t,x}, \quad (s,y),(t,x)\in \mathbb R_+\times \mathbb R.
\end{align}

	Let $W$ be an $\mathcal F$-white noise. Suppose that $f,\sigma \in \mathscr P_\mathrm e$ are supported on $\{(t,x) \in \mathbb R_+\times \mathbb R: x<\mathrm vt\}$.
	We say $v$ is a \emph{solution to SPDE}
\begin{equation} \label{eq:SPDE.2}
\begin{cases}
\partial_t v_{t,x} = \partial_x^2 v_{t,x} + f_{t,x} + \sigma_{t,x} \dot W_{t,x}, &\quad t\geq 0, x < \mathrm vt,\\
v_{t,x} = 0, &\quad t\geq 0, x\geq \mathrm vt,
\end{cases}
\end{equation}
if $v\in \mathscr P_\mathrm e$ is a continuous, is supported on $\{(t,x) \in \mathbb R_+\times \mathbb R: x<\mathrm vt\}$, and satisfies that for each $t> 0$ and $x\in \mathbb R$, almost surely 
\begin{align} \label{eq:SPDE.21}
&  v_{t,x} 
	= \iint_0^t G^{(\mathrm v)}_{s,y;t,x} M^v(\mathrm ds\mathrm dy),
\end{align}
	where 
$
 M^v(\mathrm ds\mathrm dy) = v_{0,y} \delta_0(\mathrm ds) \mathrm dy + f_{s,y}\mathrm ds\mathrm dy + \sigma_{s,y}W(\mathrm ds\mathrm dy).
$
Note that from Lemma \ref{thm:SC.4} and \eqref{eq:SPDE.16} the right hand side of \eqref{eq:SPDE.21} is always well-defined.

\begin{lemma} \label{thm:SC.5}
	Suppose that $v$ is a solution to SPDE \eqref{eq:SPDE.2}. 
	Then for each $\phi \in \cC_{\mathrm c}^{\infty}([0,\infty)\times \mathbb R)$ and $t\geq 0$, almost surely,
\begin{equation} \label{eq:SPDE.25}
	\begin{multlined}
		\int \phi_{t,y} v_{t,y} \mathrm dy
		= \iint_0^t \phi_{s,y}M^v(\mathrm ds\mathrm dy) + \iint_0^t  (\partial_r\phi_{t,x}+\partial_x^2\phi_{r,x}) v_{r,x} \mathrm dr\mathrm dx 
		\\ - \iint_0^t \Pi_{s,y}[\phi_{\tau, B_\tau}; t\geq \tau] M^v(\mathrm ds\mathrm dy)
	\end{multlined}
\end{equation}
\end{lemma}
\begin{proof}
\emph{Step 1.} 
	We show that for any $t\geq 0$ almost surely,
\begin{equation} \label{eq:SPDE.29}
\int v_{t,x} \phi_{t,x} \mathrm dx 
=   \iint_0^t M^v(\mathrm ds\mathrm dy) \int G^{(\mathrm v)}_{s,y;t,x} \phi_{t,x} \mathrm dx. 
\end{equation}
	In fact, since $\phi \in \cC_c^\infty([0,\infty)\times \mathbb R)$, there exists a $K>0$ and $T>0$ such that $\phi$ is supported on $[0,T]\times [-K,K]$.
	From $\sigma \in \mathscr P_\mathrm e$ we know $\sigma^2\in \mathscr P_\mathrm e$.
	Now we can verify that almost surely
\begin{align}
	&\int \mathrm dx \iint_0^t  (\phi_{t,x} G^{(\mathrm v)}_{s,y;t,x} \sigma_{s,y})^2 \mathrm ds\mathrm dy
	\leq \|\phi\|_\infty^2\int_{-K}^K \mathrm dx \iint_0^t (G_{s,y;t,x}^{(\mathrm v)} \sigma_{s,y})^2 \mathrm ds\mathrm dy
	\\&\overset{\eqref{eq:SPDE.16}}\leq \|\phi\|_\infty^2\int_{-K}^K \mathrm dx \iint_0^t G_{s,y;t,x}^2 \sigma_{s,y}^2 \mathrm ds\mathrm dy
	\overset{\text{Lemma \ref{thm:SC.4}}}<\infty.
\end{align}
	This allows us to use the stochastic Fubini theorem (cf. \cite[Lemma 2.4]{Iwata1987AnInfinite}) and conclude that almost surely
\begin{align}\label{eq:SPDE.291}
	\int \mathrm dx \iint_0^t \phi_{t,x} G_{s,y;t,x}^{(\mathrm v)} \sigma_{s,y} W(\mathrm ds\mathrm dy)
	= \iint_0^t \sigma_{s,y} W(\mathrm ds\mathrm dy) \int \phi_{t,x} G_{s,y;t,x}^{(\mathrm v)} \mathrm dx.
\end{align}
	Note that from the fact that $f\in \mathscr P_e$, almost surely
\begin{align}
	\int \mathrm dx \iint_0^t |\phi_{t,x} G^{(\mathrm v)}_{s,y;t,x} f_{s,y}|\mathrm ds\mathrm dy
	\leq \|\phi\|_\infty \int_{-K}^K dx \iint_0^t G_{s,y;t,x}^{(\mathrm v)} |f_{s,y}|\mathrm ds\mathrm dy \overset{\eqref{eq:SPDE.16},\text{Lemma \ref{thm:SC.4}}}< \infty.
\end{align}
	This allows us to use Fubini theorem and conclude that almost surely
\begin{align}\label{eq:SPDE.292}
	\int \mathrm dx \iint_0^t \phi_{t,x} G^{(\mathrm v)}_{s,y;t,x} f_{s,y}\mathrm ds\mathrm dy
	= \iint_0^t f_{s,y}\mathrm ds\mathrm dy \int \phi_{t,x} G^{(\mathrm v)}_{s,y;t,x} \mathrm dx.
\end{align}
	Note that from the fact that $v\in \mathscr P_e$, almost surely
\begin{align}
	\int \mathrm dx \iint_0^t |\phi_{t,x} G^{(\mathrm v)}_{s,y;t,x} v_{0,y}|\delta_0(\mathrm ds)\mathrm dy
	\leq \|\phi\|_\infty \int_{-K}^K dx \int G_{0,y;t,x}^{(\mathrm v)} |v_{0,y}|\mathrm dy 
	\overset{\eqref{eq:SPDE.16},\text{Lemma \ref{thm:SC.4}}}< \infty.
\end{align}
	This allows us to use Fubini theorem and conclude that almost surely
\begin{align}\label{eq:SPDE.293}
	\int \mathrm dx \iint_0^t \phi_{t,x} G^{(\mathrm v)}_{s,y;t,x} v_{0,y}\delta_0(\mathrm ds)\mathrm dy
	= \iint_0^t v_{0,y}\delta_0(\mathrm ds)\mathrm dy\int \phi_{t,x} G^{(\mathrm v)}_{s,y;t,x} \mathrm dx.
\end{align}
	Now we have almost surely
\begin{align}
	&\int v_{t,x} \phi_{t,x} \mathrm dx 
	\overset{\eqref{eq:SPDE.21}}= \int \mathrm dx \iint_0^t  \phi_{t,x} G^{(\mathrm v)}_{s,y;t,x} M^v(\mathrm ds\mathrm dy)
	\\&\overset{\eqref{eq:SPDE.291},\eqref{eq:SPDE.292},\eqref{eq:SPDE.293}}=\iint_0^t M^v(\mathrm ds\mathrm dy) \int G^{(\mathrm v)}_{s,y;t,x} \phi_{t,x} \mathrm dx. 
\end{align}

\emph{Step 2.}
	We show that for any $t\geq 0$ almost surely
\begin{equation} 
	\iint_0^t  (\partial_r\phi_{t,x}+\partial_x^2\phi_{r,x}) v_{r,x} \mathrm dr\mathrm dx
	=\iint_0^t M^v(\mathrm ds\mathrm dy) \iint_s^t G^{(\mathrm v)}_{s,y;r,x}(\partial_r\phi_{r,x} + \partial_x^2\phi_{r,x}) \mathrm dr\mathrm dx.
\end{equation}
	Define 
$
	\psi_{r,x} := \partial_r \phi_{r,x} + \partial_x^2 \phi_{r,x}
$
	for each $(r,x)\in \mathbb R_+ \times \mathbb R$.
	Since $\psi \in C^\infty_c(\mathbb R_+ \times \mathbb R)$, there exists a $K>0$ and $T>0$ such that $\psi$ is supported on $[0,T]\times [-K,K]$.
	We can verify almost surely that 
\begin{align}
	&\iint_0^t \mathrm dr\mathrm dx\iint_0^t (\mathbf 1_{s\leq r}\psi_{r,x}G^{(\mathrm v)}_{s,y;r,x}\sigma_{s,y})^2\mathrm ds\mathrm dy
	\\&\leq \|\psi\|_\infty^2 \int_{-K}^K \mathrm dx\int_0^t \mathrm dr\iint_0^r (G^{(\mathrm v)}_{s,y;r,x}\sigma_{s,y})^2\mathrm ds\mathrm dy
	\overset{\eqref{eq:SPDE.16},\text{Lemma \ref{thm:SC.4}}}< \infty.
\end{align}
	This allows us to use the stochastic Fubini theorem and obtain that almost surely
\begin{equation}\label{eq:SPDE.311}
	\iint_0^t \mathrm dr\mathrm dx\iint_0^r \psi_{r,x}G^{(\mathrm v)}_{s,y;r,x}\sigma_{s,y}\dot W_{s,y}\mathrm ds\mathrm dy
	= \iint_0^t \sigma_{s,y}\dot W_{s,y}\mathrm ds\mathrm dy\iint_s^t \psi_{r,x}G^{(\mathrm v)}_{s,y;r,x} \mathrm dr\mathrm dx
\end{equation}
	We can verify that almost surely
\begin{align}
	&\iint_0^t \mathrm dr\mathrm dx\iint_0^t \mathbf 1_{s\leq r} |\psi_{r,x}G^{(\mathrm v)}_{s,y;r,x}f_{s,y}| \mathrm ds\mathrm dy
	\\&\leq \|\psi\|_\infty^2 \int_{-K}^K \mathrm dx\int_0^t \mathrm dr\iint_0^r G^{(\mathrm v)}_{s,y;r,x} |f_{s,y}|\mathrm ds\mathrm dy
	\overset{\eqref{eq:SPDE.16},\text{Lemma \ref{thm:SC.4}}} < \infty.
\end{align}
	This allows us to use Fubini theorem and obtain that almost surely
\begin{align}\label{eq:SPDE.312}
	& \iint_0^t \mathrm dr\mathrm dx\iint_0^r \psi_{r,x}G^{(\mathrm v)}_{s,y;r,x}f_{s,y}\mathrm ds\mathrm dy
	= \iint_0^t f_{s,y}\mathrm ds\mathrm dy\iint_s^t \psi_{r,x}G^{(\mathrm v)}_{s,y;r,x} \mathrm dr\mathrm dx.
\end{align}
	We can also verify that almost surely
\begin{align}
	&\iint_0^t \mathrm dr\mathrm dx\iint_0^t \mathbf 1_{s\leq r} |\psi_{r,x}G^{(\mathrm v)}_{s,y;r,x}v_{0,y}| \delta_0(\mathrm ds)\mathrm dy
	\\&\leq \|\psi\|_\infty^2 \int_{-K}^K \mathrm dx\int_0^t \mathrm dr\int G^{(\mathrm v)}_{0,y;r,x} |v_{0,y}|\mathrm dy
	\overset{\eqref{eq:SPDE.16},\text{Lemma \ref{thm:SC.4}}} < \infty.
\end{align}
	This allows us to use Fubini theorem and obtain that almost surely
\begin{equation}\label{eq:SPDE.313}
	\iint_0^t \mathrm dr\mathrm dx\iint_0^r \psi_{r,x}G^{(\mathrm v)}_{s,y;r,x}v_{0,y} \delta_0(\mathrm ds) \mathrm dy 
	= \iint_0^t v_{0,y} \delta_0(\mathrm ds) \mathrm dy \iint_s^t \psi_{r,x}G^{(\mathrm v)}_{s,y;r,x} \mathrm dr\mathrm dx.
\end{equation}
	Now we have almost surely
\begin{align}
	&\iint_0^t  \psi_{r,x} v_{r,x} \mathrm dr\mathrm dx
	\overset{\eqref{eq:SPDE.21}}= \iint_0^t  \mathrm dr\mathrm dx \iint_0^r \psi_{r,x} G^{(\mathrm v)}_{s,y;r,x} M^v(\mathrm ds\mathrm dy) 
	\\&\overset{\eqref{eq:SPDE.311},\eqref{eq:SPDE.312},\eqref{eq:SPDE.313}}= \iint_0^t  M^v(\mathrm ds\mathrm dy)  \iint_s^t G^{(\mathrm v)}_{s,y;r,x} \psi_{r,x} \mathrm dr\mathrm dx.
\end{align}

\emph{Step 3.}
	We show that for each $s\geq 0$ and $y\in \mathbb R$,
	\begin{equation} \label{eq:SPDE.3}
		\int G^{(\mathrm v)}_{s,y;t,x}\phi_{t,x}\mathrm dx+\Pi_{s,y}[\phi_{\tau, B_{\tau}}; t\geq \tau] 
		= \phi_{s,y}+ \iint_s^t G^{(\mathrm v)}_{s,y;r,x} (\partial_r \phi_{r,x}+\partial_x^2\phi_{r,x}) \mathrm dr\mathrm dx.
	\end{equation} 
	In fact, according to Ito's formula \cite[Theorem 3.3 and Remark 1 on p.~147]{RevuzYor1999Continuous}, we know that under probability $\Pi_{s,y}$,
	\begin{align} 
		& \phi_{t,B_t} - \phi_{s,y} - \int_s^t (\partial_r \phi_{r,x} + \partial_x^2 \phi_{r,x})|_{x=B_r} \mathrm dr, \quad t\geq s 
	\end{align}
	is a zero-mean martingale.
	Then, according to optional sampling theorem \cite[Theorem 7.29]{Kallenberg2002Foundations} we have 
	\begin{align} 
		&  \Pi_{s,y}[\phi_{t\wedge \tau, B_{t\wedge \tau}}] 
		= \phi_{s,y} + \int_s^t \Pi_{s,y}[(\partial_r\phi_{r,x}+\partial_x^2\phi_{r,x})|_{x=B_r}; r< \tau]\mathrm dr.
	\end{align}
	The desired result in this step then follows from \eqref{eq:SPDE.155} and above.
	
\emph{Step 4.}
	We note from the fact $\phi\in \cC_c(\mathbb R_+\times \mathbb R)$ and that $v,f,\sigma$ are locally finite, the following stochastic integral
\[
	\iint_0^t \phi_{s,y} M^v(\mathrm ds dy)
	= \iint_0^t \phi_{s,y} v_{0,y} \delta_0(\mathrm ds) \mathrm dy  + \iint_0^t \phi_{s,y}  f_{s,y}\mathrm ds\mathrm dy  + \iint_0^t \phi_{s,y} \sigma_{s,y}W(\mathrm ds\mathrm dy)
\]
	is well defined for each $t\geq 0$.

\emph{Final Step.}
	Finally we can verify that for each $t\geq 0$ almost surely, 
\begin{align} 
	&  \int v_{t,x} \phi_{t,x} \mathrm dx 
	\overset{\text{Step 1}}=   \iint_0^t M^v(\mathrm ds\mathrm dy) \int G^{(\mathrm v)}_{s,y;t,x} \phi_{t,x} \mathrm dx 
	\\ &\overset{\text{Step 3}}=   \iint_0^t M^v(\mathrm ds\mathrm dy) \Big(\phi_{s,y} + \iint_s^t G_{s,y;r,x}^{(\mathrm v)} (\partial_r\phi_{t,x}+\partial_x^2\phi_{r,x}) \mathrm dr\mathrm dx - \Pi_{s,y}[\phi_{\tau, B_\tau}; t\geq \tau]\Big) 
	\\ &\begin{multlined}
	\overset{\text{Steps 2 and 4}}= \iint_0^t \phi_{s,y}M^v(\mathrm ds\mathrm dy) + \iint_0^t  (\partial_r\phi_{t,x}+\partial_x^2\phi_{r,x}) v_{r,x} \mathrm dr\mathrm dx 
	\\ - \iint_0^t \Pi_{s,y}[\phi_{\tau, B_\tau}; t\geq \tau] M^v(\mathrm ds\mathrm dy)
	\end{multlined}
\end{align}
	as desired.
\end{proof}
\begin{remark}
	Someone might what to show that there exists an adapted real-valued process $(\dot A_r)_{r\geq 0}$, such that for any $t\geq 0$ almost surely,
\begin{align} \label{eq:SPDE.32}
	& \iint_0^t \Pi_{s,y}[\phi_{\tau, B_\tau}; t>\tau] M^v(\mathrm ds\mathrm dy) = \int_0^t \phi_{r,\mathrm vr} \dot A_r \mathrm dr.
\end{align}
	Indeed, for each $s\geq 0$ and $y \in (-\infty, \mathrm vs)$, let us denote by $(p_{s,y;r})_{r\geq s}$ the density function of $\tau$ under probability $\Pi_{s,y}$, then it \emph{seems that} almost surely,
\begin{align}
	&\iint_0^t \Pi_{s,y}[\phi_{\tau, B_\tau}; t>\tau] M^v(\mathrm ds\mathrm dy) 
	= \iint_0^t M^v(\mathrm ds\mathrm dy) \int_s^t \phi_{r,\mathrm vr} p_{s,y;r} \mathrm dr
	\\& \label{eq:SPDE.35}\overset{\text{Stochastic Fubini}}= \int_0^t\phi_{r,\mathrm vr} \Big(\iint_0^r p_{s,y;r} M^v(\mathrm ds\mathrm dy)\Big)  \mathrm dr.
\end{align}
	So one might want to define 
\[
	\dot A_r := \iint_0^r p_{s,y;r} M^v(\mathrm ds\mathrm dy), \quad t\geq 0.
\]
	However, we gonna show that this definition of $\dot A$ might be a problem.
	Actually, using reflection principle, one can calculate the precise expression for the density of $\tau$:
\[
	p_{s,y;r} = \frac{(\mathrm vs - y) \exp\{\frac{-(\mathrm vr-y)^2}{4(r-s)}\}}{\sqrt{4\pi(r-s)^3}}\mathbf 1_{0\leq s<r, y<\mathrm vs}. 
\] 
	Then one can find $\sigma\in \mathscr P_\mathrm e$ such that 
\[
	\iint_0^r p_{s,y;r}^2 \sigma^2_{s,y} \mathrm ds\mathrm dy = \infty.
\]
	In this cases, the stochastic integral
\[
	\iint_0^r p_{s,y;r} \sigma_{s,y} \dot W^v_{s,y}\mathrm ds\mathrm dy
\]
	is not well defined.
	One can even find an ideal $\sigma^2$ which is bounded and compactly supported and linearly decreasing to $0$ at the boundary line $\{(s,y): y = \mathrm vs\}$ so that integral above is not well defined.
	For example, we can take
\[
	\sigma^2_{s,y} := (\mathrm vs - y) \mathbf 1_{\mathrm vs - y\in [0,K]}
\]
	for some constant $K>0$.
	Then 
\begin{align}
	& 4\pi \iint_0^r p_{s,y;r}^2 \sigma_{s,y}^2 \mathrm ds\mathrm dy 
	=  \int_{0}^{r} \mathrm ds \int  \frac{(\mathrm vs - y)^3 \exp\{\frac{-(\mathrm v(r-s)+(\mathrm vs-y))^2}{2(r-s)}\}}{(r-s)^3} \mathbf 1_{\mathrm vs-y\in [0,K]}  \mathrm dy
	\\& = \int_{0}^{r} \mathrm ds \int  \frac{x^3 \exp\{\frac{-(\mathrm v(r-s)+x)^2}{2(r-s)}\}}{(r-s)^3} \mathbf 1_{x\in [0,K]}  \mathrm dx
	=   \int_0^K \mathrm dx \int_{0}^{r} \frac{x^3 \exp\{\frac{-(\mathrm v(r-s)+x)^2}{2(r-s)}\}}{(r-s)^3} \mathrm ds 
	\\& =  \int_0^K \mathrm dx \int_{0}^{r} \frac{x^3 \exp\{\frac{-(\mathrm vu+x)^2}{2 u}\}}{u^3} \mathrm du
	= \int_0^K \mathrm dx\int_0^r  \frac{x^3 e^{-\frac{\mathrm v^2}{2} u} e^{-\mathrm vx} e^{-\frac{x^2}{2u}}}{u^3}  \mathrm du
	\\& \geq e^{-\frac{\mathrm v^2 r}{2} - \mathrm vK}\int_0^K x^3 \mathrm dx\int_0^t  \frac{e^{-\frac{x^2}{2u}}}{u^3}  \mathrm du
	= e^{-\frac{\mathrm v^2 r}{2} - \mathrm vK}\int_0^K x^3 \cdot \frac{2(x^2 + 2r) e^{-\frac{x^2}{2r}}}{x^4 r}\mathrm dx
	\\&\geq 2e^{-\frac{\mathrm v^2 r}{2} - \mathrm vK-\frac{K^2}{2r}}\Big(\int_0^K \frac{x}{r}\mathrm dx+\int_0^K \frac{2}{x }\mathrm dx\Big)
	= \infty.
\end{align}
\end{remark}

	For any real valued function $h$ on $\mathbb R$, define $L_1(h) := \inf \{x\in \mathbb R: h_x \neq 1\}$, $\tilde L(h):=\inf\{x\in \mathbb R: h_x \neq 0\}$ and $R_0(h) := \sup \{x\in \mathbb R: h_x \neq 0\}$.
	We say a real-valued function $h$ on $\mathbb R$ has \emph{compact interface} (resp. \emph{compact support}) if $|L_1(h) - R_0(h)| < \infty$ (resp. $|\tilde L(h) - R_0(h)|<\infty$).
	We say a real-valued function $h$ on $\mathbb R_+ \times \mathbb R$ has \emph{locally compact interface} (resp. \emph{locally compact support}) if  for all $T>0$, $\sup_{0\leq t\leq T}|L_1(h_{t,\cdot}) - R_0(h_{t,\cdot})| < \infty$ (resp. $\sup_{0\leq t\leq T}|\tilde L(h_{t,\cdot}) - R_0(h_{t,\cdot})| < \infty$). We say a real-valued random field $h$ has \emph{locally compact interface} (resp. \emph{locally compact support}), if almost surely, $h$ has locally compact interface (resp. \emph{locally compact support}).

\begin{lemma}
	Suppose that $v$ is a non-negative solution to SPDE \eqref{eq:SPDE.2} which has locally compact interface. 
	Also suppose that $f$ and $\sigma$ both have locally compact support.
	Then the corresponding killing process $A$ given by Lemma \ref{thm:SC.5} is a non-decreasing process.
\end{lemma}

\begin{proof}
	Since $(A_t)_{t\geq 0}$ is a continuous process, we only have to proof that for each $0\leq s < t < \infty$, almost surely $A_s\leq A_t$.
	We first claim that for each $t\geq 0$ almost surely, 
\begin{align} \label{eq:SPDE.4}
& A_t =  - \int (v_{t,x} - v_{0,x})\mathrm dx + \iint_0^t \big(f_{s,y} \mathrm ds\mathrm dy + \sigma_{s,y} W(\mathrm ds\mathrm dy)\big).
\end{align}
	Let $\varphi$ be a real function on $\mathbb R$ with compact interface which satisfies that
\begin{equation}
\varphi_x =
\begin{cases} 
{\rm smooth}, & x \in [0,\infty),
\\ -x(x+2), & x \in [-1,0],
\\ 1, & x \in (-\infty,-1].
\end{cases}
\end{equation}
We also claim that for each $t\geq 0$,
\begin{equation}\label{eq:SPDE.41}
\begin{multlined}
	-  \iint_0^t (\partial_s \varphi^{(m)}_{s,x} + \partial_x^2 \varphi^{(m)}_{s,x}) v_{s,x}\mathrm ds\mathrm dx
	\xrightarrow[m\to \infty]{\text{in probability}}
	 \\ - \int (v_{t,x} - v_{0,x})\mathrm dx + \iint_0^t \big(f_{s,y} \mathrm ds\mathrm dy + \sigma_{s,y} W(\mathrm ds\mathrm dy)\big)
\end{multlined}
\end{equation}
	where for each $s\geq 0, x\in \mathbb R$, and $m \in \mathbb N$, $\varphi^{(m)}_{s,x} := \varphi_{(x-\mathrm vs)m}$.  
	It is obvious that for each $t\geq 0$, almost surely, 
\begin{align}
	& {\rm LHS~of~\eqref{eq:SPDE.41}}
	= -\int_0^t \mathrm ds \int_{\mathrm vs - \frac{1}{m}}^{\mathrm vs} (\partial_s \phi^{(m)}_{s,x} + \partial_x^2 \phi^{(m)}_{s,x}) v_{s,x} \mathrm dx 
\\&=\int_0^t \mathrm ds \int_{\mathrm vs - \frac{1}{m}}^{\mathrm vs} \Big(2m^2 - 2\mathrm v m\big(1+ (x-\mathrm vs)m\big) \Big) v_{s,x} \mathrm dx
\\&=: J_t^{(m)}.
\end{align}

Now let us fix arbitrary $0\leq s<  t< \infty$.
Then from \eqref{eq:SPDE.4}, \eqref{eq:SPDE.41} and \cite[Lemma 4.2]{Kallenberg2002Foundations}, we know that there exists a (deterministic) sequence $(m_k)_{k\in \mathbb N} \subset \mathbb N$ such that $m_k \uparrow \infty$ when $k\uparrow \infty$ and that almost surely
\begin{align} 
& J_s^{(m_k)} \xrightarrow[k\to \infty]{} A_s, \quad J_t^{(m_k)} \xrightarrow[k\to \infty]{} A_t.
\end{align}
This implies that almost surely
\begin{align} 
& A_t - A_s  = \lim_{k\to \infty} \int_s^t \mathrm ds \int_{\mathrm vs - \frac{1}{m_k}}^{\mathrm vs} \Big(2m_k^2 - 2\mathrm v m_k\big(1+ (x-\mathrm vs)m_k\big) \Big) v_{s,x} \mathrm dx\geq 0
\end{align}
as desired.

\begin{proof*}[Proof of \eqref{eq:SPDE.4}]
	Since almost surely the both side of \eqref{eq:SPDE.4} are continuous in $t\geq 0$, we only have to proof that for each $t>0$ almost surely \eqref{eq:SPDE.4} is valid.
	Let us fix an arbitrary $t\geq 0$.
	Define $h\in C^\infty_{\rm c}(\mathbb R)$ by
\[
	h_{x}
	=  \begin{cases} 0, & x\in (-\infty, -1], \\ \text{smooth}, & x\in [-1,0],\\ 1, & x\in [0,\infty). \end{cases}	
\]
	For each $s\geq 0, x\in \mathbb R$ and $n\in \mathbb N$, define
\begin{equation} \label{eq:SPDE.5}
\phi^{(n)}_{s,x} 
= \begin{cases}
	1 - h_{x-\mathrm vs - 1}, & s\geq 0, x \in [\mathrm vs, \infty)\\
	1, & s\geq 0, x \in [-n, \mathrm vs] \\ 
	h_{x+n}, &s\geq 0, x \in (- \infty, -n]. \\
\end{cases}
\end{equation}
	Then we have almost surely, for each $n\in \mathbb N$,
	\begin{align} 
	& \begin{multlined}	
	A_t
	\overset{\eqref{eq:SPDE.25}}= -\int (\phi^{(n)}_{t,x} v_{t,x} - \phi^{(n)}_{0,x}v_{0,x}) \mathrm dx +\iint_0^t (\partial_s \phi^{(n)}_{s,x} +\Delta_x  \phi^{(n)}_{s,x})v_{s,x}\mathrm ds\mathrm dx 
	\\+ \iint_0^t \phi^{(n)}_{s,x}( f_{s,x}\mathrm ds\mathrm dx + \sigma_{s,x} W(\mathrm ds\mathrm dx))
	\end{multlined}
	\\ \label{eq:SPDE.54}&=: - I_1 + I_2 + I_3.
	\end{align}
	Notice that $v$ has locally finite interface, also notice that $f,\sigma$ have locally finite support. 
	Therefore there exists a random variable $N_t>0$ such that
\begin{align}  \label{eq:SPDE.55}
& v_{s,x} = 1,\quad  f_{s,x} = 0,\quad \sigma_{s,x} = 0, \quad 0\leq s\leq t, x\leq - N_t, \quad {\rm a.s.}
\end{align}
	Due to the fact that, almost surely, $v, f$ and $\sigma$ are all supported on $\{(s,x) \in [0,\infty)\times \mathbb R: x<\mathrm vs\}$, we have
\begin{align} \label{eq:SPDE.56}
&  v_{s,x} = 0,\quad  f_{s,x} = 0,\quad \sigma_{s,x} = 0, \quad s\geq 0, x\geq \mathrm vs, \quad {\rm a.s.}
\end{align}
	Now we can verify that almost surely for each $n > N_t$,
\begin{align} 
& I_1 
=  \int (\phi^{(n)}_{t,x} v_{t,x} - \phi^{(n)}_{0,x}v_{0,x}) \mathrm dx
= \int \Big( \phi^{(n)}_{t,x} ( v_{t,x} - v_{0,x})  + (\phi^{(n)}_{t,x} - \phi^{(n)}_{0,x}) v_{0,x} \Big)\mathrm dx
 \\ & =\Big(\int_{-\infty}^{-n} + \int_{-n}^{\mathrm vt} + \int_{\mathrm vt}^\infty \Big)  \phi^{(n)}_{t,x} ( v_{t,x} - v_{0,x}) \mathrm dx +\Big(\int_{-\infty}^0 + \int_0^\infty \Big)   (\phi^{(n)}_{t,x} - \phi^{(n)}_{0,x}) v_{0,x} \mathrm dx
\\\label{eq:SPDE.57}&\overset{\eqref{eq:SPDE.5}, \eqref{eq:SPDE.5},\eqref{eq:SPDE.55}}= \int (v_{t,x} - v_{0,x})\mathrm dx,
\end{align}
and 
\begin{align} 
& I_2 =  \int_0^t \mathrm ds \Big(\int_{\mathrm vs}^\infty + \int_{- n}^{\mathrm vs} + \int_{-\infty}^{-n}\Big) (\partial_s  \phi^{(n)}_{s,x} +\Delta_x  \phi^{(n)}_{s,x})v_{s,x} \mathrm dx  
\\ \label{eq:SPDE.58}& \overset{\eqref{eq:SPDE.5}, \eqref{eq:SPDE.55}, \eqref{eq:SPDE.56}}=  \int_0^t \mathrm ds \int_{-\infty}^{-n} (\partial_s \phi^{(n)}_{s,x} +\Delta_x \phi^{(n)}_{s,x}) \mathrm dx  
= 0,
\end{align}
and also
\begin{align} 
& I_3 
= \iint_0^t (\mathbf 1_{x\in (-\infty, -n]} + \mathbf 1_{x\in (-n,\mathrm vs]} + \mathbf 1_{x\in (\mathrm vs,\infty)}) \phi_{s,x}^{(n)}\Big(f_{s,x}\mathrm ds\mathrm dx + \sigma_{s,x} W(\mathrm ds\mathrm dx)\Big)
\\\label{eq:SPDE.59}& \overset{\eqref{eq:SPDE.5}, \eqref{eq:SPDE.55}, \eqref{eq:SPDE.56}}= \iint_0^t  \Big(f_{s,x}\mathrm ds\mathrm dx + \sigma_{s,x} W(\mathrm ds\mathrm dx)\Big).
\end{align}
The desired result then follows from \eqref{eq:SPDE.54}, \eqref{eq:SPDE.57}, \eqref{eq:SPDE.58} and \eqref{eq:SPDE.59}.
\end{proof*}
\begin{proof*}[Proof of \eqref{eq:SPDE.41}]
	Fix the time $t>0$. Define
\begin{align} 
& \varphi_{s,x}^{(m,n)} = 
\begin{cases}
\varphi_{m(x-\mathrm vs)}, &s\geq 0, x\in [-n , \infty)\\
h_{x+n}, &s\geq 0, x\in (-\infty, -n].
\end{cases}  
\end{align}
	Using an argument similar to the proof of \eqref{eq:SPDE.4}, we know that almost surely, for each $m\geq 1$ and $n \geq N_t$,
\begin{align}
&\begin{multlined}	
	0 \overset{\eqref{eq:SPDE.25}}= - \int (\varphi_{t,x}^{(m,n)} v_{t,x} - \varphi_{0,x}^{(m,n)}v_{0,x}) \mathrm dx  + \iint_0^t (\partial_s \varphi_{s,x}^{(m,n)}  + \Delta_x \varphi_{s,x}^{(m,n)})v_{s,x}\mathrm ds\mathrm dx \\ 
	+ \iint_0^t \varphi_{s,x}^{(m,n)} \Big(f_{s,x}\mathrm ds\mathrm dx+\sigma_{s,x}W(\mathrm ds\mathrm dx)\Big)
\end{multlined}
\\& \label{eq:SPDE.65}\begin{multlined}
= - \int (\varphi_{t,x}^{(m)} v_{t,x} - \varphi_{0,x}^{(m)}v_{0,x}) \mathrm dx  + \iint_0^t (\partial_s \varphi_{s,x}^{(m)}  + \Delta_x \varphi_{s,x}^{(m)})v_{s,x}\mathrm ds\mathrm dx 
\\ + \iint_0^t \varphi_{s,x}^{(m)} \Big(f_{s,x}\mathrm ds\mathrm dx+\sigma_{s,x}W(\mathrm ds\mathrm dx)\Big).
\end{multlined}
\end{align}
	Note that, using dominated convergence theorem, we have almost surely
\begin{align}\label{eq:SPDE.66}
	\int (\varphi_{t,x}^{(m)} v_{t,x} - \varphi_{0,x}^{(m)}v_{0,x}) \mathrm dx
	\xrightarrow[m\to \infty]{} \int ( v_{t,x} - v_{0,x}) \mathrm dx,
\end{align} 
\begin{align} \label{eq:SPDE.67}
&  \iint_0^t \varphi_{s,x}^{(m)} f_{s,x}\mathrm ds\mathrm dx 
\xrightarrow[m\to \infty]{} \iint_0^t f_{s,x}\mathrm ds\mathrm dx,
\end{align}
and 
\begin{align} \label{eq:SPDE.68}
&  \iint_0^t (\varphi_{s,x}^{(m)} \sigma_{s,x} - \sigma_{s,x})^2\mathrm ds\mathrm dx 
\xrightarrow[m\to \infty]{} 0.
\end{align}
According to \eqref{eq:SPDE.68} and \cite[Proposition 17.6]{Kallenberg2002Foundations}, we have that 
 \begin{align} \label{eq:SPDE.69}
 &  \iint_0^t \varphi_{s,x}^{(m)} \sigma_{s,x} W(\mathrm ds\mathrm dx)
 \xrightarrow[m\to \infty]{\text{in probability}}   \iint_0^t  \sigma_{s,x} W(\mathrm ds\mathrm dx).
 \end{align}
 The desired result then follows from \eqref{eq:SPDE.65}, \eqref{eq:SPDE.66}, \eqref{eq:SPDE.67} and \eqref{eq:SPDE.69}.

\end{proof*}
\end{proof}

\end{note}
\end{document}